\documentclass[11pt]{amsart}
\newcommand{\beq}{\begin{equation}}
\newcommand{\eeq}{\end{equation}}
\newcommand{\ben}{\begin{eqnarray}}
\newcommand{\een}{\end{eqnarray}}
\newcommand{\beno}{\begin{eqnarray*}}
\newcommand{\eeno}{\end{eqnarray*}}
\newcommand{\R}{\mathbb{R}}
\newtheorem{thm}{Theorem}[section]

\newtheorem{lem}[thm]{Lemma}
\newtheorem{prop}[thm]{Proposition}

\newtheorem{rmk}[thm]{Remark}

\begin{document}

\title[Radial terrace solutions and propagation profile]{ Radial terrace solutions and propagation profile of multistable reaction-diffusion equations over $\R^N$}
\thanks{}
\author[Y. Du and H. Matano]{Yihong Du$^\dag$\; and  Hiroshi Matano$^\ddag$}
\thanks{$^\dag$School of Science and Technology,
University of New England, Armidale, NSW 2351, Australia
({\bf Email:} ydu@\allowbreak une.\allowbreak edu.\allowbreak
au).}
\thanks{$^\ddag$Meiji Institute for Advanced Study of Mathematical Sciences,
Meiji University, 4-21-1 Nakano, Tokyo 164-8525, Japan
  ({\bf Email:} matano@meiji.ac.jp).}

\thanks{$^\$$  The research of Y. Du was supported
by the Australian Research Council, and the research of H. Matano was supported by KAKENHI (16H02151)}
\date{\today}

\begin{abstract}

We study the propagation profile of the solution $u(x,t)$ to the nonlinear diffusion problem
\[
\left\{ \begin{aligned}
 &u_t-\Delta u=f(u) & &\mbox{for}~~x\in \R^N,\;t>0,\\
 &u(x,0)=u_0(x) && \mbox{for}~~x\in\R^N,
                          \end{aligned} \right.
\]
\noindent
where   $f(u)$ is a  multistable nonlinearity. More precisely, there exists $p>0$ such that $f(0)=f(p)=0$, $f'(0)<0$, $f'(p)<0$, any zero of $f$ in $(0,p)$ which is asymptotically stable from below is linearly stable, and $\int_u^pf(s)ds>0$ for $u\in [0, p)$. This allows $f$ to have many (even a continiuum of) zeros in $(0, p)$.
 The class of initial functions $u_0$ includes in particular those which are nonnegative and decay to 0 at infinity, but compact support is not needed, nor radial symmetry. We show that, if $u(\cdot, t)$ converges to $p$ as $t\to\infty$ in $L^\infty_{loc}(\R^N)$, then  the propagation profile  of $u(x,t)$ is well approximated  by  the  one-dimensional 
  propagating terrace of Ducrot, Giletti and Matano \cite{DGM}, when the one dimensional variable is taken by $|x|$, with time shifts of the form $c_kt+o(t)$, where $\{c_k\}$ is the finite sequence of  wave speeds in the propagating terrace. Moreover, for generic $f$, we show that the $o(t)$ term has the form $\frac{N-1}{c_k}\log t+O(1)$.
To obtain such results, we first construct a special  radially symmetric solution, called a radial terrace solution, whose large time  behavior in the radial direction exhibits the one-dimensional terrace behavior found in \cite{DGM}. We then use this special solution to construct a supersolution and a subsolution to study the behavior of the general solutions that are not radially symmetric in general.

\bigskip


\smallskip


\end{abstract}

\numberwithin{equation}{section}
\maketitle

\tableofcontents

\renewcommand{\theequation}{\thesection.\arabic{equation}}
\setcounter{equation}{0}

\section{Introduction}

We study the long-time behaviour of  the solutions to 
\begin{equation}
\label{nd}
\left\{ \begin{aligned}
 &u_t-\Delta u=f(u) & &\mbox{for}~~x\in \R^N,t>0,\\
 &u(x,0)=u_0(x) && \mbox{for}~~x\in\R^N,
                          \end{aligned} \right.
\end{equation}
where $u_0\in L^\infty(\R^N)$, and $f$ is a $C^1$ ``multistable" function.

In the classical case that $f$ is bistable, namely, there exist $0<b<p$ such that
\[\begin{cases}
f(0)=f(b)=f(p)=0, \ f'(0),\ f'(p)<0,\;\\
 f(u)<0<f(v) \mbox{ for } u\in (0,b)\cup(p,\infty), v\in (-\infty, 0)\cup (b, p),\\
  \int_0^pf(u)du>0,
  \end{cases}
\]
 it is well known that there exists a unique $c>0$ such that the problem
\[
U''+cU'+f(U)=0,\; U(-\infty)=p,\; U(+\infty)=0
\]
 has a solution $U(\xi)$, which is unique if we further require, for example, $U(0)=p/2$; see 
\cite{A-W, FM}.  Such a solution $U$ is called a traveling wave solution connecting $p$ to $0$ with speed $c$, which governs  the propagating behavior of the solution to \eqref{nd}.

For example, if $u_0$ is radially symmetric, nonnegative and compactly supported, and the unique solution of \eqref{nd} satisfies $\lim_{t\to\infty}u(x,t)=p$ in $L^\infty_{\rm loc} (\R^N)$, then by \cite{U} there exists $R_0\in\R$ such that
\[
u(x,t)-U(|x|-ct+\frac{N-1}{c}\log t+R_0)\to 0 \mbox{ as } t\to\infty \]
 uniformly for $ x\in\R^N$.

If the radial symmetry requirement for $u_0$ is dropped, then  by applying the  recent result in \cite{MMN} on anisotropic equations  to the  isotropic equation \eqref{nd},  there exist $T\gg 1$ and
a $C^1$ function $l:  \mathbb S^{N-1}\times [T,\infty)\mapsto \R$ such that
\begin{equation}\label{u-U-bs}
u(x,t)-U\Big(|x|-ct+\frac{N-1}{c}\log t+l\big(\frac{x}{|x|}, t\big)\Big)\to 0 \mbox{ as } t\to\infty
\end{equation}
 uniformly for $ x\in\R^N\setminus\{0\}$. It can be further shown that $l\in L^\infty$.

In this paper, we investigate extensions of these results to the case that $f$ is multistable, namely, in the interval $(0, p)$, $f(u)$ may have multiple zeros, and 
\begin{equation}
\label{gamma}
\int_u^p f(s)ds>0 \;\;\forall u\in [0, p),
\end{equation}
which, as we will see below,  is the natural extension of the condition $\int_0^p f(s)ds>0$  in the bistable case. 

Roughly speaking, we will show that
similar results still hold, albeit  the traveling wave pair $(U, c)$ in \eqref{u-U-bs} should be replaced by a finite sequence of  traveling wave pairs
$\{(U_k, c_k)\}$, which constitute the one dimensional propagating terrace connecting $p$ to $0$. (See \eqref{u-limit-1} below.)

\bigskip

In the following, we describe our results more precisely.

\subsection{Assumptions} 

 Suppose $f(q)=0$. Then we say $q$ is {\it asymptotically stable from above} if $f(u)<0$ for $u$ in some small right neighborhood of $q$, say $u\in (q, q+\epsilon)$; we say $q$ is {\it asymptotically stable from below} if $f(u)>0$ for $u$ in some small left neighborhood of $q$; we say $q$ is nondegenerate if $f'(q)\not=0$. 
 Clearly $q$ would be asymptotically stable from both above and below if it is {\it linearly stable}, namely  $f'(q)<0$. 
These stability notions agree with that for $q$ when it is regarded as a stationary solution of the ODE problem $u'=f(u)$.

In this paper, by a {\it multistable} $f$, we mean a function $f(u)$ with the following properties:
\begin{itemize}
\item[{\bf (f1)}]  $f$ is $C^1$ and  $f(0)=0>f'(0)$,
\item[{\bf (f2)}] there exists $p>0$ such that
$f(p)=0>f'(p)$, and \eqref{gamma} holds,
\item[{\bf (f3)}] any zero of $f$  in $(0, p)$  which is asymptotically stable from below is linearly stable, and $f(u)>0$ for $u<0$, $f(u)<0$ for $u>p$.\footnote{The assumption $f(u)>0$ for $u<0$, $f(u)<0$ for $u>p$ is not essential; all our results remain true if this assumption is dropped, provided that the definition of $\mathcal T(f)$ in \eqref{T(f)} is modified accordingly.}
\end{itemize}
Let us note that under the above assumptions, it is possible for $f$ to have infinitely many (even a continium of) zeros in $(0,p)$.

Suppose {\bf (f1)-(f3)} hold and the solution of \eqref{nd} satisfies
\begin{equation}
\label{u-p}
u(x,t)\to p \mbox{ locally uniformly in $x$ as } t\to\infty.
\end{equation}
Naturally, it is expected that the zeros of $f$ in $(0, p)$ will impact on  how $u(\cdot, t)$ propagates to $p$  as $t\to\infty$.

Let us first look at some simple sufficient conditions on $u_0$ that guarantee \eqref{u-p}.  Let $b^*\in (0, p)$ be the first unstable zero of $f$ below $p$, namely
\[
f(b^*)=0, \; f(u)>0 \;\forall u\in (b^*,p).
\]
By Lemma 2.4 of \cite{Du-P}, for each $\theta\in (b^*,p)$, there exists $R(\theta)>0$ such that the unique solution of \eqref{nd} with initial function
\begin{equation}
\label{u*0}
u^*_0(x)=\left\{\begin{array}{ll} \theta & \mbox{ for } |x|\leq R(\theta),\\
0& \mbox{ for } |x|>R(\theta),
\end{array}
\right.
\end{equation}
satisfies $u(x,t)\to p$ as $t\to\infty$ locally uniformly in $x$. By the comparison principle, \eqref{u-p} holds for any solution of \eqref{nd} with
initial function $ u_0\in L^\infty(\R^N)$ satisfying
\[ u_0(x)\geq  u^*_0(x) \mbox{ in } \R^N.
\]

\smallskip

 It turns out that only certain zeros of $f$ in $(0, p)$ will be directly involved in describing the propagation behavior of $u(\cdot, t)$ as $t\to\infty$. We will show that,  for a large class of  solutions of \eqref{nd} satisfying \eqref{u-p}, the evolution of $u(\cdot,t)$ as $t\to\infty$  is determined by the
``propagating terrace" connecting $p$ to $0$ obtained from the one-dimensional equation
\beq \label{1D}
      u_t = u_{rr} + f(u) \quad\  ( r\in \R,\; t\in \R).
\eeq

\subsection{Propagating terrace}
For convenience of later reference and clarity, we now recall the notion of propagating terrace for \eqref{1D} and some of its basic properties.
Let $q^* > q_*$ be two linearly stable zeros of $f$.  By a {\bf propagating 
terrace} for \eqref{1D} connecting  $q^*$ to $q_*$,  we mean a 
sequence of  zeros of $f$: 
\[
   q^* =p_0> p_1 > \cdots > p_n = q_*, 
\]
coupled with a sequence of traveling wave solutions 
$U_1, U_2, ..., U_n$ of \eqref{1D} satisfying
\[
\begin{cases}
   U_i''+c_i U_i'+f(U_i)=0,\; U_i(-\infty) = p_{i-1}, \ \  U_i(+\infty) = p_{i} \quad\ \ 
(i=1,2,\ldots,n),\\
 c_1  \leq  c_2  \leq \cdots \leq  c_n,
\end{cases}
\]
where $c_i\;(i=1,\ldots,n)$ is called the speed of the traveling 
wave $U_i$. We call $p_0,p_1,\ldots, p_n$ the {\bf platforms} of 
the terrace. In general only a subset of the zeros of $f$ in $[q_*, q^*]$ appear on the list of platforms in the propagating terrace. It can be shown that every platform is asymptotically stable from below (see \cite{DGM}).
We will denote such a propagating terrace by 
$
\big\{p_i, U_i, c_i\big\}_{1\leq i\leq n}.
$

The notion of  propagating terrace was introduced in \cite{DGM} 
in a more general setting, where  $f=f(r,u)$  depends on $r$ periodically.\footnote{ An extension of this notion to the higher dimension space-periodic setting was given recently by Giletti and Rossi \cite{GR}.}   Further properties of propagating 
terraces were studied in \cite{GM}.  Note that, as far as spatially 
homogeneous equations of the form \eqref{1D} are concerned, a 
similar concept already appeared in \cite{FM} 
under the name ``minimal decomposition", and in \cite{VVV, V} under the term ``minimal system of waves''; see \cite{GM} and \cite{P2} for more details.

For the existence and uniqueness of propagating terrace, we have the following result, where the assumptions are stronger than necessary but this version is enough for our purpose in this paper. (A more general uniqueness result is established in \cite{GM}.)

{\bf Lemma A.} {\it Suppose that $f$ is a $C^1$ function.
Let $q^*  >  q_*$  be two linearly stable zeros of  $f$ satisfying
\beq\label{qq}
\int_{u}^{q^*}f(s)ds>0 \;\;\forall u\in [q_*, q^*).
\eeq
Then there exists 
a propagating terrace for \eqref{1D} connecting  $q^*$  to $q_*$.  Suppose additonally that any zero of $f$ in $(q_*, q^*)$ which is asymptotically stable from below is linearly stable; then
the propagating terrace is unique.
}

Here, by ``unique", we mean that the set of platforms 
$q^*=p_0 > p_1 > \cdots > p_n = q_*$ is unique, and that the 
traveling waves $U_1,\ldots,U_n$ are unique up to time shifts. By our nondegeneracy assumption for zeros of $f$ in $(q_*, q^*)$ which are asymptotically stable from below, it is easily seen that the platforms of the terrace in Lemma A  contain only linearly stable zeros of $f$.
Note that by Lemma 2.1 of \cite{FM}, each traveling wave $U_k$ satisfies $U_k'<0$, and it follows easily that its speed $c_k>0$ (see, for example, the proof of Lemma \ref{w2} below).

\begin{proof}[{\bf Proof of Lemma A}]
The existence of a propagating terrace is shown in \cite{DGM} in a much 
more general setting. Here we only need to check that  Assumption 1.1 there is satisfied, namely
there exists a solution $u$ of \eqref{nd} with compactly supported
initial function $0\leq u_0(x)<p$  that converges locally uniformly to $p$ as $t\to+\infty$.

But this follows easily from \eqref{qq} and Lemma 2.4 of \cite{Du-P}.
Regarding uniqueness, the part on the
set of platforms follows from  Theorem 2.8 of \cite{FM}. It remains to show that the 
traveling wave connecting each pair of adjacent platforms is unique subject to a time shift, but 
this follows by the standard Fife-McLeod type super-subsolution 
argument (\cite{FM,xChen})\footnote{A more general version of this argument will be given and used later in the current paper; see Lemmas \ref{fm-sup}, \ref{w=tw} and Remark \ref{3.12}.}, since each platform  is linearly stable.
\end{proof}

From now on, we will focus on multistable nonlinearities $f$ satisfying {\rm {\bf (f1)-(f3)}}.
For definiteness, we denote the linearly stable zeros of $f$ in $[0,p]$ by
\[
p=q_0>q_1>q_2>...>q_m=0.
\]
Clearly corresponding to each $q_i$ ($i=1,..., m$), there exists a unique $b_i$ satisfying
\begin{equation}\label{b_i}
\begin{cases}
f(b_i)=0,\; f(u)>0 \mbox{ for } u\in (b_i, q_{i-1}), \;  f(u)\leq 0 \mbox{ for } u\in (q_i, b_i),\\
 \int_{q_i}^u f(s)ds<0 \mbox{ for } u\in (q_i, b_i].
\end{cases}
\end{equation}
Let us note that $f$ may have many (even a continuium of) zeros in $(q_i, b_i)$, but by {\bf (f3)}, necessarily
$f(u)\leq 0$ in this interval, and $f(u)<0$ for $u>q_i$ but close to $q_i$. 

Due to \eqref{gamma}, we can apply Lemma A with $q_*=0$ and $q^*=p$ to obtain the following conclusion.

{\bf Lemma B.} {\it
Under the assumptions {\rm {\bf (f1)-(f3)}}, \eqref{1D} has a unique propagating terrace connecting $p$ to $0$.
}

We will denote the unique propagating terrace in Lemma B by 
\beq\label{terrace-q-U}
\big\{q_{i_k}, U_k, c_k\big\}_{1\leq k\leq n_0},
\eeq
 with $U_k$ the traveling wave connecting $q_{i_k}$ to $q_{i_{k-1}}$ of speed $c_k$. To simplify notations we will write
 $Q_k:=q_{i_k}$ and so
\[
0=Q_{n_0}<...<Q_0=p,\;\; 0<c_1\leq ... \leq c_{n_0},
\]
and for $k\in\{1,..., n_0\}$, $U_k(z)$ satisfies
\[
U_k''+c_k U_k'+f(U_k)=0,\; U_k'<0 \mbox{ for } z\in\R,\; U_k(-\infty)=Q_{k-1},\; U_k(+\infty)=Q_k.
\]

Since $U_k$ is only unique up to a shift of its variable, for definiteness, we normalize each $U_k$ by further requiring
\[
U_k(0)=(Q_{k-1}+Q_{k})/2.
\]
With this normalization,  the $U_k$ in \eqref{terrace-q-U} is uniquely determined, and we will use this convention in the rest of the paper.

In general,  the speeds $\{c_k: k=1,..., n_0\}$ in the propagating terrace need not be distinct from each other, although they are distinct for a generic $f$; see \cite{MP}. Similar to \cite{MP}, we can obtain more precise result under the following generic condition on $f$:
\begin{itemize}
\item[{\bf (f4)}] The speeds  in the propagating terrace $\big\{Q_k, U_k, c_k\big\}_{1\leq k\leq n_0}$ are distinct.
\end{itemize}

\subsection{Main results}

We start with a rather precise result under the extra condition {\bf (f4)}, which is the analogue of \eqref{u-U-bs} for multistable $f$, and follows as a corollary of the main results.

\begin{prop}\label{main-1}

Suppose that $f$ satisfies {\rm {\bf (f1)-(f4)}} and $u$ is a solution of \eqref{nd} satisfying \eqref{u-p}.
 If the initial function 
 $u_0$ is nonnegative and has compact support, then
  there exist   functions $\tilde\eta_k\in 
L^\infty(\R_+\times \mathbb S^{N-1})\cap C(\R_+\times \mathbb S^{N-1})$, $k=1,..., n_0$, such that
\begin{equation}\label{u-limit-1}
\lim_{t\to\infty}\left|u(x,t)-\sum_{k=1}^{n_0}\left[U_k\Big(|x|-c_kt+\frac{N-1}{c_k}\log t-\tilde\eta_k(t, \frac{x}{|x|})\Big)-Q_k\right]\right|=0
\end{equation}
uniformly for $x\in\mathbb R^N\setminus\{0\}$, where $\big\{Q_k, U_k, c_k\big\}_{1\leq k\leq n_0}$ is the propagating terrace of \eqref{1D}.

\end{prop}

 The condition on the initial function $u_0$ in  Proposition \ref{main-1} can be considerably relaxed. We now describe a more general condition, which  does not require $u_0$ to be compactly supported, or nonnegative.
 
Since $f'(0)<0$ and $f'(p)<0$, there exist  $0<b_*\leq b^*<p$ such that 
\[\begin{cases}
 f(u)<0 \mbox{ for } u\in (0, b_*),\; f(b_*)=0,\\
 f(u)>0 \mbox{ for } u\in (b^*, p),\; f(b^*)=0.
\end{cases}
\]
In other words, $b_*$ (resp. $b^*$) is the smallest (resp. largest) unstable zero of $f$ in $[0,p]$. 

Define
\begin{equation}\label{T(f)}
\mathcal T(f):=\Big\{\phi\in   L^\infty(\R^N): 
 \mbox{\rm There exists $R>0$ such that } \sup_{|x|>R}\phi(x)<b_*\Big\}.
\end{equation}
We can show that the conclusion in Proposition \ref{main-1} holds for any solution $u$ of \eqref{nd} satisfying \eqref{u-p}, provided that  $u_0\in\mathcal T(f)$; see Theorem \ref{thm-log-shift} below. 

If the generic condition {\bf (f4)} is not assumed, a less precise version of Proposition \ref{main-1} holds, where we have to replace 
 $\frac{N-1}{c_k}\log t$ in \eqref{u-limit-1} by a $C^1$ function $\zeta_k(t)$ satisfying $\lim_{t\to\infty}\zeta_k'(t)= 0$; see Theorem \ref{thm-conv-ter}.
\medskip

Our main results are the following theorems.

\begin{thm}\label{thm-conv-ter} {\rm (\underline{Convergence to the propagating terrace})}
Suppose that $f$ satisfies {\rm {\bf (f1)-(f3)}}  and $u$ is a solution of \eqref{nd} satisfying \eqref{u-p}.
If  $u_0\in\mathcal T(f)$, then  
 for every $k\in\{1,..., n_0\}$,  there exist functions $\tilde\eta_k\in 
L^\infty(\R_+\times \mathbb S^{N-1})\cap C(\R_+\times \mathbb S^{N-1})$, and  $\zeta_k\in C^1(\R_+)$ satisfying 
\[\begin{cases}
\ \lim_{t\to\infty}\zeta_k'(t)=0, \\
\mbox{ $c_k=c_{k+1}$ implies } \lim_{t\to\infty}\big[\zeta_{k+1}(t)-\zeta_{k}(t)\big]=+\infty,
\end{cases}
\]
  such that
\begin{equation}\label{u-limit-w}
\lim_{t\to\infty}\left|u(x,t)-\sum_{k=1}^{n_0}\left[U_k\Big(|x|-c_kt+\zeta_k(t)-\tilde\eta_k(t, \frac{x}{|x|})\Big)-Q_k\right]\right|=0
\end{equation}
uniformly for $x\in\mathbb R^N\setminus\{0\}$, where $\big\{Q_k, U_k, c_k\big\}_{1\leq k\leq n_0}$ is the propagating terrace of \eqref{1D}.
\end{thm}

\begin{thm}\label{thm-levset}{\rm(\underline{Level set behaviour})} Under the assumptions of Theorem \ref{thm-conv-ter}, 
for any $a\in (Q_k, Q_{k-1})$ with $k\in \{1,..., n_0\}$, the level set 
\[
\Gamma_a(t):=\big\{x\in\R^N: u(x,t)=a\big\}
\]
 is a smooth closed hypersurface in $\R^N$ for all large $t$, say $t\geq T_a$, which is given by an equation of the form
\[
x=\xi_a(t,\nu)\nu,\ \  \nu\in \mathbb S^{N-1},\; t\geq T_a,
\]
with $\xi_a\in C^1([T_a,\infty)\times \mathbb S^{N-1})$ satisfying
\begin{equation}\label{xi_a(t)}\begin{cases}
\displaystyle\lim_{t\to\infty}\frac{\xi_a(t,\nu)}{t}=c_k \mbox{ uniformly for } \nu\in \mathbb{S}^{N-1},\\
\limsup_{t\to\infty}{\rm osc}[\xi_a(t,\cdot)]<+\infty,
\end{cases}
\end{equation}
where 
\[
{\rm osc}[\xi_a(t,\cdot)]:=\max_{\nu\in \mathbb S^{N-1}}\xi_a(t,\nu)-\min_{\nu\in \mathbb S^{N-1}}\xi_a(t,\nu).
\]
\end{thm}

\begin{thm}\label{thm-conv-tw} {\rm(\underline{Convergence to a traveling wave})} Under the assumptions of Theorem \ref{thm-levset},
 for any bounded set $O\subset \R^N$, 
\begin{equation}
\label{u-Uk-loc}
\lim_{t\to\infty}u(x+\xi_a(t,\nu)\nu, t)= U_k(x\cdot \nu+\alpha_k^a)
\end{equation}
  uniformly for $x\in O$ and $\nu\in\mathbb{S}^{N-1}$, where $\alpha_k^a$ is given by $U_k(\alpha_k^a)=a$.
  \end{thm}
 
 \begin{thm}\label{thm-log-shift}{\rm (\underline{Logarithmic shifts})} Under the assumptions of Theorem \ref{thm-conv-ter}, 
 if additionally {\bf (f4)} holds, then we can take $\zeta_k(t)=\frac{N-1}{c_k}\log t$ in \eqref{u-limit-w}, and the level set function $\xi_a(t,\nu)$ in Theorem \ref{thm-levset} satisfies
 \begin{equation}\label{xi_a-sharp}
 \max_{\nu\in\mathbb S^{N-1}}|\xi_a(t,\nu)-c_kt+\frac{N-1}{c_k}\log t|\leq C \mbox{ for all large $t$ and some $C>0$}.
 \end{equation}
\end{thm}

The proofs of these results rely on the construction of a special, radially symmetric,  solution of \eqref{nd}, which we call a  radial terrace solution. 
\medskip

{\bf Definition} (\underline{radial terrace solution}): {\it Under the conditions {\rm {\bf (f1)-(f3)}} for $f$, with propagating terrace $\big\{Q_k, U_k, c_k\big\}_{1\leq k\leq n_0}$ as described above,
 a function $v(x,t)$ in $C^2(\R^N\times (0,\infty))\cap C(\R^N\times [0,\infty))$ is called a radial terrace solution of \eqref{nd} connecting $p$ to $0$ if
\begin{enumerate}
\item[(i)]
$v_t-\Delta v=f(v), \; 0<v<p,\; \mbox{ and } v_t>0 \mbox{ for } x\in\R^N, t>0.
$
\item[(ii)] $v(x,0)$ is continuous, nonnegative, radially symmetric and has compact support, and therefore, for each fixed $t\geq 0$, $v(x,t)$ is radially symmetric in $x$: $v(x,t)=V(r,t) \; (r=|x|)$.
\item[(iii)]  As $t\to\infty$, $V(r,t)$ converges to the propagating terrace of \eqref{1D} connecting $p$ to 0, in the following sense:
\begin{equation}
\label{V-limit}
\lim_{t\to\infty}\left(V(r,t)-\sum_{k=1}^{n_0} \Big[U_k(r-c_kt-\eta_k(t))-Q_{{k}}\Big]\right)=0 \mbox{ uniformly for } r\in [0,\infty),
\end{equation}
where, for $k=1,..., n_0$, $\eta_k(t)$ is a $C^1$ function on $(0,\infty)$ satisfying
\[\begin{cases}
\ \lim_{t\to\infty}\eta_k'(t)=0, \mbox{ and }\\
\mbox{ $c_k=c_{k+1}$ implies } \lim_{t\to\infty}\big[\eta_{k+1}(t)-\eta_{k}(t)\big]=+\infty.
\end{cases}
\]
\end{enumerate}
}
From \eqref{V-limit}, one sees that if $v(x,t)$ is a radial terrace solution, then, in particular,
\[
\lim_{t\to\infty} v(x,t)=p \mbox{ locally uniformly for } x\in\R^N.
\]

\begin{thm}
\label{thm-existence-rt}{\rm(\underline{Existence of a radial terrace solution})} Suppose that {\bf (f1)-(f3)} hold. Then \eqref{nd} has a radial terrace solution connecting $p$ to 0.
\end{thm}

The following result plays a crucial role in our proof of Theorems \ref{thm-conv-ter}, \ref{thm-levset}, \ref{thm-conv-tw} and \ref{thm-log-shift}.

\begin{prop}
\label{prop2}
Suppose that $f$ satisfies {\bf (f1)-(f3)},  $V(r,t)$ is a radial terrace solution of \eqref{nd},
and  $u(x,t)$ is a solution of \eqref{nd} satisfying \eqref{u-p}.
If $u_0\in \mathcal T(f)$,
 then 
there exist positive constants $T$, $T_0$, $ \sigma$ and $\beta$ such that,  for all $x\in\R^N$ and all  $t\geq T$,
\[
V(|x|, t-T)-\sigma e^{-\beta (t-T)}\leq u(x,t)\leq V(|x|, t+T_0)+\sigma e^{-\beta (t-T)}.
\]
\end{prop}

\begin{rmk}{\rm

(i)
The bounded oscillation property of $\xi_a(t,\cdot)$ in the second part of \eqref{xi_a(t)} can be deduced from the reflection argument of Jones \cite{J}
if the initial function $u_0$ is nonnegative with compact support. Here it is proved for much more general initial functions, namely $u_0\in\mathcal T(f)$, by making use of Proposition \ref{prop2} and the properties of the radial terrace solution (see Lemma \ref{|x|-bd}).

(ii) Though not pursued here,  we expect $\lim_{t\to\infty}\tilde\eta_k(t,\nu)$ exists in \eqref{u-limit-w}, and it is a function of $\nu\in\mathbb S^{N-1}$, but  not a constant in general, as demonstrated in \cite{R, Y} for the special bistable case. For the Fisher-KPP case, a related convergence result can be found in \cite{RRR}.}
\end{rmk}

\subsection{Background and related results} To put our results into perspective, let us briefly discuss the background of the problem and some related results.
There is extensive literature on the long-time behavior of solutions of \eqref{nd}, in particular on the spreading properties of fronts as $t\to\infty$. Here, the term ``front'' is a somewhat vague notion, but it can be roughly understood as the area where the solution exhibits a clear transition between stationary states. Thus the position of the
 front of a solution is roughly represented by the level set $\{x: u(x,t)=a\}$, 
where $a$ is any fixed value between the two stationary states that are under consideration.

If $f$ is a typical classical nonlinearity such as the monostable, the bistable or the combustion type, the properties of traveling wave solutions and spreading fronts are well understood; 
see, e.g., \cite{A-W,FM, G, MN, R, U, Y} and the references therein. If $U$ is a traveling wave (profile) connecting $p$ to 0 with speed $c$, then
in one space dimension, it generates a solution to \eqref{nd} of the form $U(x-ct)$, which is the real traveling wave with speed $c$. Very often people say $U$ is a traveling wave solution without distinguishing $U(x)$ and $U(x-ct)$; this convention is used here as well.  In the higher dimension
  case $N\geq 2$,  for each $\nu\in \mathbb S^{N-1}$, \eqref{nd} possesses a solution of the form  $U(x\cdot \nu-ct)$, which is called a planar wave in the direction $\nu$. As shown in \cite{A-W}, if a nonnegative solution of \eqref{nd} with compactly supported initial datum converges to $p>0$ locally uniformly as $t\to\infty$, then the asymptotic speed of the front propagating toward infinity -- which we call the spreading front -- coincides with the speed of the one dimensional traveling wave (or the minimal traveling wave speed in the case of a monostable nonlinearity $f$). Furthermore, the profile of the solution around the front region is known to converge locally to that of the traveling wave; see precise results  in \cite{J, R, U} for bistable nonlinearity and \cite{D, G, RRR, U} for the monostable case.

 For more general $f$, such as those having many zeros between $0$ and $p$, the behavior of solutions of \eqref{nd} can be far more complex. Regarding the stabilization of solutions in large time, in our earlier paper \cite{Du-M}, we have shown by simply assuming that $f$ is locally Lipschitz with $f(0)=0$, that in one space dimension, any globally bounded nonnegative solution with compactly supported initial data converges to a stationary solution, and that this limit stationary solution is either a constant solution, or a symmetrically decreasing solution. 
 Furthermore, our ``sharp transition'' result in the same paper \cite{Du-M} shows that if one considers a family of initial data, then convergence to a constant stationary solution is a generic phenomenon; all other behaviors are non-generic. Note that we proved the above sharp transition result for bistable and combustion nonlinearities, but the same proof applies to a much wider class of nonlinearities. The results in \cite{Du-M} generalize the work \cite{Z}, which considers the special case where the initial function $u_0$ is the characteristic function of a finite interval in $\R$. 
 A higher dimensional extension of the above convergence result is found in \cite{Du-P} under some additional assumptions on the nondegeneracy of the zeros of $f$. It is expected that in higher dimension, convergence to a constant stationary solution is also a generic phenomenon. Therefore \eqref{u-p} is a natural assumption.
 
If  $u$ is a solution of \eqref{nd} with  $u_0\in\mathcal T(f)$ that converges to a  constant $p>0$ as $t\to\infty$ locally uniformly, 
then the level set $\{x: u(x,t)>a\}$ for any $a\in (0,p)$ spreads over the entire space $\R^N$ as $t\to\infty$, but the stabilization results in \cite{Du-M, Du-P, Z} do not tell how this spreading occurs. In the one dimension case, this question was addressed in \cite{MP},
where among other results, it was shown that, if the generic condition {\bf (f4)} is satisfied and $u_0\geq 0$, $u_0(x)\to 0$ as $x\to\pm\infty$, then a stronger version of \eqref{u-limit-1} holds;
in this case ($N=1)$ the term $\frac{N-1}{c_k}\log t$ disappears and $\tilde\eta_k(t,\pm1)$ is shown to converge to some $L_{\pm}\in\R$ as $t\to\infty$. Apart from dealing with other questions, \cite{MP} also covers more general $f$ and allows $p$ to be replaced by a  symmetrically decreasing stationary solution (the only other possibility under these type of initial conditions).
The proof in \cite{MP} uses the zero number argument, which is not available in higher dimensions in general.

In higher dimension, if $u_0$ is nonnegative and compactly supported, and if  $f$ is bistable,  a weak version of  \eqref{u-limit-w} 
 was proved by Jones \cite{J} via a very different dynamical systems approach; for a general multistable $f$,
if $u_0$ satisfies additionally  $0\leq u_0\leq p$, and if {\bf (f4)} holds, then  the first identity in \eqref{xi_a(t)}  can  be derived from Theorem 1.7  of Rossi \cite{Rossi}, where the main interest was on a more general equation in  space-periodic medium. In the less general setting here, our  results provide more precise description of the propagation profile of $u$.

For radially symmetric solutions,  Risler \cite{Risler} considered  some rather general gradient systems and defined a
 propagating terrace of bistable fronts in such a setting. Under the assumption that every critical point of the associated energy functional is nondegenerate, which in our special case here is equivalent to every zero of $f$ is nondegenerate,
 he has obtained results in the fashion of \eqref{V-limit} except that his conclusion is only for $r\in [\epsilon t, \infty)$ with  $\epsilon>0$.
  Unlike our method  which is based on the maximum principle, his approach exploits the gradient structure and is applicable to more general situations. Under the generic assumption {\bf (f4)}, in the less general framework here, our Theorem \ref{thm-log-shift} applied to  radially symmetric solutions
  gives more precise information of the solution
  than  \cite{Risler} (see also Remark \ref{V-limit-precise}).

In one dimension, for ``front-like" initial functions (these are not contained in $\mathcal T(f)$), including in particular those satisfying $0\leq u_0\leq p$ and $u_0(-\infty)=p$, $u_0(+\infty)=0$,  Pol\'a\v{c}ik  \cite{P2} has proved that
the solution of \eqref{nd}  as $t\to\infty$ converges to the propagating terrace connecting $p$ to 0, where much more general $f$ is allowed than here; in particular, $f$ need not be multistable,  no assumption is needed on the nondegeneracy of any zeros of $f$,  and the propagating terrace there is  more general.

In \cite{P1},  Pol\'a\v{c}ik has extended results in \cite{P2} to the higher dimension case for  planar-like initial functions $u_0(x)$, namely,
$x_1\to u_0(x_1, x')$  behaves like the initial functions in \cite{P2} uniformly in $x'\in\R^{N-1}$. 
In such a case, as $t\to\infty$, the solution $u(x,t)$ converges to the corresponding one-dimensional propagating terrace in the direction $x_1$, uniformly in $x'\in\R^{N-1}$. In a sense, the situation in \cite{P1} is similar to our case here: In \cite{P1}, finer results are obtained by making use of the fact that the general solution for large time can be squeezed between two planar terrace solutions, while here finer results are obtained by the fact that the general solution for large time is squeezed between  two radial terrace solutions (see Proposition \ref{prop2}). The difference is that while in \cite{P1} the solution converges to the  planar propagating terrace, so its $\Omega$-limit set consists of functions depending on $x_1$ only, our solution in general does not converge to a radial function set obtained from the propagating terrace (see Remark 1.5 (ii)). Moreover,  the radial terrace solution here has to be constructed from scratch, and an unbounded shifting happens in the radial direction from the one dimensional propagating terrace, which does not occur in \cite{P1}.

\subsection{Organization of the paper}
 We construct the radial terrace solution in Section 2, which involves lengthy arguments given in several subsections, with a summary of the main ideas presented in subsection 2.3. One difficulty in the proof of \eqref{V-limit} is caused by the fact that a radial  solution $u(r,t)$ 
 of \eqref{nd} and the traveling waves $\{U_k(r-c_kt)\}$ from the one dimensional propagating terrace satisfy different equations, due to the term $\frac{N-1}{r}u_r$ in the equation for $u$. This renders the powerful zero number argument (or ideas involving the notion of steepness), which has played a crucial role in treating similar questions in the one dimensional case, not applicable anymore. In Section 3, we prove that if the speeds in the propagating terrace $\{Q_k, U_k, c_k\}_{1\leq k\leq n_0}$ satisfies $c_{k-1}<c_k<c_{k+1}$ for some $k\in\{1,..., n_0\}$, then the shifting function $\eta_k(t)$ in \eqref{V-limit} satisfies\footnote{See Remark \ref{V-limit-precise} for a better result, which applies to this particular $\eta_k(t)$
 under the assumption here.}
 \begin{equation}\label{eta_k-log}
 |\eta_k(t)+\frac{N-1}{c_k}\log t|\leq C \mbox{ for all large $t$ and some $C>0$}.
 \end{equation}
 Our more precise result in Theorem \ref{thm-log-shift}  under {\bf (f4)} is a consequence of this fact.
 The proof of \eqref{eta_k-log} is based on   the construction of subtle upper and lower solutions. Similar upper and lower solution techniques were first used in \cite{DMZ}, and subsequently further developed and used in \cite{DQZ, KMY}, to precisely determine logarithmic shifts in various different but related situations. The assumption $c_{k-1}<c_k<c_{k+1}$ is used to show that
 for any $\tilde c_{k-1}\in (c_{k-1}, c_k)$ and $\tilde c_k\in (c_k, c_{k+1})$, 
 \[
 |u(\tilde c_i t, t)-Q_i|\leq M_i e^{-\delta_i t} \mbox{ for all large $t$ and some $M_i,\; \delta_i>0$},\; i=k-1, k,
 \]
 which is crucial in the construction of the upper and lower solutions. 
 In Section 4, we prove Proposition \ref{prop2} and then use it to prove Theorems  \ref{thm-conv-ter}, \ref{thm-levset}, \ref{thm-conv-tw} and \ref{thm-log-shift},
   based on Theorem \ref{thm-existence-rt} obtained in Section 2 and \eqref{eta_k-log} (namely Theorem \ref{eta_k}) proved in Section 3.
   A result of Berestycki and Hamel \cite[Theorem 3.1]{BH} also plays an important role here.

\section{Construction of  radial terrace solutions}

This section is devoted to the proof of  Theorem \ref{thm-existence-rt}. 
For clarity, the arguments are grouped and given in several subsections. Recall that we always assume that
$f$ satisfies {\bf (f1)-(f3)}.

\subsection{Some useful facts on \eqref{nd} with compactly supported initial functions}

Suppose that   $u(x,t)$ is the solution of \eqref{nd} with a continuous nonnegative initial function
$u_0$ having non-empty compact support. By the properties of $f$, one sees that $u(x,t)$ is defined and positive for all $t>0$. We recall two basic properties of $u$ which will play an important role in our analysis later.

 \begin{lem}
 \label{mono}
 Let $B_0$ be a ball centered at the origin
    that contains the support of $u_0$. Then
    \begin{equation*}
      u_r(x,t):=\nabla u(x,t)\cdot x/|x|<0\quad \mbox{ for } x\in\R^N\setminus
      B_0,\ t>0.
    \end{equation*}
 \end{lem}

The above conclusion  follows from a well-known reflection argument of Jones \cite{J}; a proof can also be found
in \cite[Lemma 2.1]{Du-P}.

 \begin{lem}
 \label{infty}
 For each $t>0$, $\lim_{|x|\to\infty} u(x,t)=0$.
 \end{lem}
 \begin{proof} This is also well known. We give a simple proof here for completeness.
 Since $u(x,t)$ is bounded and $f$ is $C^1$ with $f(0)=0$, there exists $M>0$ such that
 \[ f(u(x,t))\leq M u(x,t) \mbox{ for all } x\in\R^N, \; t>0.
 \]
 Let $\bar u(x,t)$  be the solution of the following problem:
\begin{equation}
\label{alpha}
 \bar u_t=\Delta \bar u\; \ \ \hbox{for}\ (x,t)\in
\R^N\times (0,\infty),\quad\ \bar u(0,x)=u_0(x)\ \ \hbox{for}\ x\in
\R^N.
\end{equation}
Then
\begin{equation}
\label{alpha-form}
 {\bar u}(x,t)=\int_{\Omega}(4\pi
t)^{-N/2}\exp\Big(-\frac{|x-y|^2}{4t} \Big)u_0(y)dy,
\end{equation}
where $\Omega={\rm spt}(u_0)$. One easily checks that $e^{-Mt}u$
is a sub-solution of
\eqref{alpha} for $(x,t)\in \R^N\times (0,\infty)$; hence
\begin{equation}\label{u-ubar}
0\leq u(x,t)\leq e^{Mt}\:\!\bar u(x,t)\quad \
\hbox{for all} \ \ (x,t)\in \R^N\times (0,\infty).
\end{equation}
The conclusion of the lemma then follows easily since by \eqref{alpha-form},
clearly $\bar u(x,t)\to 0$ as $|x|\to\infty$ for each fixed $t>0$.
 \end{proof}

\subsection{Choosing the initial function}
In this subsection we choose a nonnegative radially symmetric initial function $u_0$ that has compact support, so that the solution of \eqref{nd} with this initial function will be a radial terrace solution.

Let $u_0^*$ be given by \eqref{u*0}. Then there exists $\epsilon_0>0$  such that
 \[
 u^*_0(x)\leq p-\epsilon_0 \mbox{ for } x\in\R^N, \mbox{ and } f(u)<0 \mbox{ for } u\in (p, p+\epsilon_0).
\]
 For $\epsilon\in (0,\epsilon_0)$, we consider the initial value problem
 \begin{equation}
 \label{ivp}
 v''(r)+\frac{N-1}{r}v'(r)+\tilde f(v)=0,\; v(0)=p-\epsilon,\; v'(0)=0,
 \end{equation}
 where $\tilde f(u)$ is a $C^1$ function which is identical to $f(u)$ for $u\leq p+\epsilon_0$, and $\tilde f(u)<0$ for all $u\geq p+\epsilon_0$.
 It is well known that \eqref{ivp} has a unique solution defined on some interval $r\in [0, R)$. Let $R_0>0$ be the maximal value such that
 $v(r)$ is defined and is positive for $r\in [0, R_0)$. Then either $R_0=+\infty$, or $R_0<+\infty$ and $v(R_0)\in\{0,+\infty\}$. (Since $\tilde f(u)<0$ for $u>p$,
 it is easily seen that
 $\limsup_{r\to R_0}v(r)=+\infty$ implies $v(r)\to+\infty$ as $r\to R_0$.)

 We claim that for all sufficiently small $\epsilon>0$, and $R(\theta)$ given in \eqref{u*0}, the following holds:
 \begin{equation}
 \label{v-prop}
 \mbox{$R(\theta)<R_0<+\infty$, $v(R_0)=0$ and $v'(r)<0$ for $r\in (0, R_0]$.}
 \end{equation}
  Since $v\equiv p$ satisfies \eqref{ivp} with $\epsilon=0$,
 by continuous dependence there exists $\epsilon_1\in (0,\epsilon_0)$ such that for each $\epsilon\in (0,\epsilon_1]$, the value $R_0$ defined above satisfies $R_0>R(\theta)$
 and $v(r)\geq p-\epsilon_0$ for $r\in [0, R(\theta)]$.  We fix such an $\epsilon$. If $R_0=+\infty$, then we have $u^*_0(x)<v(|x|)$ in $\R^N$ and hence by the comparison
 principle we deduce $u^*(x,t)<v(|x|)$ for all $x\in\R^N$ and $t>0$, where $u^*(x,t)$ is the solution of \eqref{nd} with initial function $u^*_0(x)$. By the choice of $u_0^*$, we have $\lim_{t\to\infty}u^*(x,t)=p$ ($\forall x\in\R^N$). It follows that $p\leq v(0)<p$,  a contradiction. If $R_0<+\infty$ and $v(R_0)=+\infty$ then we can similarly apply the comparison principle to deduce $u^*(x,t)<v(|x|)$ for $|x|<R_0$ and $t>0$, which leads to the same contradiction. Therefore we necessarily have $R(\theta)<R_0<+\infty$ and $v(R_0)=0$.

 To complete the proof of our claim, it remains to show that $v(r)<p$ in $[0,R_0]$. 
Indeed, $v(|x|)$ is a positive solution of the Dirichlet problem
\[
\Delta v+f(v)=0 \mbox{ in } B_{R_0}(0),\; v=0 \mbox{ on } \partial B_{R_0}(0).
\]
The well known moving plane method infers  that such a solution satisfies $v_r(r)<0$ for $r\in (0, R_0]$. Hence $v(r)\leq v(0)<p$ for $r\in [0, R_0]$. The claim is now fully proved.

 We now define
 \[
  u_0(x)=\left\{\begin{array}{ll} v(|x|)& \mbox{ for } |x|<R_0,\\
 0& \mbox{ for } |x|\geq R_0.
 \end{array}
 \right.
 \]
 Since $\tilde f(v(r))=f(v(r))$, we see that $ u_0(x)$ satisfies, in the weak sense,
 \[
 -\Delta  u_0\leq f( u_0) \mbox{ in } \R^N,
 \]
 and $ u_0$ is not a stationary solution of \eqref{nd}. Therefore the unique solution $ u$ of \eqref{nd} with initial function $u_0$ satisfies
\begin{equation}
\label{mono-t}
u_t>0 \mbox{ for } x\in \R^N, \; t>0.
\end{equation} 
Clearly $u$ is radially symmetric in $x$. We will from now on write $u=u(r,t)$ ($r=|x|)$.
Since $u_0^*(x)<u_0(x)<p$ in $\R^N$, by the choice of $u_0^*$ we see that 
\begin{equation}
\label{t-infty}
\lim_{t\to\infty}u(r,t)=p \mbox{ locally uniformly for } r\in [0,\infty).
\end{equation}
Since $u(r,0)$ is non-increasing in $r$ and $u_r(r,0)<0$ for $r\in (0, R_0)$, by the reflection argument again we further have
\begin{equation}
\label{mono-r}
u_r(r,t)<0 \mbox{ for } t>0,\; r> 0.
\end{equation}
By Lemma \ref{infty}, we have
\begin{equation}
\label{r-infty}
\lim_{r\to\infty}u(r,t)=0 \mbox{ for every } t>0.
\end{equation}

\bigskip

To complete the proof of Theorem \ref{thm-existence-rt}, it remains to show that $u(r,t)$ satisfies \eqref{V-limit},  which 
will be done in the following   subsections. 
Since the arguments  are rather lengthy, we first describe  the main ideas. 

\subsection{Main ideas in the proof of \eqref{V-limit}}

For each $b\in (0, p)$, by  \eqref{t-infty}, \eqref{mono-r} and \eqref{r-infty}, we easily see that there
exists a unique $\xi_b(t)$ for all large $t$, say $t>T_b$, such that
\begin{equation}
\label{xi_b}
u(\xi_b(t), t)=b,
\end{equation}
and $\xi_b(t)$ is increasing in $t$ with
\[
\lim_{t\to\infty}\xi_b(t)=\infty.
\]
By the implicit function theorem $\xi_b(t)$ is a $C^1$ function of $t$.

It is expected that, as $s\to\infty$, $u(r+\xi_b(s), t+s)$ converges to $U^b(r-c^bt)$, where $U^b$ is a traveling wave solution with speed $c^b$ connecting two stable zeros of $f$, with $b$ lying between them. Moreover, with $b_i$ given by \eqref{b_i}, $i=1,..., m$, we expect that $\{U^{b_i}: i=1,..., m\}$ is the collection of traveling waves appearing in the propagating terrace $\{Q_k, U_k, c_k\}_{1\leq k\leq n_0}$.

For any sequence $t_k\to+\infty$, it is easy to show that, subject to a subsequence,
\[
\lim_{k\to\infty}u(r+\xi_b(t_k), t+t_k)=w^b(r,t),
\]
with $w^b$ satisfying $w^b(0,0)=b$ and 
\[
w^b_t-w^b_{rr}=f(w^b),\ 0\leq w^b\leq p,\; w^b_t\geq 0\geq w^b_r \ \ \mbox{ for } r,\ t\in\R.
\]

In subsection 2.4, the following estimates are proved:
\[
-u_r(\xi_b(t), t),\ u_t(\xi_b(t), t), \ \xi_b'(t)\geq \sigma>0 \mbox{ for all large } t>0.
\]
These form the basic tools for the subsequent arguments  towards showing that $w^b$ is a traveling wave. In particular,
the above inequalities imply that $w^b_t>0>w^b_r$ and hence the equation $w^b(r,t)=b$ uniquely determines a $C^1$ function $r=\zeta_b(t)$ and
\[
\zeta_b(t)=\lim_{k\to\infty}\big[\xi_b(t+t_k)-\xi_b(t_k)\big].
\]

In subsection 2.5, it is shown that there exist stable zeros $q_i>q_j$ of $f$ such that
\[
\begin{cases}
\lim_{t\to+\infty}w^b(r,t)=\lim_{r\to-\infty} w^b(r,t)=q_i,\\
\lim_{t\to-\infty}w^b(r,t)=\lim_{r\to+\infty} w^b(r,t)=q_j
\end{cases}
\]
Moreover, if 
\begin{equation}
\label{zeta_b}
\mbox{ $\zeta_{b_j}(t)-\zeta_{b_{i+1}}(t)\leq C$ for all $t\in\R$ and some $C\in\R$,}
\end{equation}
 then $w^b$ is a traveling wave connecting $q_i$ to $q_j$.

Lemma 2.18 in subsection 2.6  states that for each $i\in\{1,..., m-1\}$,
\[
\rho_i(t):=\xi_{b_{i+1}}(t)-\xi_{b_i}(t)
\]
 either remains bounded as $t\to+\infty$, or it converges to $+\infty$. This implies \eqref{zeta_b}, and
thus $\{w^{b_i}: i=1,..., m\}$ consists of traveling wave solutions, which form the propagating terrace connecting $p$ and $0$.
The rest are easy consequences of this conclusion.

Part of the difficulty in the proof of \eqref{V-limit} is due to the fact that $u(r,t)$ and $w(r,t)$ satisfy different equations,
which does not allow the use of zero number argument (or ideas involving steepness) to functions obtained from the difference of suitable shifts of $u$ and $w$. Let us note that such a situation does not occur in
the one dimension case.

\subsection{Properties of the level sets of $u(r,t)$}

In this subsection, we prove the following important properties of the level set function $\xi_b(t)$:
For any $b\in [0,p]\setminus\{q_0,..., q_m\}$, there exists $T_b>0$ and $\delta=\delta_b>0$ so that
\[
u_r(\xi_b(t), t)\leq -\delta,\; u_t(\xi_b(t), t)\geq \delta,\; \xi_b'(t)\geq \delta \mbox{ for } t\geq T_b.
\]
Moreover, $T$ and $\delta$ can be chosen uniformly for $b$ outside any small neighborhood of $\{q_0,..., q_m\}$
in $[0,p]$.

We prove these properties by a sequence of lemmas.
\begin{lem}
\label{u-1}
Let $b_i\in\{b_1,..., b_m\}$. Then for every sufficiently small $\epsilon>0$, there exist $\delta=\delta(\epsilon, b_i)>0$, $T=T(\epsilon, b_i)> 0$ and $\epsilon_i>0$
such that
\begin{equation}
\label{2.5}
u_r(r,t)\leq -\delta \mbox{ whenever $t>T$ and $q_{i}+\epsilon\leq u(r,t)\leq b_i+\epsilon_i$.}
\end{equation}

\end{lem}
\begin{proof}
To simplify notations we write $b=b_i$ and $q=q_{i}$. For $\epsilon>0$ small,
we fix $\tilde q\in (q, q+\epsilon)$ and construct a $C^1$ function $\tilde f(u)$
such that
\[
\mbox{
$\tilde f(\tilde q)=\tilde f(b)=0$, $f(u)\leq \tilde f(u)\leq 0$ for $u\in (\tilde q, b)$ and $\int_{\tilde q}^u\tilde f(s)ds<0 $ for
$u\in (\tilde q, b]$.}
\]
By a simple first integral consideration, one sees that the problem
\[
-u''=\tilde f(u) \mbox{ for } r>0,\; u(0)=b,\; u(+\infty)=\tilde q
\]
has a unique solution $\underline u(r)$, and $\underline u'(r)<0$ for $r\in (0, +\infty)$. 
 We extend $\underline u(r)$ to $r<0$, say
until $r=-r_0<0$, with $r_0>0$ small so that $\underline u'(r)<0$ for $r\in [-r_0, 0]$, and $b+\epsilon_0<p$, where $\epsilon_0:=\underline u(-r_0)-b>0$.

For $\sigma\in\R$,
define $\underline u_\sigma(r):=\underline u(r-\sigma)$. Then
\[
-(\underline u_\sigma)_{rr}-\frac{N-1}{r}(\underline u_\sigma)_r\geq \tilde f(\underline u_\sigma)\geq f(\underline u_\sigma) \mbox{ for } r\in (\sigma-r_0, +\infty),
\]
and
\[
 b+\epsilon_0=\underline u_\sigma(\sigma-r_0)>\underline u_\sigma(r)> \underline u_\sigma(+\infty)=\tilde q>0 \mbox{ for } r\in(\sigma-r_0, +\infty).
\]

We now fix $t>T_b$ and consider $u(r,t)$. By \eqref{r-infty} there exists $R>R_0$ large so that $u(r,t)<\tilde q$ for $r\geq R$.
Hence for every $\sigma\geq R+r_0$ we have
\begin{equation}
\label{slide}
\underline u_\sigma(r)>\tilde q >u(r,t) \mbox{ for } r\in [\sigma-r_0, +\infty).
\end{equation}

For any given $c\in [q+\epsilon, b+\epsilon_0]$, there exists a unique $\tau=\tau(c)\geq -r_0$ such that $\underline u_\sigma(\sigma+\tau)=\underline u(\tau)=c$.
Define
\[
\sigma_*:=\inf \big\{\sigma: \underline u_\sigma(r)>u(r,t) \mbox{ for } r\in [\sigma+\tau, +\infty)\big\}.
\]
Then by \eqref{slide} we have $\sigma_*\leq R+r_0$. 

As before, due to \eqref{t-infty}, \eqref{mono-r} and \eqref{r-infty}, there exists a unique $\xi_c(t)$ defined for all large $t$ satisfying
\[
u(\xi_c(t),t)=c \mbox{ and } \lim_{t\to+\infty} \xi_c(t)=+\infty.
\]
By enlarging $T_b$ if necessary, we may assume that $\xi_c(t)$ is defined for $t>T_b$. 
Due to
 $u_r(r,t)<0$ for $r>0$,
 we have 
 \[
  u(r, t)>c \mbox{ for } r\in [0, \xi_c(t)).
 \]
 It follows that
  $\sigma_*\geq \xi_c(t)-\tau$.
Moreover, we have
\begin{equation}
\label{sigma_*}
\underline u_{\sigma_*}(r)\geq u(r,t) \mbox{ for } r\in [\sigma_*+\tau,+\infty).
\end{equation}

If $\sigma_*>\xi_c(t)-\tau$, then necessarily
 \begin{equation}
 \label{r_*}
\underline u_{\sigma_*}(r_*)= u(r_*,t)<c \mbox{ for some } r_*\in (\sigma_*+\tau,R).
\end{equation}
By \eqref{mono-t} we have $u(r,s)<u(r, t)\leq \underline u_{\sigma_*}(r)$ for $(r,s)\in [\sigma_*+\tau,R)\times [0, t)$.
Hence we may apply the parabolic maximum principle to compare $u(r,s)$ and $\underline u_{\sigma_*}(r)$ over the region
$[\sigma_*+\tau,R]\times [0, t)$ to conclude that
\[
u(r, t)<\underline u_{\sigma_*}(r) \mbox{ for } r\in (\sigma_*+\tau, R).
\]
This contradicts \eqref{r_*}.

Hence we must have $\sigma_*=\xi_c(t)-\tau$. We may now use $\underline u_{\sigma_*}(\sigma_*+\tau)=u(\xi_c(t),t)=c$ and
\eqref{sigma_*}  to conclude that
\[
u_r(\xi_c(t),t)=u_r(\sigma_*+\tau,t)\leq \underline u_{\sigma_*}'(\sigma_*+\tau)=\underline u'(\tau)<0.
\]
Thus we can take $\delta=\min\{-\underline u'(r): r\geq -r_0,\;  \underline u(r)\geq q+\epsilon\}$ and the proof is complete.
\end{proof}

\begin{lem}
\label{xi_b'}
With $\xi_b(t)$ and $\epsilon_i$ defined as in Lemma \ref{u-1}, for any small $\epsilon>0$, there exists $\sigma>0$ such that
\[
\xi'_{b}(t)\geq \sigma \mbox{ for all  $t\geq T$ and every $b\in \cup_{i=1}^m[q_{i}+\epsilon, b_i+\epsilon_i]$}.
\]
\end{lem}

The proof of Lemma \ref{xi_b'} relies on the following three lemmas, which are also  used later in the paper.

\begin{lem}
\label{w}
Suppose that the sequences $\{r_k\}, \{t_k\}\subset (0,+\infty)$ satisfy $r_k\to\infty, t_k\to\infty$ as $k\to\infty$, and define $u_k(r,t):=u(r+r_k, t+t_k)$.
Then subject to passing to a subsequence, $u_k\to w$ in $C_{loc}^{2,1}(\R^2)$, and $w=w(r,t)$ satisfies
\begin{equation}
\label{w-entire}
w_t-w_{rr}=f(w),\;  w_r\leq 0,\; w_t\geq 0\; \mbox{ for } (r,t)\in\R^2.
\end{equation}
\end{lem}

Since $\{\|u_k\|_\infty\}$ is bounded, the conclusions in Lemma \ref{w} are easily shown by making use of the parabolic  $L^p$ theory followed by the H\"{o}lder estimates,
and a standard diagonal process. Note that the term $\frac{N-1}{r+r_k}(u_k)_r$ disappears in the limit since $r_k\to\infty$, and the inequalities for $w_t$ and $w_r$ are consequences of \eqref{mono-t} and \eqref{mono-r}, respectively.
The detailed proof is omitted.

We will call a smooth function $w(r,t)$ defined in $\R^2$ satisfying \eqref{w-entire} a {\bf monotone entire solution}.  A typical monotone entire solution is a traveling wave solution:
$
w(r,t)=\Phi(r-ct) \mbox{ with } \Phi(z) $ satisfying
\[
\Phi''+c\Phi'+f(\Phi)=0,\; \Phi'<0 \mbox{ for } z\in \R.
\]

The following result is a simple extension of the well known Fife-McLeod super and sub-solution technique (see \cite{FM}).
\begin{lem}
\label{fm-sup}
Let $W(r,t)$ be a monotone entire solution satisfying \eqref{w-entire}, and suppose that
\[
\sup W=q_i,\; \inf W=q_j \mbox{ with } 0\leq i<j\leq m,
\]
and that for any small $\epsilon>0$, there exists $\delta=\delta(\epsilon)>0$ such that
\[
W_t-W_r\geq \delta \mbox{ whenever } W(r,t)\in [q_j+\epsilon, q_i-\epsilon].
\]
Then there exist positive constants $\beta$ and $\sigma$ such that
\[
U(r,t):=W(r-r_0+e^{-\beta t}, t-t_0-e^{-\beta t})+\sigma e^{-\beta t}
\]
satisfies, for any fixed $(r_0, t_0)\in\R^2$,
\[
U_t-U_{rr}\geq f(U) \mbox{ for } (r,t)\in \R\times [0,+\infty).
\]
Similarly,
\[
V(r,t):=W(r-r_0-e^{-\beta t}, t-t_0+e^{-\beta t})-\sigma e^{-\beta t}
\]
satisfies, for any fixed $(r_0, t_0)\in\R^2$,
\[
V_t-V_{rr}\leq f(V) \mbox{ for } (r,t)\in \R\times [0,+\infty).
\]
\end{lem}

\begin{proof} We only prove the conclusion for $U(r,t)$, as the proof for $V(r,t)$ is analogous. We calculate
\begin{align*}
U_t-U_{rr}&=W_t-W_{rr}+ (W_t-W_r-\sigma)\beta e^{-\beta t}\\
&=f(W)+(W_t-W_r-\sigma)\beta e^{-\beta t}\\
&=f(U)+J,
\end{align*}
with
\begin{align*}
 J&=(W_t-W_r-\sigma)\beta e^{-\beta t}+f(W)-f(W+\sigma e^{-\beta t})\\
 &=(W_t-W_r-\sigma)\beta e^{-\beta t}-f'(W+\theta )\sigma e^{-\beta t}\\
 &= e^{-\beta t}\Big\{\big[-f'(W+\theta)-\beta\big]\sigma+(W_t-W_r)\beta\Big\},
 \end{align*}
 where $W, W_t, W_r$ are evaluated at $(r-r_0+e^{-\beta t}, t-t_0-e^{-\beta t})$, and
 \[ \theta=\theta(r,t)\in [0, \sigma e^{-\beta t}].
 \]

 We now choose $\sigma$ and $\beta$ so that $J>0$. Since $f'(q_i), f'(q_j)<0$, there exist positive constants $\eta_0$ and $\epsilon$ such that
\[\mbox{
 $f'(u)<-\eta_0$ if $u\in [q_i-2\epsilon, q_i+2\epsilon]\cup [q_j-2\epsilon, q_j+2\epsilon]$.}
 \]
 Thus
 \[
 -f'(W+\theta)>\eta_0 \mbox{ when $W\in I:=[q_i-\epsilon, q_i+\epsilon]\cup [q_j-\epsilon, q_j+\epsilon]$ and $\sigma\leq \epsilon$.}
 \]
 We choose $\beta=\eta_0/2$. Next we choose $M_0>0$ such that
 \[
 -f'(W+\theta)-\beta\geq -M_0 \mbox{ for all $r\in\R$ and $t\geq 0$}.
 \]
 Finally there exists $\eta_1>0$ such that
 \[
 W_t-W_r\geq \eta_1 \mbox{ when } W\not\in I.
 \]
 Thus, if we choose $\sigma\in (0, \min\{\epsilon, \frac{\eta_0\eta_1}{2M_0}\})$, then
 \[
 J\geq \left\{\begin{array}{ll} e^{-\beta t}(-M_0\sigma+\frac{1}{2}\eta_0\eta_1)>0& \mbox{ when } W\not\in I,\vspace{0.2cm}\\
  e^{-\beta t}\cdot\frac{\eta_0}{2}\sigma>0 & \mbox{ when } W\in I.
 \end{array}
 \right.
 \]
 Therefore with $\beta$ and $\sigma$ as chosen above, we have
 \[
 U_t- U_{rr}\geq f(U) \mbox{ for } (r,t)\in\R\times [0,+\infty).
 \]
\end{proof}

\begin{lem}
\label{tilde tk}
Let $\gamma\in C^1([0,\infty))$ and $\{t_k\}$ be a sequence of positive numbers satisfying
\[
\sup_{t\geq 0}\gamma'(t)<+\infty,\;
\lim_{k\to\infty}t_k=+\infty,\; \lim_{k\to\infty} \gamma(t_k)=+\infty.
\]
Then there exists a sequence $\{\tilde t_k\}$ with the properties that
\[
\lim_{k\to\infty}\tilde t_k=+\infty,\; \lim_{k\to\infty}\gamma(\tilde t_k)=+\infty
\]
and
\[
\gamma(t+\tilde t_k)\geq \gamma(\tilde t_k) \;\forall t\in [0, k].
\]
\end{lem}
\begin{proof}
Set $C:=\sup_{t\geq 0}\gamma'(t)$. By passing to a subsequence, we may assume that
\[
\gamma(t_{k+1})>\gamma(t_k)+Ck \mbox{ for } k=1,2,... .
\]
Hence for $s\in [0,k]$, we have
\[
\gamma(t_{k+1}-s)\geq \gamma(t_{k+1})-Cs>\gamma(t_k)+C(k-s)\geq \gamma(t_k),
\]
that is,
\[
\gamma(t)>\gamma(t_k) \;\forall t\in [t_{k+1}-k, t_{k+1}].
\]
Define
\[
\tilde t_k:=\inf\{s: \gamma(t)>\gamma(t_k) \mbox{ for } t\in [s, t_{k+1}]\}.
\]
Clearly $\tilde t_k\in [t_k, t_{k+1}-k]$, and $\gamma(t)\geq \gamma(t_k)=\gamma(\tilde t_k)$ for $t\in [\tilde t_k, t_{k+1}]$. In particular, for $t\in [0,k]$, 
we have $\tilde t_k+t\in [\tilde t_k, t_{k+1}]$ and 
\[
\gamma(\tilde t_k+t)\geq \gamma(t_k)=\gamma(\tilde t_k).
\]
We clearly also have
\[
\tilde t_k\to+\infty,\; \gamma(\tilde t_k)\to+\infty \mbox{ as } k\to\infty.
\]
\end{proof}

\medskip

\noindent{\bf Proof of Lemma \ref{xi_b'}.} We break the proof into two steps. Fix $\epsilon>0$ sufficiently small.

\medskip

{\bf Step 1}. We prove that $\liminf_{t\to\infty} \big[\inf_{b\in [q_1+\epsilon, b_1+\epsilon_1]} \xi'_{b}(t)\big]>0$.

If this is not true, then there exists $t_k\to\infty$ and $b^k\to b\in [q_1+\epsilon, b_1+\epsilon_1]$ such that $\xi_{b^k}'(t_k)\to 0$. For $r>-\xi_{b^k}(t_k)$ and $t>-t_k$,
we define
\[
u_k(r,t)=u(r+\xi_{b^k}(t_k), t+t_k).
\]
By Lemma \ref{w}, subject to passing to a subsequence,
\[u_k\to w \mbox{ in } C^{2,1}_{loc}(\R^2),\]
and $w$ satisfies
\[
w_t-w_{rr}=f(w),\; w_r\leq 0,\; w_t\geq 0\; \mbox{ for } (r,t)\in\R^2.
\]
Since
\[ u_k(0,0)=b^k,\; (u_k)_r(0,0)\leq -\delta \;\;(\mbox{for all large $k$ by Lemma \ref{u-1}}),
\]
we have additionally
\[
 w(0,0)=b,\; w_r(0,0)\leq -\delta.
\]
From $u(\xi_{b^k}(t),t)=b^k$ we deduce
\[
u_t(\xi_{b^k}(t), t)+u_r(\xi_{b^k}(t),t)\xi'_{b^k}(t)=0.
\]
It follows that
\[
(u_k)_t(0,0)=u_t(\xi_{b^k}(t_k), t_k)=-(u_k)_r(0,0)\xi'_{b^k}(t_k)\to 0.
\]
Hence
\[
w_t(0,0)=0.
\]
Applying the strong maximum principle to the equation of $w_t$, it follows from the facts $w_t\geq 0$ and $w_t(0,0)=0$ that $w_t\equiv 0$.
Therefore $w$ is independent of $t$ and we may write
\[ w(r,t)=V(r)
\]
with $V$ satisfying
\[
-V''=f(V),\; V'\leq 0 \mbox{ for } r\in \R, \; V(0)=b_1,\; V'(0)\leq -\delta.
\]
The maximum principle then implies that $V'<0$ in $\R$. Standard ODE theory  indicates that $V(-\infty)$ and $V(+\infty)$ are  zeros of
$f$, and in view of \eqref{2.5} we also have  $V(-\infty), V(+\infty)\not\in \cup_{i=1}^m[q_{i}+\epsilon, b_i+\epsilon_i]$. Thus necessarily $V(-\infty)=p=q_0$ and $V(+\infty)=q_j$ with $j\geq 1$.

We show that the existence of such a $V(x)$ leads to a contradiction. Clearly $W(r,t):=V(r)$ satisfies the conditions in Lemma \ref{fm-sup}. Therefore we can find $\sigma,\beta>0$ such that, for any $R\in\R$,
\[
U(r,t):=V(r-R+e^{-\beta t})+\sigma  e^{-\beta t}
\]
satisfies
\[
U_t-U_{rr}\geq f(U) \mbox{ for } r\in\R, t\geq 0.
\]
Since $U_r=V'<0$, it follows that
\[
U_t-U_{rr}-\frac{N-1}{r}U_r\geq f(U) \mbox{ for } r>0, t\geq 0.
\]
Clearly $U_r(0,t)<0$ for all $t\geq 0$. We next show that if $R$ is chosen large enough in the definition of $U$, then $u_0(r)\leq U(r,0)$.

By \eqref{mono-t} and \eqref{t-infty},
\begin{equation}
\label{u0}
u_0(r)<\lim_{t\to\infty} u(r,t)= p.
\end{equation}
By definition,
 \[
 U(r,0)=V(r-R+1)+\sigma>\sigma>0 \mbox{ for all } r\geq 0.
 \]
 Moreover,
 \[
 U(r,0)=V(r-R+1)+\sigma\to p+\sigma \mbox{ as $R\to\infty$ uniformly for $r\in [0, R_0]$}.
 \]
 Therefore we can fix $R$ large enough such that
 \[
 U(r,0)>p \mbox{ for } r\in [0, R_0],
 \]
 and hence
  $u_0(r)<U(r,0)$ for all $r\geq 0$.

  We are now in a position to apply the parabolic comparison principle to conclude that
\[
u(r,t)\leq U(r,t) \mbox{ for $r>0$ and $t>0$}.
\]
It follows that
\[
U(\xi_{b_1}(t), t)\geq b_1 \mbox{ for all large } t.
\]
Since $\lim_{t\to\infty}\xi_{b_1}(t)=\infty$ by Lemma \ref{u-1}, and $V(+\infty)=q_j<b_1$, letting $t\to\infty$ in
\[
b_1\leq U(\xi_{b_1}(t),t)=V(\xi_{b_1}(t)-R+ e^{-\beta t})+\sigma  e^{-\beta t}
\]
we obtain $b_1\leq q_j$, a contradiction.
This completes the proof of Step 1.

\medskip

 {\bf Step 2.} We show that $\liminf_{t\to\infty} \big[\inf_{b\in [q_i+\epsilon, b_i+\epsilon_i]}\xi_{b}'(t)\big]>0$ for $i=2,..., m$.

 Otherwise, in view of Step 1, there exists $j\in\{2,..., m\}$, $T_0>0$ and $\sigma_0>0$ such that
 \begin{equation} \label{2.12}
 \left\{\begin{array}{l}
 \liminf_{t\to\infty}\big[\inf_{b\in[q_i+\epsilon, b_i+\epsilon_i]}\xi_{b_j}'(t)\big]=0 \mbox{ and }\\
  \inf_{b\in [q_i+\epsilon, b_i+\epsilon_i]}\xi'_{b}(t)\geq \sigma_0 \mbox{ for $i=1,..., j-1$ and $t\geq T_0$.}
  \end{array}\right.
 \end{equation}
 Therefore there exist $t_k\to\infty$, $b^k\to b\in [q_j+\epsilon, b_j+\epsilon_j]$  satisfying $\xi'_{b^k}(t_k)\to 0$. We claim that
 \begin{equation}
 \label{claim}
 \xi_{b^k}(t_k)-\xi_{b_{j-1}}(t_k)\to\infty.
 \end{equation}
 Otherwise by passing to a subsequence we may assume that
 \[
 \xi_{b^k}(t_k)-\xi_{b_{j-1}}(t_k)\to \eta\in\R.
 \]
 By Lemma \ref{w}, we may suppose, without loss of generality, that
 \[
 u(r+\xi_{b_{j-1}}(t_k), t+t_k)\to w(r,t) \mbox{ in } C^{2,1}_{loc}(\R^2),
 \]
 with $w$ satisfying
 \[
 w_t-w_{rr}=f(w), \; w_t\geq 0,\; w_r\leq 0 \mbox{ for } (r, t)\in\R^2,
 \]
 and
 \[
 w(0,0)=b_{j-1},\; w_r(0,0)\leq -\delta.
 \]
Moreover, using $u(\xi_{b_{j-1}}(t),t)=b_{j-1}$ and $\xi'_{b_{j-1}}(t)\geq \sigma_0$ we deduce
\[
w_t(0,0)=\lim_{k\to\infty}u_t(\xi_{b_{j-1}}(t_k), t_k)=\lim_{k\to\infty}-u_r(\xi_{b_{j-1}}(t_k), t_k)\xi'_{b_{j-1}}(t_k)\geq -w_r(0,0)\sigma_0>0.
\]
Hence we may apply the strong maximum principle to the equation satisfied by $w_t$ to conclude that $w_t(r,t)>0$ in $\R^2$.

On the other hand,
due to $\xi_{b^k}(t_k)-\xi_{b_{j-1}}(t_k)\to \eta$, we have
\[
w_t(\eta, 0)=\lim_{k\to\infty}u_t(\xi_{b^k}(t_k), t_k)=\lim_{k\to\infty}-u_r(\xi_{b^k}(t_k), t_k)\xi'_{b^k}(t_k)=0
\]
since $u_r(\xi_{b^k}(t_k), t_k)\to w_r(\eta,0)$ and $\xi'_{b^k}(t_k)\to 0$. This contradiction proves \eqref{claim}.

Next we consider the sequence of functions $u(r+\xi_{b^k}(t_k), t+t_k)$. As before we may use Lemma \ref{w} and assume that
\[
u(r+\xi_{b^k}(t_k), t+t_k)\to \tilde w(r,t) \mbox{ in } C^{2,1}_{loc}(\R^2).
\]
Then $\tilde w$ satisfies
\[
\tilde w_t-\tilde w_{rr}=f(\tilde w),\; \tilde w_t\geq 0,\; \tilde w_r\leq 0 \mbox{ for } (r,t)\in\R^2,
\]
and
\[
\tilde w(0,0)=b\in [q_j+\epsilon, b_j+\epsilon_j],\;
\tilde w_r(0,0)\leq -\delta.
\]
It follows that $\tilde w_r<0$ in $\R^2$. Moreover from $\xi'_{b^k}(t_k)\to 0$ we deduce, as before, $\tilde w_t(0,0)=0$. Hence $\tilde w_t=0$ in $\R^2$
and we may write $\tilde w(r,t)=\tilde V(r)$, with $\tilde V$ satisfying
\[
-\tilde V''=f(\tilde V),\; \tilde V'<0 \mbox{ in } \R,\;\; \tilde V(0)=b.
\]
We thus again conclude that $p^*:=\tilde V(-\infty)$ and $p_*:=\tilde V(+\infty)$ are stable zeros of $f$, and $p^*>b_j>p_*$.

For any fixed $r\in\R$, due to \eqref{claim}, $r+\xi_{b^k}(t_k)>\xi_{b_{j-1}}(t_k)$ for all large $k$. It follows that
\[
u(r+\xi_{b^k}(t_k), t_k)<u(\xi_{b_{j-1}}(t_k), t_k)=b_{j-1} \mbox{ for all large } k.
\]
Therefore $\tilde V(r)\leq b_{j-1}$ for all $r\in\R$. In particular $p^*=\tilde V(-\infty)\leq b_{j-1}$. Since $p^*>b_j$ we must have
\[
 p^*=q_{j-1}.
 \]
It follows that $p_*=q_i$ for some $i\geq j$. 
 Let $b_*:=q_i+\epsilon$. Then $b_*\leq b^k$ and hence $\xi_{b_*}(t)\geq \xi_{b^k}(t)$.
 It then follows from \eqref{claim} that
 \[
 \rho(t):=\xi_{b_*}(t)-\xi_{b_{j-1}}(t)\geq \xi_{b^k}(t)-\xi_{b_{j-1}}(t)\to\infty \mbox{ along the sequence } t=t_k.
 \]
 Since
 \[
 \rho'(t)\leq \xi_{b_*}'(t)=-\frac{u_t(\xi_{b_*}(t),t))}{u_x(\xi_{b_*}(t),t)}\leq \frac{1}{\delta}u_t(\xi_{b_*}(t),t),
 \]
 we see that $\rho'(t)$ is bounded from above for all large $t$. Consequently, the fact $\rho(t_k)\to\infty$ and Lemma \ref{tilde tk}  imply
 the existence of a sequence $\tilde t_k\to\infty$  such that
 \[
 \lim_{k\to\infty}\rho(\tilde t_k)=\infty,\; \rho(t)\geq \rho(\tilde t_k) \mbox{ for } t\in [\tilde t_k, \tilde t_k+k].
 \]
 We now consider a further sequence of functions  $u(r+\xi_{b_{*}}(\tilde t_k), t+\tilde t_k)$. As before we may assume that
\[
u(r+\xi_{b_{*}}(\tilde t_k), t+\tilde t_k)\to  w^*(r,t) \mbox{ in } C^{2,1}_{loc}(\R^2).
\]
Then $w^*$ satisfies
\[
w^*_t-w^*_{rr}=f(w^*),\; w^*_t\geq 0,\; w^*_r\leq 0 \mbox{ for } (r,t)\in\R^2,
\]
and
 \[
 w^*(0,0)=b_*,\; w^*_r(0,0)\leq -\delta.
 \]

We next show that $w^*$ has the following properties:
\begin{itemize}
\item[(i)] $w^*(\sigma_0 t, t)\geq b_*$ for all $t>0$, with $\sigma_0$ given in \eqref{2.12},
\item[(ii)] $\lim_{r\to-\infty}w^*(r,t)\leq q_{j-1}=p^*$ for all $t\in\R$,
\item[(iii)] $\lim_{r\to+\infty} w^*(r,t)\leq p_*$ for all $t\in\R$.
\end{itemize}
 To prove (i) we observe that for any fixed $t>0$, by the choice of $\tilde t_k$, we have $\rho(\tilde t_k+t)\geq \rho(\tilde t_k)$ for all large $k$.
It follows that, for all large $k$,
\[
\xi_{b_*}(\tilde t_k+t)-\xi_{b_{*}}(\tilde t_k)\geq \xi_{b_{j-1}}(\tilde t_k+t)-\xi_{b_{j-1}}(\tilde t_k)\geq \sigma_0 t.
\]
In view of \eqref{mono-r} we thus have
\[
 u(\sigma_0 t+\xi_{b_*}(\tilde t_k), \tilde t_k+t)\geq u(\xi_{b_*}(\tilde t_k +t), \tilde t_k+t)=b_*,
\]
which yields
\[
w^*(\sigma_0t, t)=\lim_{k\to\infty} u(\sigma_0 t+\xi_{b_*}(\tilde t_k), t+\tilde t_k)\geq b_*.
\]
This proves (i).

To prove (ii), we fix $(r,t)\in\R^2$ and observe that, for all large $k$, due to $\rho(\tilde t_k)\to\infty$,
\[
r+\xi_{b_*}(\tilde t_k)\geq \xi_{b_{j-1}}(\tilde t_k)+c|t|\geq \xi_{b_{j-1}}(\tilde t_k+t),
\]
where $c$ is chosen such that $c\geq \xi_{b_{j-1}}\rq{}(t)$ for all large $t$. It follows that
\[
u(r+\xi_{b_*}(\tilde t_k), t+\tilde t_k)\leq u(\xi_{b_{j-1}}(t+\tilde t_k), t+\tilde t_k)=b_{j-1}.
\]
Hence
\[
w^*(r,t)=\lim_{k\to\infty}u(r+\xi_{b_*}(\tilde t_k), t+\tilde t_k)\leq b_{j-1}.
\]
Choose an arbitrary sequence $r_k\to-\infty$ and consider the sequence $w^*(r+r_k, t)$. As before by regularity theory we can assume, without loss of generality,
that $w^*(r+r_k, t)\to W(r,t)$ in $C^{2,1}_{loc}(\R^2)$, and $W$ satisfies
\[
W_t-W_{rr}=f(W),\; W_t\geq 0,\; W_r\leq 0 \mbox{ in } \R^2,
\]
and by the above estimate for $w^*$ we also have $W\leq b_{j-1}$. Since $w^*(r,t)$ is monotone in $r$, necessarily $W$ is independent of $r$ and hence we may write
$W(r,t)=\alpha(t)$, and $\alpha(t)$ satisfies
\[
\alpha\rq{}(t)=f(\alpha(t)),\; \alpha\rq{}(t)\geq 0,\; \alpha(t)\leq b_{j-1} \mbox{ for all } t\in\R.
\]
Thus $f(\alpha(t))\geq 0$ for all $t\in\R$. This together with $\alpha(t)\leq b_{j-1}$ and the fact that $f(u)\leq 0$ for $u\in (q_{j-1}, b_{j-1})$
imply that either $\alpha(t)\leq q_{j-1}$ for all $t$, or $\alpha(t)\equiv \alpha\in [q_{j-1}+\epsilon, b_{j-1}]$, since 
$f(u)<0$ in $[q_{j-1}, q_{j-1}+\epsilon]$. However, this latter case is impossible since by \eqref{2.5} we know $w^*_r(r,t)\leq -\delta/2$ when $w^*(r,t)\in [ q_{j-1}+\epsilon, b_{j-1}]$.
Property (ii) is thus proved.

We now prove (iii). Similar to the argument for proving (ii), we choose $y_k\to+\infty$ and consider the sequence $w^*(r+y_k,t)$. Then
$\beta(t):=\lim_{k\to\infty}w^*(r+y_k, t)$ satisfies
\[
\beta\rq{}(t)=f(\beta(t)),\; \beta\rq{}(t)\geq 0 \mbox{ for all } t\in \R,
\]
and due to $w^*(0,0)=b_*$ and $w^*_r(0,0)<0$ we have $\beta(0)<b_*$. Therefore we may use $f(u)<0$ for $u\in (p_*, b_*]=(q_{j-1}, q_{j-1}+\epsilon]$ and
$f(\beta(t))=\beta\rq{}(t)\geq 0$ to deduce $\beta(0)\leq p_*$ and hence $\beta(t)\leq p_*$ for all $t\in\R$. This proves (iii).

We are now ready to deduce a contradiction by using properties (i)-(iii) of $w^*$ and the existence of $\tilde V$. We fix $t_0\in\R$
and show that
\[
w^*(r, t_0+t)\leq \tilde U(r,t) \mbox{ for all $r\in\R$ and $t>0$},
\]
where $\tilde U$ is given by
\[
\tilde U(r,t)=\tilde V(r-\tilde R+e^{-\tilde \beta t})+\tilde\sigma e^{-\tilde \beta t},
\]
with suitable choices of positive constants $\tilde R,\;\tilde \sigma$ and $\tilde\beta$.

We choose $\tilde \sigma$ and $\tilde \beta$ according to  Lemma \ref{fm-sup},
so that
\[
\tilde U_t-\tilde U_{rr}\geq f(\tilde U) \mbox{ for $r\in\R$ and $t>0$}.
\]
We next determine $\tilde R$ so that $\tilde U(r,0)\geq w^*(r, t_0)$. By (ii) and (iii) and the fact that $w^*_r<0$,
we have
\[
w^*(r, t_0)<p^* +\frac12 \tilde\sigma \mbox{ for all $r\in\R$, }
\]
and there exists $R_1>0$ so that
\[
w^*(r,t_0)<p_*+\tilde \sigma \mbox{ for } r\geq R_1.
\]
Since
\[ \tilde U(r,0)=\tilde V(r-\tilde R+1)+\tilde\sigma>p_*+\tilde \sigma \mbox{ for all $r\in\R$},
\]
and for $r\leq R_1$,
\[
\tilde U(r,0)\geq \tilde V(R_1-\tilde R+1)+\tilde\sigma\to p^*+\tilde\sigma \mbox{ as } \tilde R\to\infty,
\]
we can choose $\tilde R$ large enough such that
\[
\tilde U(r,0)\geq p^*+ \frac12 \tilde \sigma \mbox{ for all $r\leq R_1$}.
\]
Thus for $\tilde R$ chosen this way, we have $\tilde U(r,0)\geq w^*(r,t_0)$ for all $r\in\R$. We may now apply the comparison principle to
conclude that
\[
w^*(r,t_0+t)\leq \tilde U(r,t) \mbox{ for all $r\in\R$ and $t>0$}.
\]
Therefore we can use (i) to obtain
\[
b_*\leq w^*(\sigma_0t, t)\leq \tilde U(\sigma_0 t, t-t_0) \mbox{ for all $t>\max\{0,t_0\}$}.
\]
Letting $t\to+\infty$ we deduce $b_*\leq p_*$.
This contradiction completes the proof of Step 2 and hence the lemma.
\qed

\begin{lem}
\label{xi_c'}
For any small $\epsilon>0$, there exists $\sigma_\epsilon>0$ such that
\begin{equation}
\label{xi_c}
u_t(\xi_c(t),t),\; \xi_c'(t)\geq \sigma_\epsilon 
\end{equation}
 for all large $t>0$ and all $
c\in [0,p]\setminus\bigcup_{i=0}^m B_\epsilon(q_i)$,
where $B_\epsilon(q_i):=\{c\in\R: |c-q_i|<\epsilon\}$.
\end{lem}
\begin{proof} We first prove the inequality for $u_t(\xi_c(t), t)$.
Suppose the contrary. Then there exist $\epsilon>0$ small, $t_k\to+\infty$ and $\xi_k\to+\infty$ such that
\[
u(\xi_k, t_k)\in A_\epsilon:=[0,p]\setminus{\small \bigcup_{i=0}^m} B_\epsilon(q_i) \mbox{ for all } k\geq 1,
\]
and
\[\lim_{k\to\infty} u_t(\xi_k, t_k)=0.
\]
In view of Lemma \ref{w}, without loss of generality, we may assume that 
\[
\lim_{k\to\infty}u(r+\xi_k, t+t_k)=\tilde w(r,t) \mbox{ in } C_{loc}^{2,1}(\R^2).
\]
Then $\tilde w$ satisfies 
\[
\tilde w_t-\tilde w_{rr}=f(\tilde w), \; \tilde w_t\geq 0,\; \tilde w_r\leq 0 \mbox{ in } \R^2,
\]
and
\[
 \tilde w(0,0)\in A_\epsilon,\; \tilde w_t(0,0)=0.
\]
Using the maximum principle to $\tilde w_t$ we deduce $\tilde w_t\equiv 0$. Hence $\tilde w$ is a function of $r$ only. We claim that it is not a constant. 
 Otherwise we must have $\tilde w\equiv b$ for some $b\in \cup_{i=1}^m [q_i+\epsilon, b_i]$ (due to $\tilde w(0,0)\in A_\epsilon$). However, by Lemmas \ref{u-1} and \ref{xi_b'}, we have
\[
u_t(\xi_b(t),t)= -u_r(\xi_b(t),t)\xi_b'(t)\geq \delta\sigma>0 \mbox{ for all large } t.
\]
This implies that $u_t(r,t)\geq \frac12\delta\sigma$ whenever $|u(r,t)-b|$ is sufficiently small, since by standard parabolic regularity theory,  $u_t(r,t)$ is uniformly continuous in $(r,t)$. It follows that $\tilde w_t(0,0)>0$, a contradiction.
Hence $\tilde w(r,t)\equiv V(r)$, with $V(-\infty)=q_i>V(+\infty)=q_j$. This implies that $V(r_0)=b$ for some $r_0\in\R$.
Therefore
\[
U_k(t):=u(r_0+\xi_k, t_k+t)\to \tilde w(r_0,t)\equiv b \mbox{ in } C^1_{loc}(\R) \mbox{ as } k\to\infty,
\]
which implies $u_t(r_0+\xi_k, t_k+t)\to 0$ locally uniformly. On the other hand,  from $u_t(\xi_b(t),t)\geq \delta\sigma$ 
and the uniform continuity of $u_t(r,t)$, we have
 $u_t(r,t)>\frac12\delta\sigma$ whenever $|u(r,t)-b|$ is sufficiently small. This contradiction proves the required inequality for $u_t(\xi_c(t),t)$.
 
 Using $u_t(\xi_c(t),t)=-u_r(\xi_c(t),t)\xi_c'(t)$, the inequality for $\xi_c'(t)$ follows immediately from the boundedness of $|u_r(r,t)|$ and the inequality for $u_t(\xi_c(t),t)$. 
\end{proof}

\subsection{Properties of the  entire solution $w$ obtained in Lemma \ref{w}}

Let $w(r,t)$ be given by Lemma \ref{w} with $r_k=\xi_b(t_k)$ and $b\in\cup_{i=1}^m[q_i+\epsilon, b_i+\epsilon_i]$, for some sequence $t_k\to+\infty$.
Our ultimate goal is to show that $w$ is a traveling wave solution, which will be achieved in the next subsection through a careful examination of the set of all these entire solutions obtained by choosing different $b$. For this purpose,  we need to know enough properties of each entire solution in this set. In this subsection, we obtain these properties for each individual $w$ via a sequence of lemmas.

\begin{lem}
\label{alpha(t)-beta(t)}
Let $w$ be given as above.
Then $w_r(r,t)<0$ for $(r,t)\in\R^2$. Moreover, if we
define
\[
\alpha(t):=\lim_{r\to-\infty}w(r,t),\;\beta(t):=\lim_{r\to+\infty}w(r,t),
\]
then $\alpha(t)\equiv d$ is a stable zero of $f$ above $b$, and  $\beta(t)\equiv c$  is a stable zero of $f$ below $b$.
\end{lem}

\begin{proof} For definiteness we assume that $b\in[q_l+\epsilon, b_l+\epsilon_l]$. Then $w(0,0)=b$ and $w_r(0,0)\leq -\delta$ and hence $w_r(r,t)<0$ for all $(r,t)\in\R^2$.
It follows that $\alpha(t)>b>\beta(t)$ for all $t\in\R$.

For clarity we divide the rest of the proof into three steps.

{\bf Step 1.}  We show that  $\alpha(\R)= (b_{i}, q_{i-1})$ for some $i\leq l$ or $\alpha(t)\equiv d$ is a zero of $f$.
Similarly, $\beta(\R)=(b_{j}, q_{j-1})$ for some
$j\geq l+1$, or $\beta(t)\equiv c$ is  a  zero of $f$.

As in the proof of Lemma \ref{xi_b'} we know that $\alpha\in C^2(\R)$ and
\[
\alpha\rq{}(t)=f(\alpha(t)),\; \alpha\rq{}(t)\geq 0, b<\alpha(t)\leq q_0  \mbox{ for all } t\in\R.
\]
Hence if $\alpha(t)$ is not a constant then $\alpha\rq{}(t)>0$ and thus $f(\alpha(t))>0$ for all $t$. This implies that $\alpha(\R)= (b_{i}, q_{i-1})$ for some $i\in \{1,..., l\}$. If $\alpha(t)$ is a constant $d$ then necessarily $f(d)=0$.

Similarly $\beta(t)$ is either a constant which is a zero of $f$, or $\beta(\R)=(b_{j}, q_{j-1})$ for some $j\geq l+1$.

{\bf Step 2.} We show that $\alpha(t)$ and $\beta(t)$ are both constant functions.

We only consider $\alpha(t)$, as the proof for $\beta(t)$ is the same.
By Step 1, if $\alpha(t)$ is not a constant then $\alpha(\R)=(b_{i}, q_{i-1})$. By \eqref{2.5} there exist $\epsilon_0>0$ and $\delta>0$ such that
$u_r(r,t)\leq -\delta$ when $u(r,t)\in [q_{i}+\epsilon, b_{i}+\epsilon_{i}]$. It follows that
\begin{equation}
\label{w_r}
w_r(r,t)\leq -\delta \mbox{ when } w(r,t)\in [q_{i}+\epsilon, b_{i}+\epsilon_{i}].
\end{equation}
Choose $t_0$ so that $\alpha(t_0)\in (b_{i}, b_{i}+\epsilon_{i})$. Then there exists $M>0$ large so that $w(r,t_0)\in [b_{i}, b_{i}+\epsilon_{i}]$
for all $r\leq -M$. However this is impossible due to \eqref{w_r}. This proves that $\alpha(t)$ is a constant function.

{\bf Step 3.}
 We  show that
$\alpha(t)$ and $\beta(t)$ are  stable zeros of $f$.

Again we only consider $\alpha(t)$. Otherwise $\alpha(t)\equiv \tilde b$ for some $\tilde b\in [q_i+\epsilon, b_i] $ satisfying $f(\tilde b)=0$ and $i\leq l$.  Fix $t\in\R$. Since $\alpha(t)= \tilde b$ we see that for all large negative
$r$, \eqref{w_r} holds. But this is clearly impossible.

The proof is now complete.
\end{proof}

\begin{lem}\label{w1} Let $w(r,t)$ be as in Lemma \ref{alpha(t)-beta(t)}, and so
 $\lim_{r\to-\infty} w(r,t)$ and $\lim_{r\to+\infty}w(r,t)$ are  stable zeros of $f$, say
\begin{equation}
\label{qj-qi}
\lim_{r\to-\infty} w(r,t)=q_i,\; \lim_{r\to+\infty} w(r,t)=q_j,\; q_i>q_j.
\end{equation}
Then $w_t(r,t)>0$ in $\R^2$ and
\begin{equation}\label{t-pm-infty}
q_i=\lim_{t\to+\infty}w(r,t),\; q_j=\lim_{t\to-\infty} w(r,t).
\end{equation}
\end{lem}
\begin{proof}
By \eqref{xi_c} we have $w_t(0,0)>0$ and hence, by the strong maximum principle (applied to the equation satisfies by $w_t$) we deduce $w_t(r,t)>0$ for all $(r, t)\in\R^2$.
Define $w(r,\pm\infty):=\lim_{t\to\pm\infty} w(r,t)$. By the monotonicity of $w$ we easily see
\[
q_j\leq w(r,-\infty)< w(r,+\infty)\leq q_i.
\]
This and the uniform continuity of $w_t$ together with $w_t> 0$ imply that $w_t(r,\pm\infty):=\lim_{t\to\pm\infty}w_t(r,t)\equiv 0$. Using \eqref{xi_c} again we see that the range of the continuous functions $r\to w(r,\pm\infty)$ can only be single points which are stable zeros of $f$. From \eqref{qj-qi} we then easily see that $w(r,-\infty)\equiv q_j$ and $w(r,+\infty)\equiv q_i$.
 This completes the proof.
\end{proof}

Let us note that, if  $t_k\to\infty$ and
\[
w(r,t)=\lim_{k\to\infty}u(r+\xi_{b_n}(t_k), t+t_k),\; b_n=w(\zeta_{b_n}(t),t),
\]
then by \eqref{2.5} we have
\begin{equation}\label{zeta-xi}
\zeta_{b_n}(t)=\lim_{k\to\infty}\big[\xi_{b_n}(t_k+t)-\xi_{b_n}(t_k)\big].
\end{equation}

\begin{lem}
\label{w2}
Let $w(r,t)$ be as in  Lemma \ref{w1}, with $q_i$ and $q_j$ given in \eqref{qj-qi}. Then $q_i$ and $q_j$ are connected by a traveling wave solution, namely, there exist $c>0$ and $\Phi=\Phi(z)$ satisfying
\[
\Phi''+c\Phi'+f(\Phi)=0,\; \Phi'(z)<0 \mbox{ for } z\in\R,\; \Phi(-\infty)=q_i,\; \Phi(+\infty)=q_j.
\]
Moreover, 
\beq\label{3.19}
\sup_{t>0}|\zeta_{b_n}(t)-ct|<+\infty \mbox{ for } n\in\{i+1,..., j\}.
\eeq
\end{lem}
\begin{proof}
Multiplying the identity
\[
w_t-w_{rr}=f(w)
\]
by $w_r$ and then integrating for $r$ from some $r_0\in\R$ to $+\infty$, we obtain
\begin{equation}
\label{w-dr}
\int_{r_0}^{+\infty}w_tw_rdr+\frac12 w^2_r(r_0,t)=\int_{r_0}^{+\infty}f(w)w_rdr
=\int_{w(r_0,t)}^{q_j}f(u)du.
\end{equation}
Since $w_t>0>w_r$, we deduce
\[
\int_{q_j}^{w(r_0,t)}f(u)du=-\int_{r_0}^{+\infty}w_tw_rdr-\frac12 w_r^2(r_0,t)<-\int_{-\infty}^{+\infty}w_tw_rdr.
\]
On the other hand, letting $r_0\to-\infty$ in \eqref{w-dr} we obtain
\[
-\int_{-\infty}^{+\infty}w_tw_rdr=\int_{q_j}^{q_i}f(u)du.
\]
We thus have
\beq\label{w-integral}
\int_{q_j}^{q_i}f(u)du>0 \mbox{ and  }  \int_{q_j}^zf(u)du<\int_{q_j}^{q_i}f(u)du \mbox{ for any $z\in (q_j, q_i)$}.
\eeq
These properties imply, by Lemma A, that there exists a unique propagating terrace connecting $q_i$ to $q_j$. Since
each traveling wave in the propagating terrace is steeper than $w$, there can
 exist only one traveling wave in the propagating terrace; in other word, the propagating terrace consists of a single traveling wave $U(r,t)=\Phi(r-ct)$
connecting $q_i$ to $q_j$ with $c>0$ (recall that $\int_{q_j}^{q_i}f(u)du>0$). 

Next we choose positive numbers $\beta$ and $\sigma$ according to Lemma \ref{fm-sup}, so that for every $R\in\R$,
\[
\overline w(r,t):=\Phi(r-ct-R+e^{-\beta t})+\sigma e^{-\beta t}
\]
and
\[
\underline w(r,t):=\Phi(r-ct+R-e^{-\beta t})-\sigma e^{-\beta t}
\]
satisfy, respectively,
\[
\overline w_t-\overline w_{rr}\geq f(\overline w) \;\; (r\in\R,\; t>0)
\]
and
\[
\underline w_t-\underline w_{rr}\leq f(\underline w) \;\; (r\in\R,\; t>0).
\]
Due to \eqref{qj-qi}, we can choose $R>0$ large enough such that
\[
\overline w(r,0)>w(r,0)>\underline w(r,0) \;\forall r\in\R.
\]
Therefore we can apply the comparison principle to conclude that
\[
\overline w(r,t)>w(r,t)>\underline w(r,t) \;\forall r\in\R,\;\forall t>0.
\]
It follows that, for each $n\in\{i+1,..., j\}$ and $t>0$,
\[
\Phi(\zeta_{b_n}(t)-ct-R+e^{-\beta t})+\sigma e^{-\beta t}>b_n>\Phi(\zeta_{b_n}(t)-ct+R-e^{-\beta t})-\sigma e^{-\beta t}.
\]
This inequality clearly implies \eqref{3.19} for all large $t$, say $t\geq T_0$. The bound for $|\zeta_{b_n}(t)-ct|$ over $t\in [0, T_0]$ is a consequence of the continuity of $\zeta_{b_n}(t)$.
The proof is now complete.
\end{proof}

Note that by Corollary 5.5 of \cite{OM} (see also \cite{xChen}), the traveling wave profile function $\Phi(z)$ in Lemma \ref{w2} is unique up to a translation of $z$. 
We will show that $w(r,t)$ in Lemma \ref{w1} is a shift of $\Phi(r-ct)$, namely $w(r,t)\equiv\Phi(r-ct+r_0)$ for some $r_0\in\R$.
 In the following result, we prove this conclusion under an extra condition. We will see in the next subsection that
this extra condition is automatically satisfied by any $w(r,t)$ given in Lemma \ref{alpha(t)-beta(t)}.
\begin{lem}
\label{w=tw}
Let $w(r,t)$ be as in Lemma \ref{w1}. 
Moreover, when $j>i+1$, we assume further that the functions $\zeta_{b_n}(t)$, $n=i+1,..., j$, determined uniquely by
\[
w(\zeta_{b_n}(t), t)=b_n,
\]
have the property that
\begin{equation}
\label{gap-zeta}
\zeta_{b_j}(t)-\zeta_{b_{i+1}}(t)\leq C \mbox{ for all } t\in\R,
\end{equation}
where $C$ is some positive constant.
Then  $w(r,t)=\Phi(r-ct+r_0)$ for some $r_0\in\R$.
\end{lem}

\begin{proof}  We use three steps. For any $a\in (q_j, q_i)$, we define $\zeta_a(t)$ by
\[
a=w(\zeta_a(t), t).
\]

{\bf Step 1.} 
We show that for any $a\in (q_j, b_j)$, the function $\zeta_a(t)-\zeta_{b_{j}}(t)$ belongs to $L^\infty(\R)$. Similarly for any $\tilde a\in (b_{i+1}, q_i)$, the function $\zeta_{b_{i+1}}(t)-\zeta_{\tilde a}(t)$ belongs to $L^\infty(\R)$.

We only prove the conclusion for $\zeta_a(t)-\zeta_{b_{j}}(t)$; the proof for $\zeta_{b_{i+1}}(t)-\zeta_{\tilde a}(t)$ is analogous.
If the conclusion is not true, then there exists a sequence $\{s_k\}$ such that
\[
\zeta_a(s_k)-\zeta_{b_{j}}(s_k)\to+\infty.
\]
As before by passing to a subsequence we may assume that
\[
w(r+\zeta_{b_{j}}(s_k), t+s_k)\to w^*(r,t) \mbox{ in } C_{loc}^{2,1}(\R).
\]
We observe that the argument used in the proof of Lemma \ref{u-1} can be applied to show that $w_r(\zeta_{b_j}(t),t)\leq -\delta$ for some $\delta>0$ and all $t\in\R$. It follows that $w^*_r(0,0)\leq -\delta$. On the other hand, similar to before, $w^*$ satisfies
\[
w^*_t-w^*_{rr}=f(w^*),\; w^*_r\leq 0,\; w^*_t\geq 0 \mbox{ in } \R^2,
\]
and $w^*(0,0)=b_j$. In view of $w^*_r(0,0)<0$, by the strong maximum principle we have $w^*_r<0$ in $\R^2$.

For any fixed $r\in\R$, our assumption implies
\[
r+\zeta_{b_j}(s_k)\leq \zeta_a(s_k)
\mbox{ for all large } k.
\]
It follows that
\[
w(r+\zeta_{b_j}(s_k), s_k)\geq w(\zeta_a(s_k), s_k)=a \mbox{ for all large } k,
\]
and hence
$w^*(r,0)\geq a$ for all $r\in\R$. Therefore
\begin{equation}
\label{w*}
\lim_{r\to+\infty} w^*(r,0)\geq a>q_j.
\end{equation}

On the other hand, if we define 
\[
\beta(t)=\lim_{r\to+\infty}w^*(r,t),
\]
then a simple regularity consideration indicates that $\beta(t)$ satisfies
\[
\beta'(t)=f(\beta(t)) \mbox{ for } t\in\R.
\]
Since $w^*_t\geq 0$ we have $\beta'(t)\geq 0$. Moreover, $\beta(0)<w^*(0,0)=b_j$. Since $w(r,t)> q_j$, we have $w^*(r,t)\geq q_j$ and hence $\beta(t)\geq q_j$. If $\beta(t)\not\equiv q_j$ then $\beta'(t)>0$ and there exists $t_0\in R^1$ satisfying $\beta(t_0)\in (q_j, b_j)$ and hence $\beta'(t_0)=f(\beta(t_0))\leq 0$, a contradiction to $\beta'(t)> 0$. Thus we must have $\beta(t)\equiv q_j$,
which implies $\lim_{r\to+\infty} w^*(r,0)=q_j$, contradicting \eqref{w*}.
This  proves our claim in Step 1.

\medskip

{\bf Step 2.} We show that there exists $M^*>0$ such that
\[
\Phi(r-ct+M^*)\leq  w(r,t)\leq \Phi(r-ct-M^*) \mbox{ in } \R^2.
\]

 By \eqref{gap-zeta} and the conclusions in Step 1, for each pair $a$ and $\tilde a$ satisfying $a\in (q_j, b_j)$, $\tilde a\in (b_{i+1}, q_i)$, there exists $M=M(a,\tilde a)>0$ such that
\[
\zeta_{\tilde a}(t)-\zeta_a(t)\leq M \mbox{ for all } t\in\R.
\]
Applying Lemma \ref{fm-sup} with $W(r,t)=\Phi(r-ct)$, we can find positive constants $\sigma$ and $\beta$ so that, for every $R>0$,
\[
U^*(r,t):=\Phi(r-ct-R+(1+c)e^{-\beta t})+\sigma e^{-\beta t}
\]
and
\[
U_*(r,t):=\Phi(r-ct+R-(1+c)e^{-\beta t})-\sigma e^{-\beta t}
\]
satisfy
\[
U^*_t-U^*_{rr}\geq f(U^*),\; (U_*)_t-(U_*)_{rr}\leq f(U_*)  \mbox{ for $r\in\R$ and $t>0$}.
\]
We then take
\[
a=q_j+\sigma,\; \tilde a=q_i-\sigma \mbox{ and } M=M(a,\tilde a),
\]
and  choose $R>0$ large enough such that
\[
\Phi(M-R+1+c)+\sigma\geq q_i,\; \Phi(-M+R-1-c)-\sigma\leq q_j.
\]
It follows that
\[
U^*(r,0)=\Phi(r-R+1+c)+\sigma\geq q_i \mbox{ for } r\leq M
\]
and
\[ U_*(r,0)=\Phi(r+R-1-c)-\sigma\leq q_j \mbox{ for } r\geq -M.
\]
Fix $b\in\{b_{i+1},..., b_j\}$ and $s\in\R$, and consider
\[
w^s(r,t):=w(r+\zeta_b(s), t+s).
\]
We have
\begin{align*}
w^s(r,0)&=w(r+\zeta_b(s), s)\\
&\leq w(r+\zeta_a(s)-M, s)\\
&\leq w(\zeta_a(s),s)=a\\
&=q_j+\sigma\leq U^*(r,0) \mbox{ for } r\geq M.
\end{align*}
For $r\leq M$,
\[
U^*(r,0)\geq q_i>w^s(r,0).
\]
Hence
\[
w^s(r,0)\leq U^*(r,0) \mbox{ for all } r\in\R.
\]
Similarly we can show that
\[
w^s(r,0)\geq U_*(r,0) \mbox{ for all } r\in \R.
\]
Therefore we can apply the comparison principle to deduce that
\[
U_*(r,t)\leq w^s(r,t)\leq U^*(r,t) \mbox{ for } r\in\R,\; t\geq 0,\; s\in\R,
\]
that is, for $r\in\R,\; t\geq 0,\; s\in\R,$
\begin{equation}\label{U*}
\begin{aligned}
\Phi(r-ct+R-(1+c)e^{-\beta t})-\sigma e^{-\beta t}&\leq w(r+\zeta_b(s), t+s)\\
&\leq \Phi(r-ct-R+(1+c)e^{-\beta t})+\sigma e^{-\beta t}.
\end{aligned}
\end{equation}
It follows that
\[
-M_1\leq \zeta_b(t+s)-\zeta_b(s)-ct\leq M_1
\]
for some $M_1>0$, all $ s\in\R$ and all large $t$, say $t\geq T_0$. Without loss of generality we assume that $\zeta_b(0)=0$. Taking $s=0$ in the above inequalities we obtain
\[
-M_1\leq \zeta_b(t)-ct\leq M_1 \mbox{ for } t\geq T_0.
\]
Taking $s<-T_0$ and $t=-s$ we obtain
\[
-M_1\leq cs-\zeta_b(s)\leq M_1 \mbox{ for } s<-T_0.
\]
Hence, by enlarging $M_1$ so that $M_1\geq \max_{|t|\leq T_0}|\zeta_b(t)-ct|$, we obtain
\[
-M_1\leq \zeta_b(t)-ct\leq M_1 \mbox{ for all } t\in\R.
\]
Denote $\tilde r=r+\zeta_b(s)$ and $\tilde t=t+s$; then \eqref{U*} can be rewritten in the form
\begin{align*}
&\Phi(\tilde r-c\tilde t-[\zeta_b(s)-cs]+R-(1+c)e^{-\beta (\tilde t-s)})-\sigma e^{-\beta (\tilde t-s)}\\
&\leq w(\tilde r, \tilde t)\leq \Phi(\tilde r-c\tilde t-[\zeta_b(s)-cs]-R+(1+c)e^{-\beta (\tilde t-s)})+\sigma e^{-\beta (\tilde t-s)}
\end{align*}
for $\tilde r\in\R,\; \tilde t>s+T_0$ and $s\in\R$.

Since $-M_1\leq \zeta_b(s)-cs\leq M_1$, it follows that
\begin{align*}
&\Phi(\tilde r-c\tilde t+M_1+R-(1+c)e^{-\beta (\tilde t-s)})-\sigma e^{-\beta (\tilde t-s)}\\
&\leq w(\tilde r, \tilde t)\leq\Phi(\tilde x-c\tilde t-M_1-R+(1+c)e^{-\beta (\tilde t-s)})+\sigma e^{-\beta (\tilde t-s)}
\end{align*}
for  $\tilde r\in\R,\; \tilde t>s+T_0$ and $s\in\R$. Letting $s\to-\infty$ we deduce
\[
\Phi(\tilde r-c\tilde t+M^*)\leq w(\tilde r,\tilde t)\leq \Phi(\tilde r-c\tilde t-M^*)
\]
for $(\tilde r, \tilde t)\in\R^2$ and $M^*=M_1+R$. 
This completes the proof of Step 2.

\medskip

{\bf Step 3.} 
We show that there exists $r_0\in\R$ such  that
\[
w(r,t)\equiv \Phi(r-ct+r_0) \mbox{ in } \R^2.
\]

This follows directly from the conclusion proved in Step 2 above and Theorem 3.1 of \cite{BH}. (Alternatively, this can also be proved by using results of 
 \cite{OM}.)
 
The proof of  the lemma is now complete.
\end{proof}

\begin{rmk}\label{3.12}
If $w$ is a traveling wave connecting $q_j$ and $q_i$ with speed $\tilde c$, namely, 
\[
\mbox{$w(r,t)=\tilde\Phi(r-\tilde ct)$
with $\tilde\Phi(+\infty)=q_j$ and $\tilde \Phi(-\infty)=q_i$,}
\]
 then clearly the conditions of Lemma \ref{w=tw} are satisfied. It then follows that 
\[
\tilde \Phi(r-\tilde c t)\equiv \Phi(r-ct+r_0) \mbox{ for some } r_0\in\R.
\]
This clearly implies $\tilde c=c$. Therefore the traveling waves of \eqref{1D} connecting $q_j$ to $q_i$ is unique subject to a shift of time whenever they exist.
\end{rmk}

\subsection{Convergence to a traveling wave }

We now come back to $u(r,t)$ and $\xi_b(t)$ as given in subsection 2.3.
We are going to make use of the uniqueness of the propagating terrace $\big\{Q_k, U_k, c_k\}_{1\leq k\leq n_0}$, and the properties of monotone entire 
 solutions obtained in the previous 
subsection, to prove the following result:

\begin{prop}
\label{entire-2}
Fix $s\in\{1,..., n_0\}$. For each $i$ satisfying $i_{s-1}<i<i_s$, there exists $r_i^0\in\R$ such that
\[
\lim_{\tilde t\to\infty} u(r+\xi_{b_i}(\tilde t), t+\tilde t)=U_{s}(r-c_s t+r_i^0) \mbox{ locally uniformly for } (r,t)\in\R^2.
\]
\end{prop}

Take a sequence $\{t_k\}$ satisfying $t_k\to+\infty$ and consider, for each $l\in\{1,..., m\}$, the sequence of functions
 \[
\{u(r+\xi_{b_l}(t_k), t+t_k)\}.
\]
 By Lemmas \ref{xi_b'} and \ref{w}, we see that subject to a subsequence 
\[
u(r+\xi_{b_l}(t_k), t+t_k)\to w^{b_l}(r,t) \mbox{ in } C^{2,1}_{loc}(\R^2),
\]
and $w^{b_l}$ is an entire solution of $w_t-w_{rr}=f(w)$ satisfying
 \[
\mbox{ $w^{b_l}_t\geq 0$, $w^{b_l}_r\leq 0$ and $w^{b_l}(0,0)=b_l$.}
\]
By passing to a further subsequence we may assume that the convergence above holds for every $l\in\{1,..., m\}$.

\begin{lem}
\label{entire-1}
For each $w^{b_l}(r,t)$ given above, the following conclusions hold:
\begin{itemize}
\item[(i)] $
w^{b_l}_r(r,t)<0$, $w^{b_l}_t(r,t)>0$ for $(r,t)\in\R^2$.
\item[(ii)] There exist $i\leq l-1$ and $j\geq l$ such that
\[
q_i=\lim_{r\to-\infty} w^{b_l}(r,t)=\sup_{\R^2} w^{b_l},
\;\;
q_j=\lim_{r\to+\infty} w^{b_l}(r,t)=\inf_{\R^2} w^{b_l}.
\]
\end{itemize}
If we call $w^{b_l}$ with the above properties a {\bf monotone entire solution connecting $q_j$ to $q_i$}, then 
 from the set $\{w^{b_l}: l=1,..., m\}$ we can find a subset of $m'\;(m'\geq 1)$ such entire solutions
\[
w^{b_{l_1}},..., w^{b_{l_{m'}}},
\]
 and stable zeros of $f:$ 
\[
q_0=q_{j_0}>q_{j_1}>q_{j_2}>...>q_{j_{m'}}=q_m=0,
\]
 such that
 \begin{align*}
 &
 \mbox{
$w^{b_{l_1}}$ connects  $q_{j_1}$ to $q_{j_0}$,}
\\
&\mbox{
 $w^{b_{l_2}}$ connects $q_{j_2}$ to $q_{j_1}$,}\\
 &
 \;\;\;\;\;\;\;\;\;\; ... ...\\
&
 \mbox{ $w^{b_{l_{m'}}}$ connects  $q_{j_{m'}}$ to $q_{j_{m'-1}}$.}
 \end{align*}
\end{lem}

\begin{proof}
By passing to a suitable subsequence of $\{t_k\}$ we may assume that
\[
\xi_{b_{n+1}}(t_k)-\xi_{b_n}(t_k)\to\eta_n \mbox{ with } \eta_n\in [0,+\infty] \mbox{ for } n=1,..., m-1.
\]
Note that due to \eqref{mono-r} we have $\xi_{b_{n+1}}(t_k)>\xi_{b_n}(t_k)$ which implies $\eta_n\geq 0$. (We can actually show $\eta_n>0$ by using $u_r(\xi_{b_n}(t),t)\leq -\delta$, though this is not needed here.)

Note that each $w^{b_n}(r,t)$ is an entire solution, and
\[
 w^{b_n}(0,0)=b_n,\; w^{b_n}_r(0,0)\leq -\delta<0,\; w^{b_n}_r<0,\; w^{b_n}_t\geq 0,
\]
where $w^{b_n}_r(0,0)\leq -\delta$ follows from $u_r(\xi_{b_n}(t),t)\leq -\delta$, which inturn yields $w^{b_n}_r<0$.
It is easily seen that
\[
\mbox{ $\eta_n<\infty$ implies  $w^{b_n}(r,t)=w^{b_{n+1}}(r-\eta_n, t)$ in $\R^2$. }
\]
We next consider the case $\eta_n=\infty$. In such a case  for any fixed $(r,t)\in\R^2$ and $y\in\R$, we have
\[
u(r+\xi_{b_n}(t_k), t+t_k)>u(y+\xi_{b_{n+1}}(t_k), t+t_k) \mbox{ for all large } k.
\]
It follows that
\[
w^{b_n}(r,t)\geq w^{b_{n+1}}(y,t).
\]
Hence
\[
\beta^{b_n}(t):=\lim_{r\to+\infty} w^{b_n}(r,t)\geq\lim_{y\to-\infty}w^{b_{n+1}}(y,t)=: \alpha^{b_{n+1}}(t) \mbox{ for all } t\in\R.
\]
Since $w^{b_n}$ and $w^{b_{n+1}}$ satisfy the conditions of Lemma \ref{alpha(t)-beta(t)}, we have $\alpha^{b_{n+1}}(t)\equiv q_i$ for some $i\leq n$, and $\beta^{b_n}(t)\equiv c<b_n$ with $c$ a zero of $f$.
We may now apply $\beta^{b_n}(t)\geq \alpha^{b_{n+1}}(t)$ to deduce
\[
b_n>c\geq q_i\geq q_n,
\]
which implies $i=n$ and $c=q_n$. Therefore  $\beta^{b_n}(t)\equiv q_n\equiv \alpha^{b_{n+1}}(t)$. In other words,
\[
\mbox{ $\eta_n=\infty$ implies }
q_n=\lim_{r\to+\infty}w^{b_n}(r,t)=\lim_{r\to-\infty}w^{b_{n+1}}(r,t).
\]
Let us also observe that when $n=m$,  necessarily  $\beta^{b_m}(t)\equiv c=q_m=0$.

We thus have the following conclusions:
\begin{itemize}
\item[(a)]
 For each $n\in \{1,...,m\}$, $\sup_{\R}w^{b_n}$ and $\inf_{\R^2}w^{b_n}$ are stable zeros of $f$, and
\[
\lim_{r\to-\infty}w^{b_n}(r,t)=\sup_{\R^2}w^{b_n}>b_n>\inf_{\R^2}w^{b_n}=\lim_{r\to+\infty}w^{b_n}(r,t).
\]
\item[(b)] $\eta_n<\infty$ implies   $w^{b_n}(r,t)=w^{b_{n+1}}(r-\eta_n, t)$ in $\R^2$ and hence
\[
\inf_{\R^2}w^{b_n}=\inf_{\R^n}w^{b_{n+1}} \mbox{ and } \inf_{\R^2}w^{b_n}=\inf_{\R^n}w^{b_{n+1}}.
\]
\item[(c)] $\eta_n=\infty$ implies
\[
\inf_{\R^2}w^{b_n}=\sup_{\R^2}w^{b_{n+1}}=q_n.
\]
\end{itemize}

In particular, for each $l\in\{1,..., m\}$, there exist $i\leq l-1$ and $j\geq l$ such that
\begin{equation}
\label{w-l}
q_i=\lim_{r\to-\infty}w^{b_l}(r,t)=\sup_{\R^2}w^{b_l},\;
q_j=\lim_{r\to+\infty}w^{b_l}(r,t)=\inf_{\R^2}w^{b_l}.
\end{equation}
The remaining conclusions of the lemma clearly follow directly from (a)-(c) above.
\end{proof}

By Lemma \ref{w2}, each monotone entire solution $w^{b_{l_k}}(r,t)$, $k=1,..., m'$, corresponds to a traveling wave solution $\tilde U_{k}(r-\tilde c_{k}t)$ connecting $q_{j_{k}}$ to $q_{j_{k-1}}$, with speed $\tilde c_{k}>0$. 
 Let
\[
B:=\{j_k: 0\leq k\leq m'\}, \; \mathcal{B}=\big\{q_{j_k}, \tilde U_k, \tilde c_k\big\}_{1\leq k\leq m'}.
\]
If $\tilde c_1\leq \tilde c_2\leq... \leq \tilde c_{m'}$, then $\mathcal{B}$ is a propagating terrace connecting $p$ to $0$, and by uniqueness, it must coincide with the one given in \eqref{terrace-q-U}. In the following, instead of examining the order of the $\tilde c_i$'s, we show that $\{q_{j_k}: 0\leq k\leq m'\}$ coincides with the set of platforms of the propagating terrace given in \eqref{terrace-q-U}, namely
\[
\{q_{j_k}: 0\leq k\leq m'\}=\{Q_k: 0\leq k\leq n_0\}.
\]
By the uniqueness of traveling waves connecting adjacent platforms, this also implies that $\mathcal{B}$ is the propagating terrace connecting $p$ to $0$, as the following result concludes.

\begin{lem}
\label{terrace-1} $\mathcal{B}$ 
is the unique  propagating terrace of \eqref{1D} connecting $p$ to $0$.
\end{lem}

The proof of this lemma relies on the following two lemmas.

\begin{lem}
\label{lem:two-TW-1}
Let $q^*>q_*$ be linearly stable zeros of $f$ and assume that there exists a traveling wave $V$ with speed $c$ connecting $q_*$ to $q^*$. Let $q$ be a linearly stable zero of $f$ satisfying $q^*>q>q_*$.
\begin{itemize}
\item[(i)] If there exists a traveling wave $V_1$ with speed $c_1$ connecting $q$ to $q^*$, then $c_1>c$.
\item[(ii)] If there exists a traveling wave $V_2$ with speed $c_2$ connecting $q_*$ to $q$, then $c_2<c$.
\end{itemize}
\end{lem}
\begin{proof}
We follow the lines of the proof of Lemma 2.2 in \cite{FM}. We only give the detailed proof for part (i), as the proof for part (ii) is similar.

By assumption, $V=V(\xi)$ satisfies
\[
V''+cV'+f(V)=0, \; V'<0 \;\;\forall \xi\in\R,
\]
and
\[
V(-\infty)=q^*,\;\; V(+\infty)=q_*.
\]
Set $W(\xi)=V'(\xi)$. Then $W$ satisfies
\[
W'+cW+f(V)=0,\; W<0  \mbox{ for } \xi\in \R.
\]
For $v\in(q_*, q^*)$, there exists a unique $\xi\in\R$ such that $v=V(\xi)$.
We define $P: (q_*, q^*)\mapsto \R$ by
\[
P(v)=W(\xi), \mbox{ and so } P(V(\xi))=W(\xi).
\]
It follows that
\[
W\frac{dP}{dv}=W'(\xi)=-cW(\xi)-f(V(\xi))=-cW-f(v).
\]
Hence
\[
\frac{dP}{dv}=-c-\frac{f(v)}{P} \mbox{ for } v\in (q_*, q^*).
\]
Moreover,
\[
P(q_*)=P(q^*)=0,\;\; P(v)<0 \mbox{ for } v\in (q_*, q^*).
\]

Similarly we define $P_1: (q, q^*)\mapsto \R$ by
\[
P_1(v)=V'_1(\xi) \mbox{ with } v=V_1(\xi),
\]
and find that
\[
\frac{dP_1}{dv}=-c_1-\frac{f(v)}{P_1} \mbox{ for } v\in (q, q^*),
\]
\[
P_1(q)=P_1(q^*)=0,\; P_1(v)<0 \mbox{ for } v\in (q, q^*).
\]

Suppose $c_1\leq c$; we are going to derive a contradiction. Clearly we have
\[
(P-P_1)'=(c_1-c)+\frac{f(v)}{P_1P}(P-P_1) \mbox{ for } v\in (q, q^*).
\]
Fix $\overline q\in (q, q^*)$ and define
\[
F(v):=[P(v)-P_1(v)]e^{\int_{\overline q}^v \frac{-f(s)}{P_1(s)P(s)}ds}.
\]
We note that since $q^*$ is a stable zero of $f$, $f(s)>0$ for $s<q^*$ but close to $q^*$. Therefore for such $s$,
\[
\frac{-f(s)}{P_1(s)P(s)}<0,
\]
which ensures that 
\[
\lim_{v\nearrow q^*}e^{\int_{\overline q}^v \frac{-f(s)}{P_1(s)P(s)}ds } \mbox{ exists and is finite}.
\]
It follows that 
\[
F(q^*):=\lim_{v\nearrow q^*}F(v)=0.
\]
On the other hand, it is easily calculated that
\[
F'(v)=(c_1-c)e^{\int_{\overline q}^v \frac{-f(s)}{P_1(s)P(s)}ds}.
\]
Hence from $c_1\leq c$ we obtain $F'(v)\leq 0$ for $v\in (q, q^*)$, which together with $F(q^*)=0$ implies
$F(v)\geq F(q^*)= 0$ for $v\in (q, q^*)$. It follows that $P(v)\geq P_1(v)$ for such $v$, which implies
$P(q)\geq P_1(q)=0$, a contradiction to $P(v)<0$ for $v\in (q_*, q^*)$. The proof is complete.
\end{proof}

Let us recall that  $Q_k=q_{i_k}$ $(k=0,..., n_0)$ are the platforms of the unique propagating terrace of \eqref{1D} connecting $0$ to $p$, as given in \eqref{terrace-q-U}. Define
\[
A:=\{i_1,...,i_{n_0}\},
\]
 and for $i\in \{1,..., m-1\}$ define
\[
\rho_i(t):=\xi_{b_{i+1}}(t)-\xi_{b_i}(t).
\]

\begin{lem} \label{rho-dichotomy}
The following dichotomy holds:
\begin{itemize}
\item[(i)] If $i\in A$, then $\lim_{t\to+\infty} \rho_i(t)=+\infty$.
\item[(ii)] If $i\in \{1,..., m-1\}\setminus A$, then $\rho_i(t)$ remains bounded as $t\to+\infty$.
\end{itemize}
\end{lem}

\begin{proof}

{Part (i)}.
This part is easy.  Suppose that $\rho_i(t_k)$ remains bounded 
for some sequence $t_k \to\infty$. By replacing $\{t_k\}$ by its 
subsequence if necessary, we may use Lemma \ref{w} to conclude that the following limits exist:
\[
\begin{split}
 & w(r,t) := \lim_{k\to\infty} u(r + \xi_{b_i}(t_k), t + t_k),\\
 & \rho_0 := \lim_{k\to\infty} \rho_i(t_k).
\end{split}
\]
Moreover $w$ is an entire solution of \eqref{1D} and satisfies 
$0 \leq  w  \leq  p$, along with
\[
  w(0,0) = b_i, \quad   w(\rho_{0}, 0) = b_{i+1}.
\]
This implies that the graph of $w(x,0)$ crosses the level $q_{i}$. 

By the proof of Lemma A,  the propagating terrace given in \eqref{terrace-q-U} is also the minimal propagating terrace, and hence it is steeper than any entire solution lying between $0$ and $p$. Thus the fact that $w(x,0)$ crosses the level $q_i$ implies   $i\not\in A$. This completes the proof of part (i).

\smallskip

{Part (ii)}.
Suppose that the conclusion of (ii) does not hold. Then, for some 
$s \in \{1, ..., n\}$ and some $i$ with $i_{s-1} < i < i_{s}$, 
\[
\limsup_{t\to\infty}\rho_i(t)=+\infty.
\]
Fix such an $s$  and let $B_*$ denote the set of all such $i$, and we label the elements in $B_*$ by
\[
i^*_{1}<i^{*}_2<...<i^{*}_r.
\]
Set $C_*:=\{i: \;i_{s-1}<i<i_s,\; i\not\in B_*\}.$ Then clearly
\beq\label{C*}  
 \sup_{t\geq T} \rho_i(t) < \infty, \;\;\forall i\in C_*,
\eeq
where $T>0$ is chosen such that $\xi_{b_i}(t)$ is defined for all $t\geq T$ and every $i=1,..., m$.

From
\[
\rho'_i(t)=\xi_{b_{i+1}}'(t)-\xi_{b_i}'(t)=\frac{u_t(\xi_{b_i}(t),t)}{u_r(\xi_{b_i}(t),t)}-\frac{u_t(\xi_{b_{i+1}}(t),t)}{u_r(\xi_{b_{i+1}}(t),t)},
\]
 we easily see that $\sup_{t\geq T}|\rho_i'(t)|<+\infty$. Hence we can apply Lemma \ref{tilde tk} to conclude that, for each $i^{*}_n$  in $B_*$, $1\leq n\leq r$, there exists a sequence $t_k^n\to\infty$ as $k\to\infty$ such that
\beq\label{B*}
\lim_{k\to\infty}\rho_{i^{*}_n}(t_k^n)=+\infty,\; \rho_{i^*_n}(t+t_k^n)\geq \rho_{i^*_n}(t_k^n) \;\;\forall t\in [0,k].
\eeq
Moreover, replacing $\{t_k^n\}_{k=1}^\infty$ by its subsequence if necessary, we can apply Lemma \ref{entire-1} to conclude that
 the following limits exist for every $r,t\in\R$:
\[
   w_n(r,t) := \lim u(r + \xi_{b_{i^*_n}}(t_k^n), t + t^n_k),\;\;  \hat w_n(r,t) := \lim u(r + \xi_{b_{i^*_n+1}}(t_k^n), t + t^n_k),
\]
and $w_n$, $\hat w_n$ are  monotone  entire solutions of \eqref{1D}, each connecting a pair of stable zeros of $f$. 
By \eqref{B*} and \eqref{zeta-xi}, we have
\beq\label{zeta-hat zeta}
\zeta_{b_{i_n^*}}(t)\leq \hat \zeta_{b_{i_n^*+1}}(t)\;\;\forall t\in\R,
\eeq
where $\zeta_a(t)$ and $\hat \zeta_a(t)$ are  given by, respectively,  $a=w_n(\zeta_a(t),t)$ and $a=\hat w_n(\hat\zeta_a(t),t)$. 

Using \eqref{B*} and \eqref{C*} we further see that
\[
 q_n^*:=\sup w_n>q_{i_n^*}=\inf w_n,\; \hat q_n^*:=\inf \hat w_n<q_{i_n^*}=\sup \hat w_n,
\]
and 
\[
\mbox{ $ q_n^*$ is a linearly stable zero of $f$ satisfying }  q_{i_{s-1}}\geq  q_n^*\geq q_{i^*_{n-1}},
\]
\[
\mbox{ $\hat q_n^*$ is a linearly stable zero of $f$ satisfying }  q_{i_{s}}\leq \hat q_n^*\leq q_{i^*_{n+1}}.
\]
Here and in what follows, we understand that 
\[
q_{i_0^*}=q_{i_{s-1}}, \; q_{i^*_{r+1}}=q_{i_s}.
\]
Let us note that necessarily
\beq\label{n=1,r}
q_1^*=q_{i_{s-1}},\; q^*_r=q_{i_s}.
\eeq

By Lemma \ref{w2} we also know that corresponding to $w_n$ there is a traveling wave solution $V_{n}^*$ of \eqref{1D}
connecting $q_{i_n^*}$ to $q_{n}^*$ with speed $c_{n}^*=\lim_{t\to\infty}\zeta_{b_{i_{n}^*}}(t)/t$, 
and corresponding to $\hat w_n$ there is a traveling wave solution $\hat V_{n}^*$ of \eqref{1D}
connecting $\hat q_n^*$ to $q_{i_{n}^*}$ with speed $\hat c_{n}^*=\lim_{t\to\infty}\hat \zeta_{b_{i_{n}^*+1}}(t)/t$.
By \eqref{zeta-hat zeta}, we obtain
\beq\label{hat c-c}
\hat c_{n}^*\geq c_{n}^*.
\eeq

\noindent
{\bf Claim:} $c_1^*\leq c_2^*\leq ... \leq c_r^*.$

If $r=1$ then there is nothing to prove. So suppose $r\geq 2$. Fix $n\in\{1,..., r-1\}$ and consider $\hat V_n^*$ and $V^*_{n+1}$. We have
\[
\hat V_n^*(+\infty)=\hat q_{n}^*\leq q_{i^*_{n+1}}=V^*_{n+1}(+\infty)<q_{i^*_{n}}=\hat V^*_{n}(-\infty)\leq  
q^*_{n+1}=V^*_{n+1}(-\infty).
\]
By Lemma \ref{fm-sup}, we can find $\beta>0, \sigma>0$ and $t_0\in \R$ such that
\[
V^*_{n+1}(r-c^*_{n+1}t)\geq \hat V^*_{n}(r-\hat c^*_{n}t+t_0-[1+\hat c^*_{n}]e^{-\beta t})-\sigma e^{-\beta t} 
\]
for all $t\geq 0$ and $r\in\R$. If $\hat c^*_{n}>c^*_{n+1}$, then we take $c\in (c^*_{n+1}, \hat c^*_{n})$, $r=ct$ and obtain
\[
V^*_{n+1}([c-c^*_{n+1}]t)\geq \hat V^*_{n}([c-\hat c^*_{n}]t+t_0-[1+\hat c^*_{n}]e^{-\beta t})-\sigma e^{-\beta t}
\]
for all $t>0$. Letting $t\to+\infty$, we arrive at
\[
V^*_{n+1}(+\infty)\geq \hat V^*_{n}(-\infty),
\]
a contradiction. Therefore we must have
$\hat c^*_{n}\leq c^*_{n+1}$. We may then apply \eqref{hat c-c} to obtain $c^*_{n+1}\geq c^*_{n}$. This proves the Claim.
 
\smallskip

We may now reach a contradiction by making use of Lemma \ref{lem:two-TW-1}. Indeed, in view of \eqref{n=1,r},
this lemma infers that
\[
c_1^*>c_s>\hat c_r^*.
\]
By \eqref{hat c-c} we have $\hat c_r^*\geq c_r^*$, and hence we obtain $c_1^*>c_r^*$, a contradiction
to the inequalities in the Claim.
This completes the proof of the lemma.
\hfill $\Box$

\smallskip

\noindent
{\bf Proof of Lemma \ref{terrace-1}:}\  By Lemma \ref{rho-dichotomy}, we necessarily have $m'=n_0$
and $\{q_{j_k}: k=1,..., m'\}=\{q_i: i\in A\}$. By the uniqueness of traveling wave solutions (subject to time shifts) we further obtain that
$\{\tilde U_k: k=1,..., m'\}=\{U_k: k=1,..., n_0\}$, where traveling waves are identified if they connect the same pair of stable zeros of $f$. 
\end{proof}

We are now ready to use Lemmas \ref{rho-dichotomy} and \ref{entire-1} to prove Proposition \ref{entire-2}.

\begin{proof}[{\bf Proof of Proposition \ref{entire-2}}]
Let $t_k\to\infty$ be an arbitrary sequence of large positive numbers. By Lemmas \ref{entire-1} and \ref{rho-dichotomy}, we may pass to a subsequence and obtain, for each $s\in\{1,..., n_0\}$,
\[
w^{i_s}(r,t)=\lim_{k\to\infty} u(r+\xi_{b_{i_s}}(t_k), t+t_k) \;\;\forall (r,t)\in\R^2
\]
with $w^{i_s}(r,t)$ a monotone entire solution of \eqref{1D} connecting $q_{i_s}$ to $q_{i_{s-1}}$.
Since 
\[
\sup_{t\geq T} \rho_i(t)<+\infty \mbox{ for }  i_{s}<i<i_{s-1},
\]
we further obtain
\[
\zeta_{b_{i_{s+1}-1}}(\cdot)-\zeta_{b_{i_{s-1}+1}}(\cdot)\in L^\infty,
\]
where
$
\zeta_a(t)$ is determined by
\[
a=w^{i_s}(\zeta_a(t),t).
\]
Hence we can use Lemma \ref{w=tw} to conclude that $w^{i_s}$ is a traveling wave. By uniqueness we necessarily have
$w^{i_s}(r,t)=U_s(r-c_st+r_s^0)$, with  $r_s^0\in \R$ uniquely determined by 
\[
b_{i_s}=U_s(r_s^0).
\]
As $\lim_{k\to\infty} u(r+\xi_{b_{i_s}}(t_k), t+t_k)$ is uniquely determined, and $\{t_k\}$ is a subsequence of an arbitrary sequence converging to $+\infty$, we see that necessarily
\[
U_s(r-c_s t+r_s^0)=\lim_{\tilde t\to+\infty} u(r+\xi_{b_{i_s}}(\tilde t), t+\tilde t)
\]
locally uniformly in $(r,t)\in\R^2$.
\end{proof}

\subsection{Completion of the proof of Theorem \ref{thm-existence-rt}} To complete the proof that $u=u(r,t)$ is a radial terrace, 
we make use of Proposition \ref{entire-2}, and define, for $k=1,..., n_0$,
\[
\eta_k(t)=\xi_{b_{i_k-1}}(t)-r^0_k-c_kt.
\]
By Proposition \ref{entire-2}, we have, for any $C>0$,
\begin{equation}
\label{u-Uk}
\lim_{t\to\infty}\Big[ u(r,t)-U_k(r-c_kt-\eta_k(t))\Big]=0 \mbox{ uniformly for }  |r-c_kt-\eta_k(t)|\leq C.
\end{equation}

Let us note that the convergence in Proposition \ref{entire-2} actually holds in $C^{2,1}_{loc}(\R^2)$. It follows that
\[
\lim_{t\to+\infty}\xi'_{b_{i_k-1}}(t)=\lim_{t\to\infty}\frac{-u_t(\xi_{b_{i_k-1}}(t), t)}{u_r(\xi_{b_{i_k-1}}(t), t)}=\frac{c_kU_k'(r^0_{i_k-1})}{U_k'(r^0_{i_k-1})}=c_k.
\]
 Hence
\[
\lim_{t\to\infty}\eta_k'(t)=0.
\]
By the definition of $j_k$ ($k=1,..., m'$) in Lemma \ref{entire-1} and its proof, we see that
\[
\lim_{t\to\infty}\big[\xi_{b_m}(t)-\xi_{b_n}(t)\big]=+\infty
\]
if there exists $j_k$ such that $n<j_k<m$. 
By Lemma \ref{terrace-1}, we have $m'=n_0$ and $j_k=i_k$. Therefore,
 if $c_k=c_{k+1}$ for some $1\leq k<k+1\leq n_0$,
then
\[
\lim_{t\to\infty} \big[\eta_{k+1}(t)-\eta_{k}(t)\big]=+\infty.
\]

To complete the proof, it remains to show \eqref{V-limit} for $V=u$.
Given any small $\epsilon>0$, by \eqref{u-Uk}, we can find large positive constants $T$ and $C$ such that,
for every $k\in\{1,..., n_0\}$, 
\begin{equation}
\label{u1}
\big|u(r,t)-U_k(r-c_kt-\eta_k(t))\big|<\epsilon/2 \mbox{ for } t\geq T, \; |r-c_kt-\eta_k(t)|\leq C,
\end{equation}
and
\begin{equation}\label{Uk-C}
  Q_{k-1}-\epsilon/2<U_k(-C),\; \;U_k(C)< Q_k+\epsilon/2.
\end{equation}
It follows that, for $t\geq T$,
\[
Q_{k-1}-\epsilon<u(c_kt+\eta_k(t)-C,t),\; u(c_{k-1}t+\eta_{k-1}(t)+C,t)<Q_{k-1}+\epsilon.
\]
Since $u(r,t)$ is monotone decreasing in $r$ for $r>R_0$, we deduce
\begin{equation}
\label{u2}
-\epsilon<u(r,t)-Q_{k-1}<\epsilon \mbox{ for $t\geq T$, $r\in [c_{k-1}t+\eta_{k-1}(t)+C, c_kt+\eta_k(t)-C]$},
\end{equation}
\begin{equation}
\label{u3}
u(r,t)<\epsilon \mbox{ for $t\geq T$ and $r\geq c_{n_0}t+\eta_{n_0}(t)+C$}, 
\end{equation}
and 
\[
u(R_0, t)\geq u(r,t)> p-\epsilon \mbox{ for $t\geq T$ and } r\in [R_0, c_1t+\eta_1(t)-C].
\]
Since $u<p$ and $\lim_{t\to\infty}u(r,t)=p$ uniformly for $r\in [0, R_0]$, by enlarging $T$ if necessary, we can assume that
\[
p>u(r,t)>p-\epsilon \mbox{ for $t\geq T$ and $r\in [0, R_0]$}.
\]
Thus we have
\begin{equation}
\label{u4}
p>u(r,t)> p-\epsilon \mbox{ for $t\geq T$ and } r\in [0, c_1t+\eta_1(t)-C].
\end{equation}
Combining inequalities \eqref{u1}, \eqref{Uk-C}, \eqref{u2}, \eqref{u3} and \eqref{u4}, we obtain
\[
\Big|u(r,t)-\sum_{k=1}^{n_0}\Big[U_k(r-c_kt-\eta_k(t))-q_{i_k}\Big]\Big|<(n_0+1)\epsilon \mbox{ for } t\geq T \mbox{ and } r\geq 0.
\]
This clearly implies \eqref{V-limit}. The proof of Theorem \ref{thm-existence-rt} is now complete. \hfill$\Box$

\section{Determination of the logarithmic shifts}

Throughout this section, we always assume that $f$ satisfies {\bf (f1)-(f3)}, $u(r,t)$ is a radial terrace solution of \eqref{nd}, and $\{q_{i_k}, U_k, c_k\}_{1\leq k\leq n_0}$ is the associated one dimensional propagating terrace.
 Since $f'(q_{i_k})<0$ for $k=1,..., n_0$, it is well known that there exist $\check\delta>0$ and $\check M>0$ such that, for every $k\in\{0,..., n_0\}$,
\[\begin{cases}
0<q_{i_{k-1}}-U_k(x)\leq e^{\check\delta x} &\mbox{ for } x\leq -\check M,\\
0<U_k(x)-q_{i_{k}}\leq e^{-\check\delta x} &\mbox{ for } x\geq \check M.
\end{cases}
\]
For convenience of notations below, we define $c_0=0$ and $c_{n_0+1}=+\infty$. 

We want to better understand the functions $\eta_k(t)$, $k=1,..., n_0$, in \eqref{V-limit} for a radial terrace solution $u(r,t)$. The main result in this section is  the following theorem.

\begin{thm}\label{eta_k}
Suppose that $f$ satisfies {\bf (f1)-(f3)}, and the speeds $\{c_k: k=1,..., n_0\}$ in the propagating terrace $\{q_{i_k}, U_k, c_k\}_{1\leq k\leq n_0}$ satisfy\[
c_{k_0-1}<c_{k_0}<c_{k_0+1} \ \mbox{  for some $k_0\in \{1,..., n_0\}$.}
\] 
Then for any radial terrace solution $u(r,t)$ of \eqref{nd}, the function $\eta_{k_0}(t)$ in \eqref{V-limit} satisfies, for some positive constants $C$ and $T$,
\[
\big|\eta_{k_0}(t)+\frac{N-1}{c_{k_0}}\log t\big|\leq C \mbox{ for } t\geq T.\footnote{It is possible to show that $\eta_{k_0}(t)+\frac{N-1}{c_{k_0}}\log t$ converges as $t\to\infty$; see Remark 4.5 for details.}
\]
\end{thm}

We prove Theorem \ref{eta_k} by making use of several lemmas. The key steps involve the construction of suitable upper and lower solutions in Lemmas \ref{ub} and \ref{lb} below, following \cite{DMZ, DQZ, KMY}.

\subsection{Use of the speeds gap}

\begin{lem}\label{tilde-c}
Suppose for some $k\in\{0,..., n_0\}$, we have $c_k<c_{k+1}$. Then for any $\tilde c_k\in (c_k, c_{k+1})$ there exist positive constants $M_k$ and $\delta_k$ such that
\begin{equation}\label{u-ct}
|u(\tilde c_k t, t)-q_{i_k}|\leq M_k e^{-\delta_k t} \ \mbox{ for all large } t>0.
\end{equation}
\end{lem}
The proof of this lemma makes use of the following estimate.

\begin{lem}\label{estimate}
Suppose $\delta, T, C_0$ and $c$ are positive constants, and for any given $\sigma>0$, we denote $D_{\sigma}:=[-\sigma,\sigma]^N=\{x\in\R^N: |x_i|<\sigma,\; i=1,..., N\}$.
Then there exist $M_0>0$, $T_0>0$ and $\epsilon_0\in (0,1)$ such that the unique solution $\Psi(x)$ to
\[\begin{cases}
\Psi_t-\Delta\Psi=\delta C_0e^{\delta t} &\mbox{ for } (x,t)\in D_{cT}\times (0,\infty),\\ 
\Psi=0 &\mbox{ for } (x,t)\in (\partial D_{cT}\times [0,\infty))\cup (D_{cT}\times \{0\})
\end{cases}
\]
satisfies
\begin{equation}\label{Psi}
\Psi(x,t)\geq C_0(1-M_0 e^{-T/2})(e^{\delta t}-1) \end{equation}
 for $  T\geq T_0, \ x\in D_{(1-\epsilon)cT}, \ 0\leq t\leq \frac{\epsilon^2c^2}{4}T$ and $\epsilon\in (0,\epsilon_0]$.

\end{lem}
\begin{proof}
This is contained in the proof of Lemma 2.6 in \cite{Du-P}, which involves the use of Green's function for $\partial_t-\Delta$ over 
$D_{cT}$ with homogeneous Dirichlet boundary conditions, and an estimate  for the one dimensional case in the proof of Lemma 6.5 in \cite{Du-L}.
\end{proof}

\begin{proof}[{\bf Proof of Lemma \ref{tilde-c}}]
Choose $c_k^-$ and $c_k^+$ such that
\[
c_k<c_k^-<\tilde c_k<c_k^+<c_{k+1}.
\]
By \eqref{V-limit} we see that
\[
u(r,t)\to q_{i_k} \mbox{ uniformly for $r\in [c_k^-t, c_k^+t]$ as } t\to\infty.
\]
Therefore, for any given $\sigma\in (0, q_{i_k})$, we can find $T_1=T_1(\sigma)$ such that
\[
q_{i_k}-\sigma \leq u(r,t)\leq q_{i_k}+\sigma \mbox{ for $r\in [c_k^-t, c_k^+t]$ and } t\geq T_1.
\]
Since $f'(q_{i_k})<0$, by fixing $\sigma>0$ sufficiently small, 
\[\delta^*:=-\max_{|u-q_{i_k}|\leq \sigma}f'(u)>0.
\]
We now consider the auxiliary problems
\[\begin{cases}
\overline u_t-\Delta \overline u=\delta^*(q_{i_k}-\overline u), & |x|\in [c_k^- t, c_k^+t], \; t>T_1,\\
\overline u=q_{i_k}+\sigma, & |x|\in \{c_k^-t, c_k^+t\}, \; t>T_1,\\
\overline u=q_{i_k}+\sigma, & |x|\in [c_k^-t, c_k^+t], \; t=T_1,
\end{cases}
\]
and
\[
\begin{cases}
\underline u_t-\Delta \underline u=\delta^*(q_{i_k}-\underline u), & |x|\in [c_k^- t, c_k^+t], \; t>T_1,\\
\underline u=q_{i_k}-\sigma, & |x|\in \{c_k^-t, c_k^+t\}, \; t>T_1,\\
\underline u=q_{i_k}-\sigma, & |x|\in [c_k^-t, c_k^+t], \; t=T_1.
\end{cases}
\]
It follows easily from the comparison principle that
\[
\underline u(x,t)\in [q_{i_k}-\sigma, q_{i_k}),\; \overline u(x,t)\in (q_{i_k}, q_{i_k}+\sigma] 
\]
for $|x|\in  [c_k^-t, c_k^+t], \; t\geq T_1$. Therefore, for such $(x,t)$ we have
\[
f(\underline u)\geq  \delta^* (q_{i_k}-\underline u),\; f(\overline u)\leq \delta^* (q_{i_k}-\overline u),
\]
and so we can apply the comparison principle to deduce
\[
\underline u(x,t)\leq u(|x|, t)\leq \overline u(x,t)\  \mbox{ for $|x|\in  [c_k^-t, c_k^+t], \; t\geq T_1$.}
\]
Define
\[
\underline\psi:=e^{\delta^* t}(\underline u -q_{i_k}+\sigma),\; \overline\psi:=-e^{\delta^* t}(\overline u -q_{i_k}-\sigma).
\]
Then it is easily seen that they both solve the following problem
\[
\begin{cases}
 \psi_t-\Delta  \psi=\delta^* \sigma e^{\delta^* t},\  \psi\geq 0, & |x|\in [c_k^- t, c_k^+t], \; t>T_1,\\
 \psi=0, & |x|\in \{c_k^-t, c_k^+t\}, \; t>T_1,\\
 \psi=0, & |x|\in [c_k^-t, c_k^+t], \; t=T_1.
\end{cases}
\]
Let $\psi^*$ denote the unique solution of the above problem; then $\underline\psi=\overline\psi=\psi^*$.
For any $T\geq T_1$ and $\eta\in (0,c_k^+-c_k^-)$, define $\tilde T:=\frac{\eta}{c_k^-} T$. Then clearly
\[
c_k^-(T+\tilde T)=(c_k^-+\eta)T<c_k^+T,
\]
which implies that 
\[
[c_k^-(T+\tilde T), c_k^+T]=[(c_k^-+\eta)T, c_k^+T]\subset [c_k^-t, c_k^+t] \mbox{ for all } t\in [T, T+\tilde T].
\]

A simple geometric consideration shows that there exists $\hat c_k>0$ such that with 
\[
x^k:=(\frac{c_k^-+\eta+c_k^+}2 T,0,..., 0)\in\R^N,
\]
we have
\[
x^k+D_{\hat c_kT}\subset\left\{x\in\R^N: |x|\in [(c_k^-+\eta)T, c_k^+T]\right\}.
\]
We now define
\[
\Psi^*(x,t):= \psi^*(x+x^k, t+T).
\]
Then we have, for any $T\geq T_1$,
\[
\begin{cases}
 \Psi^*_t-\Delta \Psi^*=\delta^* \sigma e^{\delta^* t},\  & x\in D_{\hat c_kT}, 0<t\leq \tilde T,\\
\Psi^*\geq 0, & x\in\partial D_{\hat c_k T}, 0<t\leq \tilde T,\\ 
 \Psi^*\geq 0, & x\in D_{ \hat c_k T}, \ t=0.
\end{cases}
\]
We may now use Lemma \ref{estimate} and the comparison principle to conclude that, there exist $T_0\geq T_1$, $M_0>0$ and $\epsilon\in(0,1)$ small such that
\[
 \Psi^*(x,t)\geq \sigma (1-M_0 e^{-T/2})(e^{\delta^* t}-1)
\]
for   $  T\geq T_0, \ x\in D_{(1-\epsilon)\hat c_kT} \mbox{ and } 0\leq t\leq \frac{\epsilon^2\hat c_k^2}{4}T.$

Using the definitions of $\Psi^*$, $\psi^*$, $\underline\psi$ and $\overline\psi$, we obtain
\[\begin{cases}
e^{\delta^* t}[\underline u(x+x_k, T+t)-q_{i_k}+\sigma]\geq \sigma (1-M_0 e^{-T/2})(e^{\delta^* t}-1),\\
-e^{ \delta^* t}[\overline u(x+x_k, T+t)-q_{i_k}-\sigma]\geq \sigma (1-M_0 e^{-T/2})(e^{\delta^* t}-1)
\end{cases}
\]
for   $  T\geq T_0,\; |x|\leq (1-\epsilon)\hat c_kT \mbox{ and } 0\leq t\leq \frac{\epsilon^2\hat c_k^2}{4}T.$
From the definitions of $\underline u$ and $\overline u$ it is easily seen that they only depend on $|x|$. So we may write
$\underline u=\underline u(|x|, t)$ and $\overline u=\overline u(|x|, t)$, and the above estimates yield
\[
\begin{cases}
\underline u(|x|, T+t)- q_{i_k}\geq -\sigma M_0 e^{-T/2}-\sigma e^{-\delta^* t},\\
\overline u(|x|, T+t)- q_{i_k}\leq \sigma M_0 e^{-T/2}+\sigma e^{-\delta^* t}
\end{cases}
\]
for   $  T\geq T_0, \ (|x^k|-(1-\epsilon)\hat c_k)T\leq |x|\leq  (|x^k|+(1-\epsilon )\hat c_k)T\  \mbox{ and } 0\leq t\leq \frac{\epsilon^2\hat c_k^2}{4}T.$ It follows that
\[
| u(|x|, T+t)-q_{i_k}|\leq \sigma M_0 e^{-T/2}+\sigma e^{-\delta^* t}
\]
for   $  T\geq T_0, \ (|x^k|-(1-\epsilon)\hat c_k)T\leq |x|\leq  (|x^k|+(1-\epsilon) \hat c_k)T\  \mbox{ and } 0\leq t\leq \frac{\epsilon^2\hat c_k^2}{4}T.$

By taking $\eta$ and $\epsilon$ sufficiently small, and $c_k^-=\tilde c_k-2\eta,\ c_k^+=\tilde c_k+\eta$, we obtain $|x^k|=\tilde c_k$, and so 
we can guarantee that
\[
|x^k|-(1-\epsilon)\hat c_k\leq \tilde c_k(1+\frac{\epsilon^2\hat c_k^2}{4})\leq |x^k|+(1-\epsilon) \hat c_k.
\]
Fix such $(\epsilon, \eta, c_k^-, c_k^+)$, take $t=\frac{\epsilon^2\hat c_k^2}{4}T$ and denote $\hat T:=(1+\frac{\epsilon^2\tilde c_k^2}{4})T$; then we obtain from the above estimate for $u$ that
\[
| u(\tilde c_k\hat T, \hat T)-q_{i_k}|\leq \sigma M_0 e^{-T/2}+\sigma e^{-\delta^*\frac{\epsilon^2\hat c_k^2}{4}T }
\]
for   $  \hat T\geq \frac{\epsilon^2\hat c_k^2}{4}T_0$. This clearly implies \eqref{u-ct}.
\end{proof}

\subsection{Sharp upper and lower bounds for the radial terrace solution}
An upper bound for $u(r,t)$ is given below.
\begin{lem}\label{ub} Suppose $k\in\{1,..., n_0\}$ and $c_{k-1}<c_{k}$.
Then for any  $\tilde c_{k-1}\in (c_{k-1}, c_{k})$, there exist constants $L\gg 1, T_0\gg 1$ and $H^0\in\R$ such that
\[
u(r,t)\leq U_k\left(r-c_{k}t+\frac{N-1}{c_{k}}\log t+H^0\right)+\frac{\log t}{t^2}
\]
for $t\geq T_0$, $r\in [\tilde c_{k-1}, c_{k} t+L\log t]$. 
\end{lem}

In order to prove Lemma \ref{ub},  we first  obtain a rough upper bound.

\begin{lem}\label{ub-rough}
Under the assumptions of Lemma \ref{ub}, for any $\tilde c_{k-1}\in (c_{k-1}, c_{k})$ and $c>c_{n_0}$, there exist constants $\delta>0, T>0$ and $H\in\R$ such that
\[
u(r,t)\leq U_{k}\left(r-c_{k}t+\frac{N-1}{c}\log t+H\right)+e^{-\delta t}
\]
for $t\geq T$, $r\in [\tilde c_{k-1} t, ct]$.
\end{lem}
\begin{proof}
We define
\[
\overline u(r,t):=U_{k}\left(r-c_{k}(t-T)+\frac{N-1}{c}\log \frac{t}T -R-\rho  (e^{-\delta T}-e^{-\delta t})\right)+e^{-\delta t},
\]
and show that by choosing the positive constants $T, R, \rho$ and $\delta$ suitably, $\overline u(r,t)$ satisfies
\begin{equation}\label{u-us}
\begin{cases}
\overline u_t-\overline u_{rr}-\frac{N-1}{r}\overline u_r\geq f(\overline u) & \mbox{ for } r\in [\tilde c_{k-1} t, ct],\\
\overline u(\tilde c_{k-1}t,t)\geq u(\tilde c_{k-1} t, t) &\mbox{ for } t\geq T,\\
\overline u(ct,t)\geq u(c t, t) &\mbox{ for } t\geq T,\\
\overline u(r, T)\geq u(r, T) &\mbox{ for } r\in  [\tilde c_{k-1} T, cT].
\end{cases}
\end{equation}
The desired estimate then follows directly from the comparison principle and the monotonicity of $U_{k}$.
Therefore, to complete the proof, it suffices to prove \eqref{u-us}.

By \eqref{u-ct},
\[
u(\tilde c_{k-1} t, t)\leq q_{i_{k-1}}+M_{k-1} e^{-\delta_{k-1}t} \leq q_{i_{k-1}}+\frac 12 e^{-\delta t}\mbox{ for all large } t>0
\]
provided that $\delta<\delta_k$.
Since
\begin{align*}
\overline u(\tilde c_{k-1}t,t)&\geq U_k\left(\frac{\tilde c_{k-1}-c_k}2 t\right)+e^{-\delta t}\\
&\geq q_{i_{k-1}}-e^{\check\delta\frac{\tilde c_{k-1}-c_k}2 t}+e^{-\delta t} \mbox{ for all  large } t>0,
\end{align*}
we see that the second inequality of \eqref{u-us} holds provided that $\delta<\min\left\{\delta_k, -\check\delta\frac{\tilde c_{k-1}-c_k}2\right\}$ and $T$ is sufficiently large.

Let $\hat\delta:=2(c-c_{k})\check\delta >0$; then  we have, for all large $t$, 
\[
\overline u(ct,t)\geq U_{k}\left(2( c-c_{k})t\right)+e^{-\delta t}\geq q_{i_{k}}-e^{-\hat\delta t}+e^{-\delta t}>q_{i_{n_0}}+\frac 12 e^{-\delta t}
\]
provided that $\delta<\hat\delta$.
By Lemma \ref{estimate} we have
\[
u(ct,t)\leq q_{i_{n_0}}+M_{n_0}e^{-\delta_{n_0} t}\leq q_{i_{n_0}}+\frac 12 e^{-\delta t}
\] for all large $t$ provided that $\delta<\delta_{n_0}$.
Hence
$
\overline u(ct,t)\geq  u(ct,t)
$
for all large $t$ provided that  $\delta<\min\{\hat \delta, \delta_{n_0}\}$.
Thus the third inequality in \eqref{u-us}  holds if we choose $\delta$ and $T$ properly.

Summarising, if we choose $\delta>0$ sufficiently small and $T=T(\delta)>0$ sufficiently large, then the second and the third inequalities in \eqref{u-us} hold. 

For later analysis, let us assume by shrinking $\delta$ suitably  that $-\delta>\max_{0\leq k\leq n_0}f'(q_{i_k})$, and then we fix $\sigma>0$ small so that 
\[
f'(u)< -\delta<0 \mbox{ for } u\in\cup_{k=0}^{n_0}[q_{i_k}-\sigma,q_{i_k}+\sigma].
\]
By enlarging $T$ we can guarantee that
\[
e^{-\delta T}\leq \sigma/2.
\]
We now fix  $\delta$ and $T$ meeting all the above requirments. 

For $r\in [\tilde c_{k-1} T, cT]$,  by the monotonicity of $U_{k}(\cdot)$ and $u(\cdot, t)$, we have
\[
\overline u(r,T)\geq U_{k}(-R+cT)+e^{-\delta T}\geq q_{i_{k-1}}+\frac 12 e^{-\delta T}\geq u(\tilde c_{k-1}T,T)\geq u(r,T)
\]
provided that $R>0$ is sufficiently large.

Finally we prove the first inequality in \eqref{u-us}. We have
\[
\overline u_t=\left(-c_{k}+\frac{N-1}{c t}-\delta \rho  e^{-\delta t}\right)U'_{k}-\delta  e^{-\delta t},\; \overline u_r=U_{k}',\; \overline u_{rr}=U''_{k},
\]
where for simplicity of notation we have used $U_{k}=U_{k}(\tilde\xi(r,t))$ with
\[
\tilde\xi(r,t):=r-c_{k}(t-T)+\frac{N-1}{ct}\log \frac{t}T -R-\rho M (e^{-\delta T}-e^{-\delta t}).
\]
Therefore, for $r\in [\tilde c_{k-1} t, ct]$ and $t\geq T$,
\begin{align*}
&\overline u_t-\overline u_{rr}-\frac{N-1}{r}\overline u_r- f(\overline u) \\
&=\left(-c_{k}+\frac{N-1}{c t}-\delta \rho  e^{-\delta t}-\frac{N-1}r\right)U'_{k}-U_{k}'' -\delta  e^{-\delta t}-f(\overline u)\\
&=\left(\frac{N-1}{c t}-\delta \rho  e^{-\delta t}-\frac{N-1}r\right)U'_{k} -\delta  e^{-\delta t}+f(U_{k})-f(\overline u)\\
&\geq -\delta \rho e^{-\delta t}U'_{k}-\delta  e^{-\delta t}+f(U_{k})-f(\overline u)\\
&=\big[-\delta \rho U'_{k}-\delta -f'(\theta(r,t))\big] e^{-\delta t}=:\tilde J(r,t),
\end{align*}
where $\theta(r,t)\in [U_k, U_{k}+e^{-\delta t}]$.

If $r\in [\tilde c_{k-1}t, ct]$ and $t\geq T$ is such that $U_{k}(\tilde \xi(r,t))\in [q_{i_{k}}+\sigma/2, q_{i_{k-1}}-\sigma/2]$, then
there exists $\epsilon_0>0$ such that 
\[
U'_{k}(\tilde \xi(r,t))\leq -\epsilon_0 \mbox{ for such } (r,t),
\]
and so
\[
\tilde J(r,t)\geq \big[\delta\rho \epsilon_0-\delta -f'(\theta(r,t))\big] e^{-\delta t}\geq 0
\]
if we choose 
\[
\rho\geq \frac{\delta+\max_{u\in[q_{i_{k}}, q_{i_{k-1}}+\sigma]}|f'(u)|}{\delta\epsilon_0}.
\]

If $r\in [\tilde c_{k-1} t, ct]$ and $t\geq T$ is such that $U_{k}(\tilde \xi(r,t))\in [q_{i_{k-1}}-\sigma/2, q_{i_{k-1}})\cup( q_{i_k}, q_{i_k}+\sigma/2]$, then $\theta(r,t)\in  [q_{i_{k-1}}-\sigma/2, q_{i_{k-1}}+\sigma/2)\cup[ q_{i_k}, q_{i_k}+\sigma]$ and so
$f'(\theta(r,t))<-\delta$,
\[
\tilde J(r,t)\geq [-\delta -f'(\theta(r,t))\big] e^{-\delta t}\geq 0.
\]
Thus the first inequality of \eqref{u-us} always hold when $\rho $ is chosen as above.
\end{proof}

\begin{proof}[{\bf Proof of Lemma \ref{ub}}] 
Define
\[
\tilde w(r,t):=U_{k}\left(r-c_{k}(t-T)+\frac{N-1}{c_{k}}\log \frac tT-M(\frac{\log T}T-\frac{\log t} t)-R\right)+\frac{\log t}{t^2};
\]
we  show that for fixed $\tilde c_{k-1}\in (c_{k-1}, c_{k})$, there exist   positive constants $T$, $M$, $R$ and $L$ such that
\begin{equation}\label{w-us}
\begin{cases}
\tilde w_t-\tilde w_{rr}-\frac{N-1}{r}\tilde w_r\geq f(\tilde w) & \mbox{ for } r\in [\tilde c_{k-1} t, c_{k} t+L\log t],\\
\tilde w(c_{k}t+L\log t,t)\geq u(c_{k}t+L\log t, t) &\mbox{ for } t\geq T,\\
\tilde w(\tilde c_{k-1} t,t)\geq u(\tilde c_{k-1} t, t) &\mbox{ for } t\geq T,\\
\tilde w(r, T)\geq u(r, T) &\mbox{ for } r\in  [\tilde c_{k-1}T, c_{k}T+L\log  T].
\end{cases}
\end{equation}
The conclusion of Lemma \ref{ub} clearly follows directly from \eqref{w-us} and the comparison principle. So to prove the lemma, it suffices to prove \eqref{w-us}.

By Lemma \ref{ub-rough} we have, for all large $t>0$,
\begin{align*}
u(c_{k}t+L\log t, t)&\leq U_{k}\left(L\log t+\frac{N-1}{c}\log  t+H\right)+e^{-\delta t}\\
&\leq U_{k}(\frac L 2 \log t)+e^{-\delta t}\\
&\leq q_{i_k}+e^{-\check \delta \frac{L}2 \log t}+e^{-\delta t}\\
&\leq q_{i_k}+t^{-2}\leq \tilde w(c_{k}t+L\log t, t)
\end{align*}
provided that $L> \max\left\{2\frac{N-1}{c}, \frac 4{\check\delta}\right\} $. 
Moreover, for all large $t>0$,
\begin{align*}
\tilde w(\tilde c_{k-1} t, t)&\geq U_{k}\left(\frac 12 (\tilde c_{k-1} -c_{k}) t\right)+\frac{\log t}{t^2}\\
&\geq q_{i_{k-1}}-e^{\check \delta  \frac{\tilde c_{k-1}-c_{k}}2  t}+\frac{\log t}{t^2}\\
&\geq q_{i_{k-1}}+M_k e^{-\delta_{k-1} t}\geq u(\tilde c_{k-1} t, t).
\end{align*}

Thus we can fix $L$ and $T>0$ large so that the second and third inequalities in \eqref{w-us} hold
for $t\geq T$. 

We next prove the first inequality in \eqref{w-us}.
For $r\in [\tilde c_{k-1}t, c_{k} t+L\log t]$ and $t\geq T$,
\begin{align*}
&\tilde w_t-\tilde w_{rr}-\frac{N-1}{r}\tilde w_r- f(\tilde w) \\
&=\left(-c_{k}+\frac{N-1}{c_{k} t}-\frac {M\log t-M}{t^2} -\frac{N-1}r\right)U'_{k}-U_{k}'' -\frac{2\log t}{t^3}+\frac 1{t^3}-f(\tilde w)\\
&=\left(\frac{N-1}{c_{k} t}-\frac {M\log t-M}{t^2} -\frac{N-1}r\right)U'_{k} -\frac{2\log t}{t^3}+\frac 1{t^3}  +f(U_{k})-f(\tilde w)\\
&\geq \left(\frac{N-1}{c_{k} t}-\frac {M\log t-M}{t^2} -\frac{N-1}{c_{k} t+L\log t}\right)U'_{k} -\frac{2\log t}{t^3}  +f(U_{k})-f(\tilde w).
\end{align*}
We have
\begin{align*}
&\frac{N-1}{c_{k} t}-\frac {M\log t-M}{t^2} -\frac{N-1}{c_{k} t+L\log t}\\
&=\left[-M+\frac{L(N-1)}{c_{k}^2}+o(1)\right]\frac{\log t}{t^2}\\
&\leq
\left[-\frac M2+\frac{L(N-1)}{c_{k}^2}\right]\frac{\log t}{t^2}
\end{align*}
for all large $t$ provided that $M>2L(N-1)/c_{k}^2$.
Moreover,
\[
f(U_{k})-f(\tilde w)=-f'(\zeta(r,t))\frac{\log t}{t^2}
\]
with $\zeta(r,t)\in [U_k, U_{k}+\frac{\log t}{t^2}]$, where $U_{k}=U_{k}(\tilde \eta(r,t))$ and
\[
\tilde \eta(r,t):=r-c_{k}(t-T)+\frac{N-1}{c_{k}}\log \frac tT-M(\frac{\log T}T-\frac{\log t} t)-R.
\]
We thus have, for all large $t$ and $r\in [\tilde c_{k-1}t, c_{k} t+L\log t]$,
\[
\tilde w_t-\tilde w_{rr}-\frac{N-1}{r}\tilde w_r- f(\tilde w)\geq \tilde I(r,t)\frac{\log t}{t^2}
\]
with
\[
\tilde I(r,t):=\left[-\frac{M}2+\frac{L(N-1)}{c_{k}^2}\right]U'_{k}(\tilde \eta(r,t))-\frac{2}{t}-f'(\zeta(r,t)).
\]

With $\sigma$ chosen as in the proof of Lemma \ref{lb-rough}, if  $r\in [\tilde c_{k-1}t, c_{k}t+L\log  t]$  is such that $U_{k}(\tilde \eta(r,t))\in [q_{i_{k-1}}-\sigma/2, q_{i_{k-1}})\cup( q_{i_k}, q_{i_k}+\sigma/2]$, then for all large $t$, 
\[
\zeta(r,t)\in  [q_{i_{k-1}}-\sigma/2, q_{i_{k-1}}+\sigma/2)\cup ( q_{i_k}, q_{i_k}+\sigma]
\]
 and so
\[f'(\zeta(r,t))<-\delta,\ \ 
\tilde I(r,t)\geq -\frac{2}t+\delta>0.
\]

If  $r\in [\tilde c_{k-1} t, c_{k}t+L\log  t]$ is such that $U_{k}(\tilde\eta(r,t))\in [q_{i_{k}}+\sigma/2, q_{i_{k-1}}-\sigma/2]$, then
there exists $\epsilon_0>0$ such that 
\[
U'_{k}(\tilde\eta(r,t))\leq -\epsilon_0 \mbox{ for such } (r,t),
\]
and so
\[
\tilde I(r,t)\geq \left[\frac{M}2-\frac{L(N-1)}{c_{k}^2}\right] \epsilon_0-\frac 2 t -\max_{u\in [q_{i_{k}}, q_{i_{k-1}}+\sigma]}f'(u)\geq 0
\]
for all large $t$ provided that $M$ is chosen large enough. 

Thus the first inequality in \eqref{w-us} holds for fixed large $M$ and all large $t$, say $t\geq T_1$.  So by further enlarging $T$ we see that the first three inequalities of \eqref{w-us} hold for $t\geq T$, provided that $M$ and $L$ are chosen as indicated above.
It is now important to observe that these choices of $M, L$ and $T$ are all independent of $R$.

Next we choose $R$ such that the fourth inequality of \eqref{w-us} holds. Before doing that, let us note that
by enlarging $T$ further if necessary we may assume  that 
\[
e^{-\delta T}\leq T^{-2}.
\]
Then for $r\in  [\tilde c_{k-1}T, c_{k}T+L\log  T]$ we have
\[
\tilde w(r, T)\geq U_{k}(c_{k}T+L\log T-R)+\frac{\log T}{T^2}\geq q_{i_{k-1}}+T^{-2}
\]
provided that $R>0$ is large enough, and by  Lemma \ref{ub-rough},
\[
u(r, T)\leq q_{i_{k-1}}+e^{-\delta T}\leq q_{i_{k-1}}+T^{-2}\leq \tilde w(r, T) \mbox{ for } r\in  [\tilde c_{k-1}T, c_{k}T+L\log  T].
\]
Hence the fourth inequality in \eqref{w-us} also holds.
\end{proof}

The following lemma gives a lower bound for $u$.
\begin{lem}\label{lb} Suppose $k\in\{1,..., n_0\}$ and $c_k<c_{k+1}$.
Then for any  $\tilde c_k\in (c_k, c_{k+1})$, there exist constants $L\gg 1, T_0\gg 1$ and $H_0\in\R$ such that
\[
u(r,t)\geq U_k\left(r-c_{k}t+\frac{N-1}{c_{k}}\log t+H_0\right)-\frac{\log t}{t^2}
\]
for $t\geq T_0$, $r\in [c_{k} t-L\log t, \tilde c_k t]$. 
\end{lem}

To prove Lemma \ref{lb}, similar to the proof of Lemma \ref{ub}, we first  obtain a rough lower bound.

\begin{lem}\label{lb-rough}
Under the assumptions of Lemma \ref{lb}, for any $\tilde c_k\in (c_k, c_{k+1})$ and $c\in (0, c_1)$, there exist constants $\delta>0, T>0$ and $H\in\R$ such that
\[
u(r,t)\geq U_{k}\left(r-c_{k}t+\frac{N-1}{c}\log t+H\right)-e^{-\delta t}
\]
for $t\geq T$, $r\in [ct, \tilde c_k t]$.
\end{lem}
\begin{proof}
We define
\[
\underline u(r,t):=U_{k}\left(r-c_{k}(t-T)+\frac{N-1}{c}\log \frac{t}T +R+\rho  (e^{-\delta T}-e^{-\delta t})\right)-e^{-\delta t},
\]
and show that by choosing the positive constants $T, R, \rho$ and $\delta$ suitably, $\underline u(r,t)$ satisfies
\begin{equation}\label{u-ls}
\begin{cases}
\underline u_t-\underline u_{rr}-\frac{N-1}{r}\underline u_r\leq f(\underline u) & \mbox{ for } r\in [ct, \tilde c_k t],\\
\underline u(\tilde c_kt,t)\leq u(\tilde c_k t, t) &\mbox{ for } t\geq T,\\
\underline u(ct,t)\leq u(c t, t) &\mbox{ for } t\geq T,\\
\underline u(r, T)\leq u(r, T) &\mbox{ for } r\in  [cT, \tilde c_k T].
\end{cases}
\end{equation}
The desired estimate then follows directly from the comparison principle and the monotonicity of $U_{k}$.
Since the proof of \eqref{u-ls} is analogous to that for \eqref{u-us}, the details are omitted.
\end{proof}

\begin{proof}[{\bf Proof of Lemma \ref{lb}}] 
Define
\[
w(r,t):=U_{k}\left(r-c_{k}(t-T)+\frac{N-1}{c_{k}}\log \frac tT+M(\frac{\log T}T-\frac{\log t} t)+R\right)-\frac{\log t}{t^2};
\]
we  show that for fixed $\tilde c_k\in (c_{k}, c_{k+1})$, there exist   positive constants $T$, $M$, $R$ and $L$ such that
\begin{equation}\label{w-ls}
\begin{cases}
w_t-w_{rr}-\frac{N-1}{r}w_r\leq f(w) & \mbox{ for } r\in [c_{k} t-L\log t, \tilde c_k t],\\
w(c_{k}t-L\log t,t)\leq u(c_{k}t-L\log t, t) &\mbox{ for } t\geq T,\\
w(\tilde c_k t,t)\leq u(\tilde c_k t, t) &\mbox{ for } t\geq T,\\
w(r, T)\leq u(r, T) &\mbox{ for } r\in  [c_{k}T-L\log  T, \tilde c_kT].
\end{cases}
\end{equation}
The conclusion of Lemma \ref{lb} clearly follows directly from \eqref{w-ls} and the comparison principle. 

To prove \eqref{w-ls}, we use Lemma \ref{lb-rough} and proceed similarly as in the proof of Lemma \ref{ub}.
As the changes are minor and easy to see, the details are again omitted.
\end{proof}

\subsection{Proof of Theorem \ref{eta_k}}
By Lemmas \ref{ub} and \ref{lb}, there exists $T_1>0$ large such that for all $t\geq T_1$, 
\[
\begin{cases}
u(c_{k_0}t-\frac{N-1}{c_{k_0}}\log t,t)\geq U_{k_0}\left(H_0\right)-\frac{\log t}{t^2},\\
u(c_{k_0}t-\frac{N-1}{c_{k_0}}\log t,t)\leq U_{k_0}\left(H^0\right)+\frac{\log t}{t^2}.
\end{cases}
\]
By \eqref{V-limit}, for any given small $\epsilon>0$, there exists $T_2=T_2(\epsilon)$ such that
\[
|u(c_{k_0}t-\frac{N-1}{c_{k_0}}\log t,t)-U_{k_0}(-\eta_{k_0}(t)-\frac{N-1}{c_{k_0}}\log t)|\leq \epsilon \mbox{ for } t\geq T_2.
\]
It follows that for $t\geq \max\{T_1, T_2\}$,
\[
\begin{cases}
U_{k_0}(-\eta_{k_0}(t)-\frac{N-1}{c_{k_0}}\log t)\geq U_{k_0}\left(H_0\right)-\frac{\log t}{t^2}-\epsilon,\\
U_{k_0}(-\eta_{k_0}(t)-\frac{N-1}{c_{k_0}}\log t)\leq U_{k_0}\left(H^0\right)+\frac{\log t}{t^2}+\epsilon.
\end{cases}
\]
Therefore, by the strict monotonicity of $U_{k_0}$, we have
\[
-\eta_{k_0}(t)- \frac{N-1}{c_{k_0}}\log t\leq \tilde H_0,\; -\eta_{k_0}(t)- \frac{N-1}{c_{k_0}}\log t\geq \tilde H^0
\]
for  some  constants $\tilde H_0, \tilde H^0$, and all large $t$, say $t\geq T$. Thus the desired inequality holds with $C=\max\{|\tilde H_0|,|\tilde H^0|\}.$
\hfill $\Box$

\section{Proof of the main results}

\subsection{Proof of Proposition \ref{prop2}}
In this subsection, we prove Proposition \ref{prop2}  based on the following result, which may be viewed as a variation of Lemma \ref{fm-sup}.

\begin{lem}
\label{W-sup-sub}
Let $V(r,t)$ be a radial terrace solution of \eqref{nd}.  Then there exist positive constants $\sigma_0$ and $\beta_0$  such that, for every  $t_0\geq 1$, $\beta\in (0,\beta_0]$ and $\sigma\in (0,\sigma_0]$,
\begin{equation}\label{W}
\overline W(x,t):=V(|x|, t+t_0+1-e^{-\beta t})+\sigma \beta e^{-\beta t}
\end{equation}
satisfies
\begin{equation}\label{W-sup}
\overline W_t-\Delta \overline W\geq f(\overline W) \mbox{ for $ x\in\R^N$ and $ t>0$,}
\end{equation}
and 
\[
\underline W(x,t):=V(|x|, t+t_0-1+e^{-\beta t})-\sigma \beta e^{-\beta t}
\]
satisfies 
\[
 \underline W_t-\Delta \underline W\leq f(\underline W)  \mbox{ for $ x\in\R^N$ and $ t>0$}.
\]
\end{lem}
\begin{proof}
For $k\in \{0,..., n_0\}$, we have $f'(q_{i_k})<0$. Therefore we can find small positive constants $\eta$ and $\epsilon$ such that
\[
f'(u)<-\eta \mbox{ for } u\in I_\epsilon:=\bigcup_{k=0}^{n_0}[q_{i_k}-\epsilon, q_{i_k}+\epsilon].
\]
Next we choose a large constant $C>0$ such that
\[
U_k(\pm C)\in I_{\epsilon/3} \mbox{ for } k=1,..., n_0.
\]
Then there exists $\delta>0$ such that
\[
c_kU_k'(r)<-2\delta \mbox{ for } r\in [-C, C],\; k=1,..., n_0.
\]
By the monotonicity of $U_k$ we find that
\[
U_k(\R\setminus [-C, C])\subset I_{\epsilon/3},\; k=1,..., n_0.
\]
Hence by \eqref{V-limit}, we can find $T_1>0$ large such that
\begin{equation}\label{V1}
V([0,\infty)\setminus I^C(t), t)\subset I_{\epsilon/2} \mbox{ for } t\geq T_1,
\end{equation}
where
\[
I^C(t):=\bigcup_{k=1}^{n_0}I_k^C(t):=\bigcup_{k=1}^{n_0}[c_kt+\eta_k(t)-C, c_kt+\eta_k(t)+C].
\]
By standard parabolic regularity theory, it follows from \eqref{V-limit} that, for every $k\in \{1,..., n_0\}$,
\[
\max_{r\in I_k^C(t)}|V_t(r,t)+c_kU_k'(r-c_kt-\eta_k(t))|\to 0 \mbox{ as } t\to\infty.
\]
Therefore there exists $T_2\geq T_1$ such that
\begin{equation}
\label{V2}
V_t(r,t)\geq \max_{1\leq k\leq n_0}\left[-c_kU_k'(r-c_kt-\eta_k(t))\right]-\delta>\delta \mbox{ for } r\in I^C(t),\; t\geq T_2.
\end{equation}

By Lemma \ref{infty}, $V(r, T_2)\to 0$ as $r\to\infty$. Since $V_t>0$, it follows that
$\lim_{r\to\infty}V(r,t)=0$ uniformly for $t\in [0, T_2]$. Hence there exists $R>0$ such that
\[
V(r,t)\in I_{\epsilon/3}  \mbox{ for } r\geq R,\; t\in [0, T_2].
\]
Set
\[
\tilde\delta:=\min\{V_t(r,t): r\in[0,R],\; t\in [1, T_2]\},\; \delta_0:=\min\{\delta, \tilde\delta\}.
\]

We now fix $t_0\geq 1$ and define, for $r\geq 0$, $t\geq 0$,
\begin{equation*}
W(r,t):=V(r, t+t_0+1-e^{-\beta t})+\sigma \beta e^{-\beta t},
\end{equation*}
with $\sigma$ and $\beta$ positive constants to be determined.

We calculate
\begin{align*}
W_t-W_{rr}-\frac{N-1}{r}W_r&=V_t-V_{rr}-\frac{N-1}{r}V_r+\beta e^{-\beta t}(V_t-\sigma\beta)\\
&=f(V)+\beta e^{-\beta t}(V_t-\sigma\beta)\\
&=f(W)+J,
\end{align*}
with
\begin{align*}
 J&=(V_t-\sigma\beta)\beta e^{-\beta t}+f(V)-f(V+\sigma\beta e^{-\beta t})\\
 &=(V_t-\sigma\beta)\beta e^{-\beta t}-f'(V+\theta )\sigma\beta e^{-\beta t}\\
 &= \beta e^{-\beta t}\Big\{\big[-f'(V+\theta)-\beta\big]\sigma+V_t\Big\},
 \end{align*}
 where $V, V_t, V_r$ are evaluated at $(r, t+t_0+1-e^{-\beta t})=:(r,\tilde t)$, and
 \[ \theta=\theta(r, \tilde t)\in [0, \sigma\beta e^{-\beta t}].
 \]
Take $\beta\in (0,\beta_0]:=(0, \eta]$ and then
choose $M_0>0$ such that 
\[
-f'(V+\theta)-\beta_0\geq -M_0 \mbox{ for all $r\geq 0$ and $t\geq 0$}.
\]
We now set
\[
\sigma_0=\min\left\{\frac{\epsilon}{2\beta_0},\frac{\delta_0}{M_0}\right\},
\]
and take $\sigma\in (0,\sigma_0]$.

For $r\in [0, \infty)\setminus I^C(\tilde t)$ and $\tilde t\geq T_2$, by \eqref{V1} we have
$V\in I_{\epsilon/2}$, and since now $\theta\in [0,\epsilon/2]$, we obtain
\[
V+\theta\in I_\epsilon \mbox{ and hence } [-f'(V+\theta)-\beta]\geq \eta-\beta_0=0.
\]
We note that $V_t\geq 0$ always holds. Hence $J\geq 0$ in this case.

For $r\geq R$ and $\tilde t\leq T_2$, we have $V\in I_{\epsilon/3}$ and $\theta\in [0,\epsilon/2]$, and hence $V+\theta\in I_\epsilon$ and
\[
[-f'(V+\theta)-\beta]\geq \eta-\beta_0=0.
\]
Thus in this case we also have $J\geq 0$.

For $r\in [0, R]$ and $\tilde t\leq T_2$, by the definition of $\tilde \delta$, we have
\[
V_t\geq \tilde\delta\geq \delta_0.
\]
On the other hand,
\begin{equation}
\label{[]}
\big[-f'(V+\theta)-\beta\big]\sigma\geq -M_0\sigma\geq -\delta_0.
\end{equation}
Thus we have $J\geq 0$ in this case too.

For the remaining case  $\tilde t\geq T_2$ and $r\in  I^C(\tilde t)$, by \eqref{V2},
$V_t\geq \delta\geq \delta_0$ and hence, due to \eqref{[]},
  $J\geq 0$.

We have thus proved that $J\geq 0$ for all $r\geq 0$ and $t\geq 0$.
It follows that, for every $T\geq 1$, $\sigma\in (0,\sigma_0]$ and $\beta\in (0,\beta_0]$, 
\[
W_t-W_{rr}-\frac{N-1}{r}W_r\geq f(W) \mbox{ for } r>0,\; t> 0.
\]
Clearly $W_r(0,t)=0$. Thus $\overline W(x,t):=W(|x|, t)$ satisfies \eqref{W-sup}.

The proof for $\underline W$ is analogous and we omit the details.
\end{proof}

In the rest of this section, $u(x,t)$ always stands for a solution of \eqref{nd}  satisfying \eqref{u-p}, with  $u_0\in \mathcal T(f)$, namely
\begin{equation}\label{condition-u0}
u_0\in L^\infty(\R^N),\; \limsup_{|x|\to\infty} u_0(x)<b_*.
\end{equation}
\begin{lem}\label{i}
Let  $V$, $\beta_0$ and $\sigma_0$ be given in Lemma \ref{W-sup-sub}. Then there exist positive constants $T$ and $ T_0$ such that, for all $x\in\R^N$ and $t>T$,
\begin{equation}
\label{V-u-V}
V(|x|, t-T)-\sigma_0\beta_0 e^{-\beta_0(t-T)}\leq u(x,t)\leq V(|x|, t+T_0)+\sigma_0\beta_0 e^{-\beta_0(t-T)}.
\end{equation}
\end{lem}
\begin{proof}

Since $f(u)>0$ for $u<0$ and $f(0)=0$, the ODE solution $\rho_*(t)$ to
\[
\rho_*'=f(\rho_*),\; \rho_*(0)=-\|u_0\|_\infty
\]
satisfies $\lim_{t\to\infty} \rho_*(t)=0$. Due to \eqref{condition-u0}, the comparison principle infers that
\[
u(x,t)\geq \rho_*(t) \mbox{ for } x\in\R^N, t>0.
\]
Therefore there exists $T_1>0$ such that
\[
u(x,t)>-\beta_0\sigma_0/2 \mbox{ for } x\in\R^N, t\geq T_1.
\]
Set 
\[
\underline W(x, t):=V(|x|, t+e^{-\beta_0 t})-\sigma_0 \beta_0 e^{-\beta_0 t}.
\]
Clearly 
\[
\underline W(x,0)=V(|x|,1)-\sigma_0\beta_0,\; \lim_{|x|\to\infty} \underline W(x,0)=-\sigma_0\beta_0.
\]
Therefore we can find $R_1>0$ such that
\[
\underline W(x,0)<-\beta_0\sigma_0/2 \mbox{ for } |x|\geq R_1.
\]
By \eqref{u-p}, there exists $T_2\geq T_1$ such that
\[
u(x,t)\geq p-\beta_0\sigma_0 \mbox{ for } |x|\leq R_1,\; t\geq T_2.
\]
Since $V(|x|, t)<p$ for all $x\in\R^N$ and $t\geq 0$, we thus have 
\[
u(x, T)>\underline W(x,0) \mbox{ for $x\in\R^N$ and $T\geq T_2$}.
\]
By Lemma \ref{W-sup-sub} and the comparison principle we immediately obtain
\[
u(x, T+t)\geq \underline W(x,t) \mbox{ for } x\in\R^N, t>0.
\]
That is,  the first inequality in \eqref{V-u-V} holds for any $T\geq T_2$.

We now set to prove the second inequality in \eqref{V-u-V}. Firstly by comparing $u$ with the ODE solution of
\[
\rho'=f(\rho),\; \rho(0)=\|u_0\|_\infty+p
\]
 we can find $T_3\geq T_2$ such that
\[
u(x,t)<p+\beta_0\sigma_0/2 \mbox{ for } x\in\R^N, t\geq T_3.
\]

We next show that there exist $T_4\geq T_3$ and $R_2>0$ such that
\[
u(x,T_4)<\beta_0\sigma_0 \mbox{ for } |x|\geq R_2.
\]
To this end, we choose a radially symmetric continuous function $\tilde u_0(|x|)$ satisfying \eqref{condition-u0}
and $\tilde u_0(|x|)\geq \max\{u(x, 0), 0\}$, and moreover $\tilde u_0(r)$ is nonincreasing in $r$. Let $\tilde u(|x|, t)$ be the
solution of \eqref{nd} with initial function $\tilde u_0$. Then $\tilde u(r,t)\geq 0$ is nonincreasing in $r$.
Hence 
\[
\rho^*(t):=\lim_{r\to\infty} \tilde u(r,t) \mbox{ exists},
\]
and by a regularity consideration one sees that $\rho^*(t)$ satisfies the ODE
\[
(\rho^*)'=f(\rho^*),\; \rho^*(0)\in[0, b_*).
\]
Since $f(u)<0$ for $u\in (0, b_*)$  we have $\rho^*(t)\to 0$ as $t\to\infty$. Therefore we can find $T_4\geq T_3$ such that 
\[
\rho^*(T_4)<\beta_0\sigma_0/2.
\]
The definition of $\rho^*(t)$ then gives some $R_2>0$ such that 
\[
\tilde u(x,T_4)<\rho^*(T_4)+\beta_0\sigma_0/2<\beta_0\sigma_0 \mbox{ for } |x|\geq R_2.
\]
Since $\tilde u_0\geq u_0$ the comparison principle yields $u(x,t)\leq \tilde u(|x|, t)$.
Therefore
\[
 u(|x|,T_4)\leq \tilde u(|x|, T_4)<\beta_0\sigma_0 \mbox{ for } |x|\geq R_2,
\]
as claimed.

Since $V(|x|, t)\to p$ as $t\to\infty$ locally uniformly for $x\in\R^N$, we can find $T_5>T_4$ such that
\[
V(|x|, t)>p-\beta_0\sigma_0/2 \mbox{ for } |x|\leq R_3, t\geq T_5.
\]
We now define
\[
\overline W(x, t)=V(|x|, t+T_5)+\sigma_0\beta_0 e^{-\beta_0 t}.
\]
Then 
\[
\overline W(x,0)>p+\beta_0\sigma_0/2 \mbox{ for } |x|\leq R_3
\]
and due to $V>0$,
\[
\overline W(x,0)>\beta_0\sigma_0 \mbox{ for all } x\in\R^N.
\]
Thus we have
\[
u(x, T_4)\leq \overline W(x,0).
\]
By Lemma \ref{W-sup-sub} and the comparison principle we deduce
\[
u(x, T_4+t)\leq \overline W(x, t) \mbox{ for } x\in\R^N, t>0.
\]
Hence if we take $T=T_4$ and $T_0=T_5-T_4$, then the second inequality in \eqref{V-u-V} holds.
Since $T_4\geq T_2$, the first inequality in \eqref{V-u-V} also holds with this $T$.
\end{proof}

Proposition \ref{prop2} clearly is a direct consequence of Lemma \ref{i}.

\subsection{Proof of Theorems \ref{thm-levset} and \ref{thm-conv-tw}}

We now derive the properties of the level set of $u$ by making use of Lemma \ref{i}.
\begin{lem}
\label{|x|-bd}
Let $a\in (Q_k, Q_{k-1})$ for some $k\in\{ 1,..., n_0\}$. Then there exist $T_a>0$ and $C_1, C_2\in\R$  such that
\begin{equation}
\label{<|x|<}
c_kt+\eta_k(t)+C_1\leq |x|\leq c_kt+\eta_k(t)+C_2 \mbox{ for $t\geq T_a$ and $x\in\Gamma_a(t)$}.
\end{equation}
\end{lem}
\begin{proof}
We make use of \eqref{V-u-V}. Firstly choose $\epsilon>0$ small so that $[a-\epsilon, a+\epsilon]
\subset (q_{i_k}, q_{i_{k-1}})$. Then choose $T_1>0$ large so that for $t\geq T_1$,
\[
\epsilon_0\sigma_0e^{-\beta_0(t-T)}<\epsilon/2.
\]
Then for any $x^t\in\Gamma_a(t)$ and $t\geq T_1$, \eqref{V-u-V} infers
\[
V(|x^t|, t-T)-\epsilon/2<a<V(|x^t|, t+T_0)+\epsilon/2.
\]
Using this and \eqref{V-limit} we can find $T_2\geq T_1$ such that
\[
\left|V(r,t)-U_k(r-c_kt-\eta_k(t))\right|<\epsilon/2 \mbox{ for } t\geq T_2-T, \ r\in\{|x^{t+T}|, |x^{t-T_0}|\}.
\]
It then follows that for $t\geq T_2$,
\[
U_k(|x^t|-c_k(t-T)-\eta_k(t-T))-\epsilon<a<U_k(|x^t|-c_k(t+T_0)-\eta_k(t+T_0))+\epsilon
\]
We thus obtain, for $t\geq T_2$,
\[
|x^t|-c_k(t-T)-\eta_k(t-T)>a_\epsilon^+,\; |x^t|-c_k(t+T_0)-\eta_k(t+T_0))<a_\epsilon^-,
\]
with $a_\epsilon^-$ and $a_\epsilon^+$ determined by
\[
U_k(a_\epsilon^-)=a-\epsilon,\; U_k(a_\epsilon^+)=a+\epsilon.
\]
Since $\eta_k'(t)\to 0$ as $t\to\infty$, there exists $T_3\geq T_2$ such that, for $ t\geq T_3$,
\[
\eta_k(t+T_0)\leq \eta_k(t)+\epsilon T_0,\; \eta_k(t-T)\geq \eta_k(t)-\epsilon T .
\]
We hence obtain, for $t\geq T_3$,
\[
a_\epsilon^+-c_kT-\epsilon T<|x^t|-c_kt-\eta_k(t)<a_\epsilon^-+c_kT_0+\epsilon T_0.
\]
This clearly implies \eqref{<|x|<} with $T_a=T_3$ and 
\begin{equation}
\label{C12}
C_1:=a_\epsilon^+-c_kT-\epsilon T,\;  C_2:=a_\epsilon^-+c_kT_0+\epsilon T_0.
\end{equation}
\end{proof}

Let us note that the above proof also indicates that, for $t\geq T_a$,
\[
u(x,t)<a \mbox{ for } |x|\geq c_kt+\eta_k(t)+C_2,\; u(x,t)>a \mbox{ for } |x|\leq c_kt+\eta_k(t)+C_1.
\]
Therefore, for any $\nu\in \mathbb{S}^{N-1}$ and $t\geq T_a$, there exists a  $\xi\in (c_kt+\eta_k(t)+C_1, c_kt+\eta_k(t)+C_2)$ such that
$\xi\nu\in \Gamma_a(t)$. We denote the minimal such $\xi$ by $\xi_a(t, \nu)$. Then
\begin{equation}
\label{xi_a}
\xi_a(t,\nu)\in  (c_kt+\eta_k(t)+C_1, c_kt+\eta_k(t)+C_2),\; \xi_a(t,\nu)\nu\in\Gamma_a(t),\; \forall t\geq T_a.
\end{equation}

The proof of the following result is based on \eqref{V-u-V} and a useful result of Berestycki and Hamel \cite[Theorem 3.1]{BH}.
\begin{lem}\label{iii}
Let $\xi_a(t,\nu)$ be as above. By enlarging $T_a$ if necessary, the following conclusions hold for $t\geq T_a${\rm :}
\begin{itemize}
\item[(i)] $\xi\nu\in \Gamma_a(t)$ implies $\xi=\xi_a(t,\nu)$, and $\xi_a(t,\nu)$ is a $C^1$ function on $(T_a, \infty)\times \mathbb{S}^{N-1}$.
\item[(ii)] For any bounded set $\Omega\subset \R^N$, 
\begin{equation*}
\label{u-U_k}
\lim_{t\to\infty}u(x+\xi_a(t,\nu)\nu, t)= U_k(x\cdot \nu+\alpha_k^a) 
\end{equation*}
  uniformly for $x\in \Omega$ and $\nu\in\mathbb{S}^{N-1}$, where $\alpha_k^a$ is given by $U_k(\alpha_k^a)=a$.
\item[(iii)] Let  $T$ and $T_0$ be  given in \eqref{V-u-V}. Then
\[
\lim_{t\to\infty}\frac{\xi_a(t,\nu)}{t}=c_k \mbox{ uniformly for } \nu\in \mathbb{S}^{N-1}, \mbox{ and }
\]
\[
\limsup_{t\to\infty} \left[\max_{\nu\in\mathbb{S}^{N-1}}\xi_a(t,\nu)-\min_{\nu\in\mathbb{S}^{N-1}}\xi_a(t,\nu)\right]\leq (T+T_0)c_k.
\]
\end{itemize}
\end{lem}
\begin{proof}
Let us note that once (i) is proved, then the conclusions in (iii) follow directly from \eqref{<|x|<} and  the following abservations: 
\[
C_2-C_1=(c_k+\epsilon)(T+T_0)+a^-_\epsilon-a^+_\epsilon\; \mbox{ [by \eqref{C12}]  and}
\]
\[
\lim_{\epsilon\to 0} a^-_\epsilon=\lim_{\epsilon\to 0} a^+_\epsilon=\alpha_k^a.
\]
Therefore, to complete the proof, we only need to prove (i) and (ii).

Let $\{t_n\}$ be an arbitrary sequence converging to $\infty$. Without loss of generality we may assume that $t_n>T_a$ for all $n\geq 1$. Fix $\nu\in\mathbb{S}^{N-1}$ and let $\xi_n>0$ be chosen such that $x_n:=\xi_n \nu\in \Gamma_a(t_n)$.
By \eqref{xi_a},
\begin{equation}
\label{xi_n}
\xi_n-c_k t_n-\eta_k(t_n) \in (C_1, C_2)\;\;\forall n\geq 1.
\end{equation}

Define
\[
u_n(x,t):=u(x+x_n, t+t_n).
\]
Since $u_n$ has an $L^\infty$ bound which is independent of $n$, by a standard regularity consideration and diagonal process, subject to passing to a subsequence we
may assume that 
\[
\lim_{n\to\infty}u_n(x,t)= \tilde u(x,t) \mbox{ in } C^{2,1}_{loc}(\R^N\times \R),
\]
and $\tilde u$ satisfies
\[
\tilde u_t-\Delta \tilde u=f(\tilde u)   \; \mbox{ for } x\in\R^N,\; t\in\R.
\]

By \eqref{V-u-V} we obtain
\begin{equation}
\label{V<u}
V(|x+x_n|, t+t_n-T)-\sigma_0\beta_0 e^{-\beta_0(t+t_n-T)}\leq u_n(x,t)
\mbox{ and }
\end{equation}
\begin{equation}
\label{u<V}
u_n(x,t)\leq V(|x+x_n|, t+t_n+T_0)+\sigma_0\beta_0 e^{-\beta_0(t+t_n-T)}.
\end{equation}

We calculate, for all large $n$,
\[
|x+x_n|-c_k(t+t_n)-\eta_k(t+t_n)=\xi_n-c_kt_n-\eta_k(t_n)+J,
\]
with
\[
J:=|x+x_n|-|x_n|-c_kt-\eta_k(t+t_n)-\eta_k(t_n)=x\cdot\nu-c_k t+o_n(1),
\]
where $o_n(1)\to 0$ as $n\to\infty$.
In view of \eqref{xi_n}, by passing to a subsequence we may assume that
\[
\xi_n-c_k t_n-\eta_k(t_n)\to \alpha\in [C_1, C_2] \mbox{ as } n\to\infty.
\]
These imply, by \eqref{V-limit}, 
\[
\lim_{n\to\infty} V(|x+x_n|, t+t_n-T)= U_k(x\cdot \nu-c_k(t-T)+\alpha) \mbox{ and }
\]
\[
\lim_{n\to\infty} V(|x+x_n|, t+t_n+T_0)= U_k(x\cdot \nu-c_k(t+T_0)+\alpha).
\]
Letting $n\to\infty$ in \eqref{V<u} and \eqref{u<V} we thus obtain
\[
U_k(x\cdot \nu-c_kt+c_kT+\alpha)\leq \tilde u(x,t)\leq U_k(x\cdot \nu-c_kt-c_kT_0+\alpha) \mbox{ for } x\in\R^N,\; t\in\R.
\]
We may now apply Theorem 3.1 of \cite{BH} to conclude that there exists $\tilde \alpha \in [\alpha+c_k T, \alpha-c_kT_0]$ such that
\[
\tilde u(x,t)\equiv U_k(x\cdot \nu-c_kt+\tilde \alpha) \mbox{ for } x\in\R^N, t\in\R.
\]
Since $u_n(0,0)=a$, we have $\tilde u(0,0)=a$ and hence $U_k(\tilde\alpha)=a$. It follows that $\tilde \alpha=\alpha_k^a$.
Thus $\tilde u(x,t)$ is uniquely determined, and we may conclude that for $s>T_a$ and any $x^\nu_s:=\xi_s\nu\in \Gamma_a(s)$,
\begin{equation}
\label{u-limit}
\lim_{s\to\infty} u(x+x^\nu_s, t+s)=U_k(x\cdot\nu-c_k t+\alpha_k^a) \mbox{ in } C_{loc}^{2,1}(\R^N\times \R).
\end{equation}
The arguments leading to \eqref{u-limit} show that this limit is uniform in $\nu\in\mathbb{S}^{N-1}$.
In particular,
\[
\nabla_x u(x^\nu_s, s)\to U_k'(\alpha_k^a)\nu \mbox{ as } s\to\infty \mbox{ uniformly in } \nu\in\mathbb{S}^{N-1}.
\]
Therefore by enlarging $T_a$ we may assume that $\partial_\nu u(x_s^\nu, s)<U_k'(\alpha_k^a)/2<0$ for $s>T_a$. By the implicit function theorem we know that in a small neighborhood of $(x_s^\nu, s)$ in $\R^N\times \R$, the solutions of $u(x,s)=a$ may be expressed as $(\xi(s,\nu)\nu, s)$
with $\xi(s,\nu)$ a $C^1$ function of its arguments. 

The above analysis also shows that whenever $s>T_a$ and $u(\xi\nu, s)=a$, we have $\partial_\nu u(\xi\nu, s)<0$.
Hence for each $\nu\in\mathbb{S}^{N-1}$ and $s>T_a$, we can have no more than one $\xi>0$ such that $u(\xi\nu, s)=a$. Thus  $t>T_a$ and $\xi\nu\in \Gamma_a(t)$ imply $\xi=\xi(t,\nu)$. That is
\[
\Gamma_a(t)=\{\xi(t,\nu)\nu: \nu\in\mathbb{S}^{N-1}\} \;\; \forall t>T_a.
\]
We have thus proved the conclusions in part (i) of the lemma. Part (ii) clearly follows from \eqref{u-limit} by taking $t=0$ and noticing that $x^\nu_s=\xi(s,\nu)\nu$.
\end{proof}

It is clear that Theorems \ref{thm-levset} and \ref{thm-conv-tw} follow directly from Lemma  \ref{iii}.

\subsection{Proof of Theorems \ref{thm-conv-ter} and \ref{thm-log-shift}} 
\begin{proof}[{\bf Proof of Theorem \ref{thm-conv-ter}}] Keeping the notations in Lemma \ref{iii} we define, for $k=1,..., n_0$ and $a=a_k:=(Q_k+Q_{k-1})/2$,
\[
\zeta_k(t):=\frac1{|\mathbb S^{N-1}|}\int_{\mathbb S^{N-1} }\xi_{a}(t,\nu)d\nu-c_kt,
\]
\[\tilde\zeta_k(t, \nu):=\xi_a(t,\nu)-c_kt-\zeta_k(t)-\alpha_k^a.
\]
By conclusion (iii) in Lemma \ref{iii}, we find that
\[
\lim_{t\to\infty}\frac{\zeta_k(t)}{t}=0\; \mbox{ and } \tilde\zeta_k\in C([T_{a},\infty)\times\mathbb S^{N-1})\cap L^\infty([T_{a},\infty)\times\mathbb S^{N-1}).
\]
By conclusion (ii) of Lemma \ref{iii}, we have
\[
\lim_{t\to\infty} u(\big[s+\xi_a(t, \nu)\big]\nu, t)=U_k(s+\alpha_k^a)
\]
uniformly for $s\in[-R, R]$ for any $R>0$, which is equivalent to
\[
\lim_{t\to\infty}\left[ u(y, t)-U_k(|y|-\xi_a(t,\frac{y}{|y|})+\alpha_k^a)\right]=0
\]
uniformly for $|y|-\xi_a(t,\frac{y}{|y|})\in [-R, R]$ for any $R>0$.
Since
\[
\xi_a(t,\nu)-\alpha_k^a=c_kt+\zeta_k(t)+\tilde\zeta_k(t,\nu) \mbox{ and } \tilde \zeta_k\in L^\infty,
\]
we may rewrite the above conclusion as
\begin{equation}\label{level-k}
\lim_{t\to\infty}\left[u(y, t)-U_k(|y|-c_kt-\zeta_k(t)-\tilde\zeta_k(t, \frac{y}{|y|}))\right]=0
\end{equation}
uniformly for $|y|-c_kt-\zeta_k(t)\in[-R, R]$ for any $R>0$. 

From  \eqref{level-k}, \eqref{V-u-V} and \eqref{V-limit} we easily see that
$\zeta_k-\eta_k\in L^\infty$. Therefore for any given $R>0$ and for all large $t$, say $t\geq T_0$, the intervals
\[
I_k(t):=[c_kt+\zeta_k(t)-R, c_kt+\zeta_k(t)+R], \; k=1,2,..., n_0
\]
are non-overlapping, with the gap between the neighboring ones converging to infinity as $t\to\infty$, namely
\[
\lim_{t\to\infty}\min_{2\leq k\leq n_0} \big[\min I_k(t)-\max I_{k-1}(t)\big]=\infty.
\]

Set
\[
C_0:=\max_{1\leq k\leq n_0}\big(\|\tilde\zeta_k\|_\infty+\|\zeta_k-\eta_k\|_\infty\big).
\]
Given any $\epsilon>0$, by fixing $R>0$ sufficiently large we can guarantee that, for every $k\in \{1,..., n_0\}$,
\[
0<U_k(s)-Q_{k}<\epsilon \mbox{ for } s\geq R-C_0,\; 0<Q_{k-1}-U_k(s)<\epsilon \mbox{ for } s\leq -R+C_0.
\]
It then follows that, for every $k\in\{1,..., n_0\}$ and $t\geq T_0$,
\begin{align*}
&0<U_k(|y|-c_kt-\zeta_k(t)-\tilde\zeta_k(t, \frac{y}{|y|}))-Q_k<\epsilon&& \mbox{ for } |y|\geq\max T_k(t),
\\
&0<Q_{k-1}-U_k(|y|-c_kt-\zeta_k(t)-\tilde\zeta_k(t, \frac{y}{|y|}))<\epsilon&& \mbox{ for } |y|\leq\min T_k(t),
\\
&0<U_k(|y|-c_k-\eta_k(t))-Q_k<\epsilon &&\mbox{ for } |y|\geq\max T_k(t),
\\
&0<Q_{k-1}-U_k(|y|-c_k-\eta_k(t))<\epsilon&& \mbox{ for } |y|\leq\min T_k(t).
\end{align*}

These inequalities, together with   \eqref{V-u-V}, \eqref{V-limit} and \eqref{level-k}, imply that there exists $T\geq T_0$ such that for $t\geq T$,
\begin{align*}
&|u(y,t)-Q_0|<2\epsilon && \mbox{ for } |y|\leq\min I_1(t),
\\
&|u(y,t)-Q_k|<2\epsilon&& \mbox{ for } |y|\in[\max I_k(t),\min I_{k+1}(t)],\; k=1,..., n_0-1,
\\
&0<u(y,t)<2\epsilon &&\mbox{ for } |y|\geq \max I_{n_0}(t),
\end{align*}
and
\[
|u(y,t)-U_k(|y|-c_kt-\zeta_k(t)-\tilde\zeta_k(t, \frac{y}{|y|}))|<2\epsilon \mbox{ for }|y|\in I_k(t),\; k=1,..., n_0.
\]

Define
\[
J(y,t):=\sum_{k=1}^{n_0}\left[U_k\big(|y|-c_kt-\zeta_k(t)-\tilde\zeta_k(t, \frac{y}{|y|})\big)-Q_k\right].
\]
By our choice of $R$ and $T$, we have from the earlier inequalities that for $t\geq T$,
\begin{align*}
&|J(y,t)-Q_0|\leq n_0\epsilon && \mbox{ for } |y|\leq \min I_1(t),\\
&|J(y,t)-Q_k|\leq n_0\epsilon && \mbox{ for } |y|\in [\max I_k(t), \min I_{k+1}(t)],\; k=1,..., n_0-1,\\
&0<J(y,t)<n_0\epsilon && \mbox{ for } |y|\geq\max I_{n_0}(t),
\end{align*}
and
\[
|J(y,t)-U_k\big(|y|-c_kt-\zeta_k(t)-\tilde\zeta_k(t, \frac{y}{|y|})\big)|<n_0\epsilon \mbox{ for } |y|\in I_k(t),\; k=1,..., n_0.
\]
We thus obtain
\[
|u(y,t)-J(y,t)|<(n_0+2)\epsilon \mbox{ for  all } |y|\geq 0,\; t\geq T.
\]
That is, 
\begin{equation}\label{u-U}
\lim_{t\to\infty}\left|u(x,t)-\sum_{k=1}^{n_0}\left[U_k\big(|x|-c_kt-\zeta_k(t)-\tilde\zeta_k(t, \frac{x}{|x|})\big)-Q_k\right]\right|=0
\end{equation}
uniformly for $x\in\mathbb R^N\setminus\{0\}$.
Define
\[
\tilde\eta_k(t, \nu):=\tilde\zeta_k(t,\nu)+\zeta_k(t)-\eta_k(t).
\]
Then $\tilde\eta_k\in C\cap L^\infty$ and \eqref{u-limit-w} follows directly from \eqref{u-U}. \end{proof}

\begin{proof}[{\bf Proof of Theorem \ref{thm-log-shift}}]
 If {\bf (f4)} holds, then by Theorem \ref{eta_k}, there exist $C>0$ and $T_1\geq T$ such that
\[
\hat\eta_k(t):= \eta_k(t)+\frac{N-1}{c_k}\log t\in[-C, C]  \mbox{ for all  } t\geq T_1,\; k\in\{1,..., n_0\}.
\]
Define
\[
\tilde\eta_k(t, \nu):=\tilde\zeta_k(t,\nu)+\zeta_k(t)+\frac{N-1}{c_k}\log t;
\]
then $\tilde\eta_k\in C([T_1,\infty)\times \mathbb S^{N-1})\cap L^\infty([T_1,\infty)\times \mathbb S^{N-1})$ and we see from \eqref{u-U} that \eqref{u-limit-1} holds for $u$. Using the strengthened \eqref{u-limit-w}, namely \eqref{u-limit-1}, we immediately obtain \eqref{xi_a-sharp}.
\end{proof}

\begin{rmk}\label{V-limit-precise}
When {\bf (f4)} holds, applying Theorem \ref{thm-log-shift} to any radial  solution $u(r,t)$ we see that \eqref{u-limit-1} holds with $\tilde\eta_k=\tilde\eta_k(t)$ which belongs to $L^\infty(\R_+)$. Using this fact, we can further argue as in Section 6 of
\cite{DQZ} to show that
 $\tilde \eta_k(t)\to R_k$ as $t\to\infty$, for some $R_k\in\R$. Hence in this case, $u(r,t)$ satisfies
 \eqref{u-limit-1} with $\tilde\eta_k$ replaced by $R_k$.

\end{rmk}

\bigskip

\bibliographystyle{amsplain}

\end{document}